\documentclass[12pt,leqno]{amsart}
\usepackage[latin1]{inputenc}
\usepackage{color}
\usepackage{amssymb,amsmath,amsthm,amscd,amsfonts}
\usepackage{enumerate}
\usepackage[all]{xy}
\usepackage{mathtools}

\usepackage{hyperref,cleveref,graphics,mathrsfs}

\numberwithin{equation}{section}
\def\hangbox to #1 #2{\vskip3pt\hangindent #1\noindent \hbox to #1{#2}$\!\!$}

\usepackage{hyperref,cleveref,amssymb,mathtools}

\theoremstyle{plain}
\newtheorem{theorem}{Theorem}[section]
\newtheorem{proposition}[theorem]{Proposition}
\newtheorem{corollary}[theorem]{Corollary}
\newtheorem{lemma}[theorem]{Lemma}

\theoremstyle{definition}
\newtheorem{remark}[theorem]{Remark}

\DeclareSymbolFont{bbold}{U}{bbold}{m}{n}
\DeclareSymbolFontAlphabet{\mathbbold}{bbold}

\def\sfrac#1#2{\kern.1em\raise.5ex\hbox{$#1$}
        \kern-.1em/\kern-.05em\lower.25ex\hbox{$#2$}}

\newcommand{\fw}{\text{\fw}}

\title[The optimal Sobolev threshold for evolution equations with rough nonlinearities]{On the optimal Sobolev threshold for evolution equations with rough nonlinearities}
\author{Ben Pineau}
\address{Courant Institute for Mathematical Sciences\\
New York University
} \email{brp305@nyu.edu}
\author{Mitchell~A.\ Taylor}
\address{Department of Mathematics\\
ETH Z\"urich, Ramistrasse 101, 8092 Z\"urich, Switzerland.
} \email{mitchell.taylor@math.ethz.ch}

\date{\today}
\subjclass[2020]{35A01, 35B65 (primary);  35K58, 35Q55  (secondary)} 

\keywords{Nonlinear Schr\"odinger equation; nonlinear heat equation; well-posedness; rough nonlinearity.}

\begin{document}

\allowdisplaybreaks

 \begin{abstract}
In this article we are concerned with evolution equations of the form
\begin{equation*}
\partial_tu-A(D)u=F(u,\overline{u},\nabla u, \nabla \overline{u})
\end{equation*}
where $A(D)$ is a Fourier multiplier of either dispersive or parabolic type and the nonlinear term $F$ is of limited regularity. Our objective is to develop a robust set of principles which can be used in many cases to predict the \emph{highest} Sobolev exponent $s=s(q,d)$ for which the above evolution is well-posed in $W_x^{s,q}(\mathbb{R}^d)$ (necessarily restricting to $q=2$ for dispersive problems). We will confirm the validity of these principles for two of the most important model problems; namely, the nonlinear Schr\"odinger and heat equations. More precisely, we will prove that the nonlinear heat equation
\begin{equation*}
\partial_tu-\Delta u=\pm |u|^{p-1}u, \hspace{5mm} p>1,
\end{equation*}
is well-posed in $W_x^{s,q}(\mathbb{R}^d)$ when $\max\{0,s_c\}<s<2+p+\frac{1}{q}$ and is \emph{strongly ill-posed} when $s\geq \max\{s_c,2+p+\frac{1}{q}\}$ and $p-1\not\in 2\mathbb{N}$  in the sense of non-existence of solutions even for smooth, small and compactly supported data. When $q=2$, we establish the same ill-posedness result for the nonlinear Schr\"odinger equation and the corresponding well-posedness result when $p\geq \frac{3}{2}$. Identifying the optimal Sobolev threshold for even a single non-algebraic $p>1$ was a rather longstanding open problem in the literature. As an immediate corollary of the fact that our ill-posedness threshold is dimension independent, we may conclude by taking $d\gg p$ that there are nonlinear Schr\"odinger equations which are ill-posed in \emph{every} Sobolev space $H_x^s(\mathbb{R}^d)$.
\end{abstract} 
\maketitle
\setcounter{tocdepth}{1}
\tableofcontents
\section{Introduction}\label{Intro}
The objective of this article is to develop a robust set of principles which can be used to predict the maximal possible Sobolev regularity of solutions to a large class of evolution equations with rough nonlinearities. We will then rigorously justify these heuristics for two of the most important families of such equations, namely, the nonlinear Schr\"odinger equations
\begin{equation}\label{NLS}\tag{NLS}
\begin{cases}
&i\partial_tu-\Delta u=\pm |u|^{p-1}u,
\\
&u(0)=u_0,
\end{cases}
\end{equation}
and the nonlinear heat equations
\begin{equation}\label{NLH}\tag{NLH}
\begin{cases}
&\partial_tu-\Delta u=\pm |u|^{p-1}u,
\\
&u(0)=u_0.
\end{cases}
\end{equation}
The question of characterizing the Sobolev exponents $s$ for which the above evolutions are well-posed in $H^s(\mathbb{R}^d)$ and $W^{s,q}(\mathbb{R}^d)$, respectively, is a fundamental problem, which has received significant attention over several decades. Focusing our attention first on \eqref{NLS}, we recall that due to the advent of Strichartz estimates, the following general result is by now classical.
\begin{theorem}[\cite{MR1667895,MR1691575,MR1946761}]\label{Classical LWP}
    Let $p>1$ and let $s\geq \max\{0,s_c\}$ where $s_c:=\frac{d}{2}-\frac{2}{p-1}$ is the scaling critical Sobolev index. If $p-1\not\in 2\mathbb{N}$, assume further that $\lfloor s\rfloor<p-1$. Then \eqref{NLS} is locally well-posed in $H^s(\mathbb{R}^d)$.
\end{theorem}
When $p-1\in 2\mathbb{N}$ and $s_c\geq 0$, the above theorem characterizes all Sobolev exponents $s$ for which \eqref{NLS} is well-posed in $H^s(\mathbb{R}^d)$. On the other hand, the lower bound in \Cref{Classical LWP} can be improved for certain integrable, mass subcritical ($s_c<0$) models. See, for instance, the recent sharp result \cite{harrop2020sharp} for the 1D cubic \eqref{NLS}, which builds on ideas from the breakthrough paper \cite{MR3990604}. 
\medskip

Unsurprisingly, in the case when $p-1\not\in 2\mathbb{N}$, the condition $\lfloor s\rfloor<p-1$ in \Cref{Classical LWP} is not optimal, and it has been a longstanding open question (for even a single example of $p$) to identify the sharp upper bound on $s$ for which \Cref{Classical LWP} remains true. Perhaps the first progress on this question was made in the 1980s by Kato \cite{MR877998, MR1037322} and Tsutsumi \cite{MR0913674}, who were able to establish $H^2(\mathbb{R}^d)$ well-posedness for certain \eqref{NLS} equations with $p<2$. Roughly speaking,  their strategy was to differentiate \eqref{NLS} in time,  obtain an $L^2(\mathbb{R}^d)$ a priori estimate for $\partial_tu$, and then use the equation to convert this into an estimate for $\Delta u$ (and thus, the $H^2(\mathbb{R}^d)$ norm of $u$). The motivation for this approach is that it only requires one to differentiate the nonlinearity a single time, as opposed to twice when carrying out the analogous estimate using only spatial derivatives. 



\medskip

It is natural to ask whether one can extend the  strategy of Kato and Tsutsumi  to higher order fractional Sobolev spaces $H^s(\mathbb{R}^d)$. At a high level, one might try to control $D_x^su$ (where $D_x^s$ is the multiplier $\xi\mapsto |\xi|^s$) by $D_t^{\frac{s}{2}}u$ (neglecting for the time being that this expression may not a priori make sense without suitable time truncations) and then carry out an energy estimate for the latter term. This strategy comes with a host of delicate issues. For instance, it will generally only be possible to implement when the spacetime Fourier transform of $u$ is localized near the characteristic paraboloid $\tau=-|\xi|^2$, where the space and time frequencies of $u$ are comparable.  Moreover, this method is particularly difficult to implement from a technical standpoint for fractional $s$, as it requires one to carefully truncate the nonlinearity in time. Since such a truncation will often have to be performed at small time scales,  it will cause the nonlinearity to shift to higher time frequencies, making it difficult to use the length of the time interval as a smallness parameter. This issue becomes especially non-trivial to deal with when $s>2$ (in which case, more than one time derivative would be involved in the analysis), as we will later see. We remark that for mass subcritical problems it is possible to circumvent this issue by rescaling the (inhomogeneous) $H^s(\mathbb{R}^d)$ data to be small so that one can work with unit time scales, but in supercritical regimes, this is not feasible. Notwithstanding the above caveats, we remark that implementing this strategy might lead one to believe that a natural upper bound on $s$ to ensure local well-posedness of \eqref{NLS} would be $s<2p$. However, this is not entirely accurate. 
\medskip


Due to the  combined efforts  of several authors (see~e.g.~\cite{MR2765515,MR3546788, MR3017270,MR4388268, MR3424613, MR3953030,MR1472820,MR4581790, uchizono2012well,MR3917711}) the  consensus became that the \eqref{NLS} equations with $p-1\not\in 2\mathbb{N}$ should be well-posed when
\begin{equation}\label{well-posedness known}
\max\{0,s_c\}\leq s <\begin{cases}
    2p,& \text{if } p\in (1,2),\\
    p+2,& \text{if} \ p\in (2,\infty).        
\end{cases}
\end{equation}
Morally speaking, the upper bound $s<2p$ in \eqref{well-posedness known} comes from the H\"older regularity of $|u|^{p-1}u$ and a proper execution of the ideas of Kato and Tsutsumi in fractional Sobolev spaces. The restriction $s<p+2$ essentially comes from the fact that away from the characteristic set $\tau=-|\xi|^2$ (where one can no longer use time derivatives to measure the spatial regularity of a solution) the operator $(i\partial_t-\Delta)$ can be viewed as elliptic of order $2$ in the spatial variable,  which leads one to estimate (roughly speaking) $D_x^{s-2}(|u|^{p-1}u)$ in $L_t^{\infty}L_x^2$. In view of the H\"older regularity of $|u|^{p-1}u$, this naturally leads to the restriction $s<p+2$ in the above conjecture. 

\begin{remark} One can make a similar conjecture for the nonlinear heat equation \eqref{NLH}. The situation turns out to be simpler in this case than $\eqref{NLS}$, as the characteristic hypersurface $i\tau=|\xi|^2$ for the heat operator $(\partial_t-\Delta)$ consists of the single point $(\tau,\xi)=(0,0)$ and therefore only the restriction $s<p+2$ is relevant. 
\end{remark}
Complementing the above well-posedness conjecture, Cazenave, Dickstein and Weissler \cite{cazenave2017non}   have shown that the \eqref{NLS} equations are \emph{ill-posed} in $H^s(\mathbb{R}^d)$ when $1<p<2$ and $s>p+2+\frac{d}{2}$, with a similar result for the \eqref{NLH} equations in $W^{s,q}(\mathbb{R}^d)$ for $1< q<\infty$. To our knowledge, \cite{cazenave2017non}  is the only article that proves high-regularity ill-posedness results for \eqref{NLS} and \eqref{NLH}. However, as we will soon see, both the ill-posedness and expected well-posedness thresholds for \eqref{NLS} and \eqref{NLH} can be considerably improved.
\medskip

 

The main objective of the current article is to provide a general method that can be used to correctly predict the sharp well-posedness (and ill-posedness) thresholds for a large class of nonlinear evolution equations. To validate these predictions, we give a mostly definitive resolution to the high regularity well-posedness problems for \eqref{NLS} and \eqref{NLH}, by proving the following theorems.
\begin{theorem}\label{M1}
    Let $p\geq \frac{3}{2}$. Then \eqref{NLS} is locally well-posed in $H^s(\mathbb{R}^d)$ when $\max\{0,s_c\}< s<p+\frac{5}{2}$.
\end{theorem}
\begin{remark}
    For now, we have chosen to exclude the case  $p\in (1,\frac{3}{2})$ from the statement of \Cref{M1}; we will state our well-posedness theorem for general $p>1$ in \Cref{M1 OOTP} which also considerably improves upon all previous well-posedness results in this range. 
\end{remark}
Remarkably, we can prove that the above theorem is sharp by complementing it with a corresponding ill-posedness result!
\begin{theorem}\label{M2}
   Let $p>1$ with $p-1\not\in 2\mathbb{N}$. There is an initial data $u_0\in C_c^\infty(\mathbb{R}^d)$ of arbitrarily small norm such that for any $T>0$ there is no  corresponding solution $u\in C([0,T];H^{s}(\mathbb{R}^d))$ for any $s\geq \max\{s_c,p+\frac{5}{2}\}$.
\end{theorem}
For \eqref{NLH}, our main result is as follows and is sharp for the entire range $p>1$.
\begin{theorem}\label{M1H}
    Let $p>1$ with $p-1\not\in 2\mathbb{N}$ and $1<q<\infty$. Let $s_c=\frac{d}{q}-\frac{2}{p-1}$ be the scaling critical threshold. Then \eqref{NLH} is locally well-posed in $W^{s,q}(\mathbb{R}^d)$ when $\max\{0,s_c\}<s<p+2+\frac{1}{q}$. Moreover, there is an initial data $u_0\in C_c^\infty(\mathbb{R}^d)$ of arbitrarily small norm such that for any $T>0$ there is no  corresponding solution $u\in C([0,T];W^{s,q}(\mathbb{R}^d))$ for any $s\geq\max\{s_c,p+2+\frac{1}{q}\}$.
\end{theorem}
\begin{remark}\label{Rem3}
    The methods that we will develop in this paper to prove sharp high regularity well-posedness for \eqref{NLS} and \eqref{NLH} are rather general. In particular, they should apply not only to \eqref{NLS} and \eqref{NLH} but also to nonlinear evolution equations where the principal linear operator has order other than $2$ as well as to certain classes of evolution equations with derivatives in the nonlinearity. We will postpone a precise discussion of these general models to \Cref{Broader applications}.
\end{remark}
\begin{remark}
In our ill-posedness results, we assume that $s$ is not in the supercritical range (i.e., we impose the restriction that $s\geq s_c$). This is natural and is mostly done for technical convenience. For instance, it will ensure (among other things) that the nonlinearity is a well-defined space-time distribution. Moreover, it will allow us to justify certain uniqueness properties of solutions which, in turn, will allow us to conclude non-existence of solutions rather than other standard forms of ill-posedness such as non-uniqueness or norm inflation. One should obviously still expect a similar kind of ill-posedness mechanism to be present in supercritical regimes, but the theorems would need to be phrased more carefully and we do not pursue this optimization here.
\end{remark}

As a small sample of the applications of Theorems \ref{M1}, \ref{M2}, and \ref{M1H}, we can answer several outstanding open questions from the literature. For example, in \cite{MR3611666,MR3968019} it was asked whether there are Sobolev spaces $H^s(\mathbb{R}^{12})$ for which \eqref{NLS} with $p=2$ is well-posed. Notice that for this equation we have  $s_c=4$, which is borderline for \eqref{well-posedness known}, but well within the scope of  \Cref{M1}. On the other extreme, it was not known whether \emph{every} \eqref{NLS} equation admits a well-posed evolution in \emph{some} Sobolev space $H^s(\mathbb{R}^d)$ with $s\geq 0$. Since $s_c$ increases with dimension but the high regularity threshold of $p+\frac{5}{2}$ in \Cref{M2} does not, we may immediately rule out well-posedness for any $s\geq s_c$. On the other hand, when $s<s_c$, ill-posedness is known (see e.g. \cite{MR3546788}).  Hence, we may answer this question in the negative.
\begin{corollary}\label{cor non}
    There are \eqref{NLS} equations which are ill-posed in every Sobolev space $H^s(\mathbb{R}^d)$.
\end{corollary}
As a more general comment, issues regarding  roughness of the nonlinearity can be found throughout  the literature on nonlinear evolution equations.  See, e.g.,~the  books \cite[p.~7, 18, 27, 115]{MR1691575},  \cite[Sections 4.8-4.10]{MR2002047}, \cite[p.~54]{sulem2007nonlinear}, \cite[Chapter 3]{tao2006nonlinear} or the papers \cite{MR2018661,MR1961447,MR3513584,MR1941262,MR2154347}. Notably, such difficulties are not only restricted to the local theory --  they appear in the analysis of blowup solutions \cite{MR1655515}, in  the scattering theory \cite{MR1655835,MR2753625}, and make the question of extending a global well-posedness result to higher regularity norms non-trivial. Although we will not pursue such questions here, we expect that the methods developed in this paper will help clarify several of these issues. We refer the reader to \Cref{Broader applications} for a further discussion on this.

\subsection{Outline of the proof}\label{OOTP}
We now outline the basic strategy for proving Theorems \ref{M1}, \ref{M2} and \ref{M1H} as well as the general principles that make the scheme of proof applicable to other evolution equations. 
\medskip

Let us begin by fixing a Schwartz function $u\in\mathcal{S}(\mathbb{R}^d)$. In order to understand the obstructions to high regularity well-posedness for \eqref{NLS} and \eqref{NLH} we must first understand the precise regularity of the nonlinear term $|u|^{p-1}u$. To the best of our knowledge, all previous approaches to this problem estimate $|u|^{p-1}u$ by using some variant of the standard fractional chain rule. That is, they use the  estimate:
\begin{equation}\label{fractionalchainintro}
    \| F(u)\|_{W^{s,q}(\mathbb{R}^d)}\lesssim_{q,\|u\|_{L^\infty(\mathbb{R}^d)}}\|u\|_{W^{s,q}(\mathbb{R}^d)}, \hspace{5mm} F(0)=0,\hspace{5mm} 1<q<\infty,
\end{equation}
which holds when $F$ is H\"older continuous of order larger than $s$. However, somewhat surprisingly, the H\"older regularity of the function $z\mapsto |z|^{p-1}z$ is not what obstructs high regularity well-posedness for \eqref{NLS} and \eqref{NLH}. In fact, a version of the fractional chain rule \eqref{fractionalchainintro} actually holds for a certain range of $s>p$.
 \medskip
 
 To get a better idea of the true regularity of $|u|^{p-1}u$, we take our cue from the classical papers of Bourdaud-Meyer \cite{MR1111186} and Oswald \cite{MR1173747}, which do not seem to be particularly well-known in the dispersive community. These papers study the boundedness of the absolute value function $u\mapsto |u|$ in the scale of Besov spaces $B^s_{q,r}(\mathbb{R}^d;\mathbb{R})$ and (independently and by two different methods) prove the following theorem.
 \begin{theorem}\label{Abs value reg}
     Let $1\leq q,r\leq \infty$ and $0<s<1+\frac{1}{q}$. There exists a constant $C>0$ such that for all $u\in B^s_{q,r}(\mathbb{R}^d;\mathbb{R})$ we have 
     \begin{equation}\label{abs val reg ine}
         \| |u|\|_{B^s_{q,r}(\mathbb{R}^d)}\leq C\|u\|_{B^s_{q,r}(\mathbb{R}^d)}.
     \end{equation}
     Moreover, for $s\geq 1+\frac{1}{q}$ there exist $u\in B^s_{q,r}(\mathbb{R}^d)$ such that $|u|\not\in B^s_{q,r}(\mathbb{R}^d)$.
 \end{theorem}
 To understand the numerology of the above theorem, we observe that if $\chi$ is a cutoff function equal to $1$ in a neighborhood of the origin then the function $u(x)=\chi(x) x_1$ is Schwartz and (for $q\neq\infty$) one may compute that $|u|\in B_{q,r}^s(\mathbb{R}^d)$ if and only if $s<1+\frac{1}{q}$. The excess regularity when compared with \eqref{fractionalchainintro} is heuristically tied to the fact that one may ``integrate" over the local singularity in $D_x^{s}(\chi|x_1|)$ in $L^q$ when $s<1+\frac{1}{q}$. Morally speaking, \Cref{Abs value reg} states that for a generic function $u\in B_{q,r}^s(\mathbb{R}^d)$, the local singularities in $|u|$ coming from the simple zeros of $u$ can be harmlessly ``integrated" over if $s<1+\frac{1}{q}$. This extra regularity will be crucial for our well-posedness proofs, as it will allow us to place another $\frac{1}{q}$ derivatives on the nonlinear term. However, we will not be able to use \Cref{Abs value reg} directly for two reasons. First, it does not apply to the higher-order expression $|u|^{p-1}u$ and, in our setting, $|u|^{p-1}u$ can be complex-valued, which poses considerable challenges compared to the real-valued case.
\medskip

Nevertheless, in light of \Cref{Abs value reg}, it is natural to conjecture that an inequality of the form
\begin{equation}\label{our desired bound}
    \||u|^{p-1}u\|_{W^{s,q}(\mathbb{R}^d)}\lesssim \|u\|_{L^\infty(\mathbb{R}^d)}^{p-1}\|u\|_{W^{s,q}(\mathbb{R}^d)}
\end{equation}
holds if and only if $s<p+\frac{1}{q}$. However, to the best of our knowledge, such an estimate is not known. Some related estimates have been proven in \cite{MR2094590,MR1950719}, but unfortunately these only apply to real-valued expressions.  
In  \Cref{Nonlinear est section} (and more precisely in \Cref{Nonlinear estimate}) we will establish a slightly weaker version of the estimate \eqref{our desired bound} (which will still entirely suffice for our applications) in both the real and complex-valued case. We remark that the method of proof that we use to establish \Cref{Nonlinear estimate} seems to also be adaptable to more general nonlinear expressions $F(u)$ where $F(0)=0$ and $F\in C^{k,\alpha}$. This may be of independent interest. 
\medskip

 To prove the standard fractional chain rule \eqref{fractionalchainintro} one may invoke well-known techniques from paradifferential calculus. However, this strategy seems to break down in the range $p<s<p+\frac{1}{q}$ as it is too ``global" in nature. Instead, motivated by the fact that the main obstructions to obtaining \eqref{our desired bound} should come from points where $u$ vanishes linearly, it is natural to try to approximate $u$ by a suitable family of piecewise-linear functions, where one can more explicitly estimate the $W^{s,q}$ norm. This is essentially the method employed in \cite{MR1173747} for the case $u\mapsto |u|$ and $u$ real-valued. Of course, when $s$ is large, one cannot use this strategy directly because piecewise-linear functions may not themselves belong to $W^{s,q}$. We will postpone describing the necessary modifications to \Cref{Nonlinear est section}. For now, we remark that carrying out this strategy for complex-valued functions is not straightforward and will be the source of several interesting technical issues in our proof. Generally speaking, almost all previous results on the boundedness of nonlinear superposition operators are proven for real-valued functions.  In certain special cases (such as $u\mapsto |u|^p$) one can upgrade to a vector-valued inequality by standard tensorization techniques \cite{MR1781826}. However,  \eqref{our desired bound} is new even in the real case, and our proof works directly for complex scalars.
\medskip

To our knowledge, sharp inequalities of the form \eqref{abs val reg ine} or \eqref{our desired bound} have not been extensively applied to analyze nonlinear dispersive or parabolic equations. We expect, however,  that they will have a broad range of applicability. On the other hand, we stress that certain difficulties must be overcome when utilizing these inequalities in a well-posedness proof. In particular, since we will not always have control of the $L^\infty$ norm of solutions to \eqref{NLS} and \eqref{NLH},  we will have to appropriately modify  \eqref{our desired bound} before using it. We will not prove (or use) \eqref{our desired bound} precisely as stated, but will prove something of a similar form.
\medskip

With the fixed-time regularity of $|u|^{p-1}u$ now essentially understood, we move on to understanding the local well-posedness of \eqref{NLS} and \eqref{NLH}. For \eqref{NLS}, we begin by identifying the regime where the ideas of Kato and Tsutsumi may be justified on a fractional level. For this, we recall that solutions to the linear Schr\"odinger equation
\begin{equation}\label{linear schrodinger}
    i\partial_tu=\Delta u
\end{equation}
are strictly localized to the characteristic hypersurface $\tau=-|\xi|^2$, and therefore, one can freely measure the regularity of the solution using both time and space derivatives. On the other hand, for the nonlinear equation we can only ``trade" space and time derivatives when we a priori know that the solution has spacetime Fourier support near this hypersurface. Outside of this regime, $\partial_tu$ may not well approximate $\Delta u$, so other methods are needed.
\medskip

For the sake of our general discussion, let us suppose that we have a global solution to \eqref{NLS} and see to what extent we can prove a priori estimates. In practice, to be able to work with global solutions, we will need to work with an \eqref{NLS} equation with an appropriately time-truncated and regularized nonlinearity  (which we caveat is quite non-trivial to implement). Generally speaking, we will divide our analysis into estimating $u$ in two distinct regimes. The first is the ``low modulation" regime, where the solution is localized near the characteristic hypersurface and the second is the ``high modulation" regime, where it is localized far away from this hypersurface. In the low modulation regime,  we will obtain suitable $L_t^{\infty}L_x^2$ (and auxiliary) Strichartz estimates for $D_t^\frac{s}{2}u$. This strategy will require one to obtain an estimate for $\|D_t^{\frac{s}{2}}(|u|^{p-1}u)\|_{L^{q'}([0,T];L_x^{r'})}$ where $(q,r)$ is an admissible Strichartz pair. The simplest way to estimate this is as follows: Since necessarily $q',r'\leq 2$, we can apply H\"older's inequality in time and then Minkowski's inequality to obtain
\begin{equation*}
\|D_t^{\frac{s}{2}}(|u|^{p-1}u)\|_{L^{q'}([0,T];L_x^{r'})}\lesssim_T \|D_t^{\frac{s}{2}}(|u|^{p-1}u)\|_{L_x^{r'}L_t^2}.
\end{equation*}
In light of \eqref{our desired bound}, we expect such an estimate to close when $s<2p+1$, which already constitutes a whole derivative of improvement over the first threshold in \eqref{well-posedness known}. 
\begin{remark}
The reader may (correctly) note that applying H\"older in time in the above estimate to work exclusively with the $L_t^2$ norm is rather wasteful. We will later describe how one can further improve the restriction $s<2p+1$ by better optimizing the choice of Strichartz exponents above. This is an interesting point since classically Strichartz estimates had been used to improve the low regularity well-posedness theory for \eqref{NLS} rather than the high regularity theory.
\end{remark}
In the high modulation regime, we can interpret the Schr\"odinger operator $(i\partial_t-\Delta)$ as being elliptic (of order $1$ in time or of order $2$ in space). Therefore, in this region, we expect an estimate (possibly up to a minor dyadic summation loss) of the form
\begin{equation*}
\|u\|_{L_t^{\infty}H_x^s}\lesssim \|(i\partial_t-\Delta)u\|_{L_t^{\infty}H_x^{s-2}}\lesssim \||u|^{p-1}u\|_{L_t^{\infty}H_x^{s-2}}.
\end{equation*}
One then expects to be able to estimate the latter term on the right when $s<p+\frac{5}{2}$, which is a whole half derivative better than the second threshold in \eqref{well-posedness known}! That these heuristics can be justified is the content of our next theorem.
\begin{theorem}\label{M1 OOTP}
      Let $p>1$. Then \eqref{NLS} is locally well-posed in $H^s(\mathbb{R}^d)$ when $\max\{0,s_c\}< s<\min\{p+\frac{5}{2}, 2p+1\}$.
\end{theorem}
Notice that \Cref{M1 OOTP} recovers \Cref{M1}, as when $p\geq \frac{3}{2}$, the threshold $p+\frac{5}{2}$ is what obstructs well-posedness. Moreover, the above heuristics clearly identify the root cause of ill-posedness when $s\geq p+\frac{5}{2}$; we will show how to take advantage of this in \Cref{Ill section}.  For now, we remark that if one substitutes \eqref{our desired bound} with a more na\"ive Moser estimate, the above heuristics predict the previous well-posedness threshold in \eqref{well-posedness known}. However, contrary to previous works on \eqref{NLS}, we now have a clear and simple explanation as to why this threshold appears. That being said, the execution of \Cref{M1 OOTP} is still far from trivial, as we will have to carefully truncate the nonlinearity in order to use Littlewood-Paley theory and fractional derivatives in the time variable. The method of truncation that we present is a novel technical part of our analysis. Indeed, previous attempts at establishing results even in the restricted range \eqref{well-posedness known} at fractional regularity levels seem to have non-trivial gaps. For instance, in some previous approaches, there are issues with endpoint regularity losses that arise when appealing to abstract interpolation theorems. These issues were initially observed by the authors of \cite{MR3017270}. Difficulties in proving continuity of the data-to-solution map due to not having enough room to apply a standard contraction argument were also pointed out and resolved in special cases in \cite{MR2765515}. A more subtle (and pervasive) issue was observed in \cite{MR3917711} by Wada. There it was noted that many previous works do not make a consistent choice of function space to measure the fractional time regularity of solutions. For instance, such works often switch between the restriction characterization of fractional (in time) Besov spaces and the characterization by finite differences without accounting for the dependence on the length of the time interval $[0,T]$ (which is crucially used as a smallness parameter to close the requisite a priori estimates). A partial fix to this issue was given in \cite{MR3917711}, but it only applies when $s\leq 4$ and imposes the restrictions in \eqref{well-posedness known}.
\medskip

A second benefit of our well-posedness scheme is that it readily predicts the optimal high regularity well-posedness threshold for \eqref{NLH}. For \eqref{NLH}, the characteristic hypersurface degenerates to a point, so no low modulation region exists. Therefore, only the elliptic region matters. In this region, we may use the full scale of $W^{s,q}(\mathbb{R}^d)$ norms, which directly leads to the first half of \Cref{M1H}, i.e., the following theorem.
\begin{theorem}\label{heat theorem}
    Let $p>1$. Then \eqref{NLH} is locally well-posed in $W^{s,q}(\mathbb{R}^d)$ when $\max\{0,s_c\}< s<p+2+\frac{1}{q}$.
\end{theorem}
As we will see below, the high regularity threshold $s<p+2+\frac{1}{q}$ in  \Cref{heat theorem} is sharp for all $p>1$. Moreover, the proofs of ill-posedness for \eqref{NLS} and \eqref{NLH} at high regularity are very similar, as the same mechanism is responsible for the non-existence. We remark that \Cref{heat theorem} answers a question from \cite{cazenave2017non}. Interestingly, \cite[Theorem 1.1]{cazenave2017non} identifies the optimal high regularity well-posedness threshold for \eqref{NLH} in H\"older spaces when $1<p<2$ -- their theorem is consistent with the limiting case of our result but is seemingly much easier to execute on a technical level as standard Moser estimates may be employed.
\medskip

More generally, the above heuristics  predict a high regularity well-posedness threshold for a wide variety of dispersion generalized equations with a wide variety of nonlinearities $F(u)$. Roughly speaking, if the dispersion is of order $\alpha>1$ and the nonlinearity is bounded  on $H^s(\mathbb{R}^d)$ when $s<\mu$ (with a good estimate), then we expect to prove well-posedness when $s<\min\{\alpha\mu, \alpha + \mu\}$. Since the condition $s<\alpha+\mu$ stems from elliptic regularity, we expect it to be rather sharp. However, one may ask whether the low modulation threshold  is sharp. Surprisingly, the answer to this question is no!
\begin{theorem}\label{NLS improved}
    Let $p>1$. Then \eqref{NLS} is locally well-posed in $H^s(\mathbb{R})$ when $\max\{0,s_c\}\leq s<\min\{3p,p+\frac{5}{2}\}$.
\end{theorem}
The idea for proving \Cref{NLS improved} is to use the dispersive properties of the equation  (in particular, Strichartz estimates) to deduce improved integrability for the solution. From this extra integrability, we gain access to estimates of the form \eqref{our desired bound} with $q<2$, and hence can place more derivatives on the nonlinearity. To our knowledge, this is the first time Strichartz estimates have been used to seriously improve the high (rather than low) regularity results for a nonlinear dispersive PDE. Although \Cref{NLS improved} demonstrates that the low modulation threshold can generally be improved, we do not attempt to further optimize our results in this paper. Rather, \Cref{NLS improved} is presented more so as a proof-of-concept in its simplest, one-dimensional form. Whether this new threshold is sharp (and what the sharp threshold in higher dimensions is), we believe, is a very interesting open question. We also emphasize that the low modulation issues are confined to very low power nonlinearities and low dispersion models. In particular, for cubic and higher dispersion, we expect only the elliptic region to matter and hence sharp well-posedness to be entirely within reach in all cases.
\subsection{Broader applications}\label{Broader applications}
In this article we have primarily focused on the local well-posedness problem for the nonlinear Schr\"odinger and heat equations. However, we believe that the methods developed in this paper should have a more far-reaching impact. In particular, we expect that our proofs should apply (after suitable modifications) to dispersion generalized equations with a wide class of (derivative) nonlinearities, which appear quite frequently in the literature \cite{MR3901833,MR4533334,MR4304690}. Moreover, since we now know that solutions to rough evolution equations can have much more regularity than initially expected, certain problems related to their global dynamics should now be within reach. Below, we discuss a handful of interesting problems where our methods could prove to be useful. We emphasize, however, that the statements below should be seen as informal \emph{conjectures} that we believe to be within reach and of interest to the community. Proving them is well beyond the scope of this article.
\subsubsection{General nonlinearities and dispersion} The method of proof outlined above should apply to a wide class of dispersion generalized nonlinear evolution equations. As a model example, let us consider equations of the form
\begin{equation}\label{model eqn}
\begin{cases}
&\partial_tu-A(D_x)u=F(u),
\\
&u(0)=u_0,
\end{cases}
\end{equation}
where $u$ is the unknown, $F$ vanishes at the origin  and the symbol $A=A(\xi)$ satisfies the asymptotics
\begin{equation*}\label{multiplier condition}
0<\liminf_{|\xi|\to\infty}\frac{|A(\xi)|}{|\xi|^{\alpha}}<\infty
\end{equation*} 
 for some $\alpha> 1$. To formulate a precise high regularity well-posedness threshold for \eqref{model eqn}, let us recall that the fundamental property of the function $|z|^{p-1}z$ appearing in the nonlinearity of \eqref{NLS} and \eqref{NLH} is the boundedness of the superposition operator $u\mapsto |u|^{p-1}u$ on $W^{s,q}(\mathbb{R}^d)\cap L^\infty(\mathbb{R}^d)$. We remark that there is an extensive literature devoted to the boundedness of superposition operators $u\mapsto F(u)$ on various function spaces. In particular, we refer the reader to \cite{MR2652183,MR3262646} for a general conjecture which hopes to characterize all nonlinear maps $F:\mathbb{R}\to\mathbb{R}$ for which the superposition operator $u\mapsto F(u)$ is bounded on the Besov scale $B^s_{q,r}(\mathbb{R})\cap L^\infty(\mathbb{R})$ or on the Triebel-Lizorkin scale  $F^s_{q,r}(\mathbb{R})\cap L^\infty(\mathbb{R})$. For our purposes, however, let us \emph{assume} that for a fixed $q$ we have an estimate of the form
 \begin{equation*}\label{our desired bound gen}
    \|F(u)\|_{W^{s,q}(\mathbb{R}^d)}\lesssim_{\|u\|_{L^\infty(\mathbb{R}^d)}}\|u\|_{W^{s,q}(\mathbb{R}^d)}
\end{equation*}
for all $0\leq s<s^*=s^*(q)$ and then predict the high regularity well-posedness threshold for \eqref{model eqn}. From the above heuristics, in the high modulation regime we expect to place up to $s^*$ derivatives on the nonlinearity $F(u)$ and also obtain an elliptic-type gain of $\alpha$ derivatives for the solution. This leads to a $W^{\sigma,q}$ well-posedness threshold of $\sigma< s^*+\alpha$. Ignoring low regularity issues, in the parabolic case, we conjecture that this threshold should be sharp. On the other hand, in the dispersive case, one should expect $H^\sigma$ well-posedness at least up to the threshold $\sigma < \min \{\alpha s^*,s^*+\alpha\}$, with the $s^*+\alpha$ threshold being entirely sharp but with the low modulation threshold of $\alpha s^*$ being improvable in certain cases. One also expects these heuristics to be partially applicable to certain variable coefficient systems of pseudodifferential equations more general than \eqref{model eqn}. However, since low regularity issues play a prominent technical role in the analysis, we have chosen to carry out the systematic analysis only for \eqref{NLS} and \eqref{NLH}. Similarly, to keep the analysis in this paper somewhat manageable, we have chosen to work on $\mathbb{R}^d$ rather than on other domains such as $\mathbb{T}^d$ and have not attempted to further optimize to obtain the critical endpoint $s=\min\{0,s_c\}$ on $\mathbb{R}^d$. Nevertheless, this is likely of interest, especially if one can use it to obtain a low regularity continuation criterion for the full range of Sobolev exponents for which  \eqref{NLS} and \eqref{NLH} are well-posed. See \Cref{ssprop} for a further discussion on this.
\subsubsection{Derivative nonlinearities and quasilinear equations}  Although the scheme of proof that we present below will not be directly applicable, the ideas from this paper should also predict the high regularity well-posedness threshold for a large family of equations with derivative nonlinearities and more general quasilinear interactions. For the sake of simplicity, we will focus our exposition on Schr\"odinger equations rather than parabolic or dispersion generalized models.
\medskip

The study of the low regularity dynamics of general quasilinear Schr\"odinger equations with smooth background metrics and nonlinearities has seen significant advances in recent years \cite{ifrim2025global,ifrim2024global, MR4331023,MR4830552,MR4820290}. However, the case where the coefficients have limited regularity has been so far largely neglected, except in very specific situations. Rather than attempt to state a general conjecture, let us consider the most well-studied model problem; namely, that of the \emph{generalized derivative nonlinear Schr\"odinger equation}
\begin{equation}\label{gDNLS eq}\tag{gDNLS}
i\partial_tu+\partial_x^2u=i|u|^{2\sigma}\partial_xu.
\end{equation}
The equation \eqref{gDNLS eq} lies on the border between quasilinear and semilinear Schr\"odinger equations, and the semilinear techniques presented in this paper are insufficient to prove well-posedness of \eqref{gDNLS eq}. Nevertheless, many techniques have been developed to analyze such equations and, in particular, in our previous paper \cite{us} (see also \cite{MR4820290,MR4675424}) we proved global well-posedness for \eqref{gDNLS eq} in $H^s(\mathbb{R})$ when $1\leq s<4\sigma$ and $\sigma<1$ is not too small. 
\medskip

With the combination of the techniques from \cite{us} and this paper we now believe that the \emph{sharp} high regularity well-posedness threshold for \eqref{gDNLS eq} is within reach, for a large range of exponents $\sigma\geq \frac{1}{2}$. Indeed, in the proof of our main nonlinear estimate in \Cref{Nonlinear est section} we will characterize the precise boundedness properties of the mapping $u\mapsto |u|^{2\sigma}\partial_xu$, which when combined with the heuristics mentioned above and the tools from \cite{us} should allow one to prove well-posedness of  \eqref{gDNLS eq} in $H^s$ when $s<\min\{2\sigma+\frac{5}{2}, 2(2\sigma+\frac{1}{2})\}$. Conceivably, by carefully optimizing the estimates in \cite{us}, one could extend this to a sharp global well-posedness result, at least in the range of exponents $\sigma<1$ considered in \cite{us}. We emphasize, however, that actually proving this high regularity global well-posedness result for \eqref{gDNLS eq} would require significant effort, as establishing the necessary $H^1$ continuation criterion (which was necessary for a global result in this paper) is far from trivial. Moreover, we remark that when $\sigma<\frac{1}{2}$ we are unaware of \emph{any} well-posedness results for \eqref{gDNLS eq} and we do not claim that our methods will be applicable in this case.
\subsubsection{Nonlinear wave equations}
We next note that our methods are not restricted to equations which are first-order in time. In fact, they yield a large improvement in the well-posedness theory of nonlinear wave equations and related models, as can be found, e.g., in \cite{MR4627226,MR4228955,MR2763767,MR1368792,MR1683051}.
\medskip

 Since the linear operator $(\partial_{tt}-\Delta)$ for the nonlinear wave equation
\begin{equation}\label{NLW}\tag{NLW}
\begin{cases}
&\partial_{tt}u-\Delta u=\pm |u|^{p-1}u,
\\
&u(0)=u_0, \ \ \partial_tu(0)=u_1,
\end{cases}
\end{equation}
has a balanced number of space and time derivatives, the low modulation regime is what restricts the analysis. For this reason (as the wave operator can be viewed as elliptic of order $1$ near the characteristic surface), our methods predict an upper well-posedness threshold of $s<p+\frac{3}{2}$ for $(u,u_t)\in H^s(\mathbb{R}^d)\times H^{s-1}(\mathbb{R}^d)$. Assuming \eqref{our desired bound} (or more precisely, using \Cref{Nonlinear estimate}), it is not particularly difficult to verify that well-posedness for \eqref{NLW} holds when $\frac{d}{2}<s<p+\frac{3}{2}$. This seems to already significantly improve over the known high regularity results. However, a complete characterization in the style of Theorems~\ref{M1} and \ref{M1H} of all Sobolev exponents $s$ for which \eqref{NLW} is well-posed in $H^s(\mathbb{R}^d)\times H^{s-1}(\mathbb{R}^d)$ is beyond the scope of this article, and is an interesting open question. 
\subsubsection{Propagation of regularity and global dynamics}\label{ssprop} We finally make two remarks regarding global dynamics. First, we note that often one should expect that if a nonlinear evolution equation is global (and scattering) in some low regularity Sobolev space, then the flow should also be global (and scattering) in higher order Sobolev spaces for which the evolution is well-posed locally in time. When the nonlinearity is smooth, this often comes for free from the estimates one typically establishes in the proof of low regularity well-posedness. However, the question of propagation of regularity for rough evolution equations is often quite difficult to deal with. By adapting the techniques introduced in this paper, one should hopefully be able to prove sharp propagation of regularity results for various semilinear models of interest. We caveat that for quasilinear models, ensuring that the estimates are strong enough to facilitate a high regularity continuation result is more difficult. As an example, in our previous work \cite{us} we proved that global well-posedness in $H^1(\mathbb{R})$ for \eqref{gDNLS eq} when $\sigma<1$ propagates to $H^s(\mathbb{R})$ when $1\leq s<4\sigma$, but this involved a very delicate analysis and a more complicated functional structure than we use in the current article. Given our new advances, however, one would now expect global well-posedness to hold in even higher regularity norms, though this would require a detailed verification.
\medskip

Secondly, we note that to properly address the question of global well-posedness for evolution equations, one typically needs to know a priori that the solution is of sufficient regularity for known techniques to apply. As an example, we highlight the works of Killip-Visan \cite{MR2753625,MR2763767}, where for the nonlinear Schr\"odinger and wave equations the authors developed techniques to address the global well-posedness of solutions, assuming a strong enough local well-posedness theory. With the advances made in the current article, it is natural to expect that the problems left open in \cite{MR2753625,MR2763767} regarding removing the technical restrictions on the parameter $p$ which ensure sufficient smoothness of the nonlinearity may now also be within reach.
\subsection{Acknowledgments} We thank Sung-Jin Oh and Daniel Tataru for helpful discussions related to this work. During the writing of this paper, the authors were partially supported by the NSF grant DMS-2054975 as well as by the Simons Investigator grant of Daniel Tataru. The first author was also supported by a fellowship in the Simons Society of Fellows.  

\section{Notation and preliminaries}
 In this section, we settle notation and recall some standard tools. 
\subsection{Littlewood-Paley theory, function spaces and frequency envelopes}\label{LWP}
First, we recall the standard Littlewood-Paley decomposition. For this, we let $\phi_0:\mathbb{R}^d\to [0,1]$ be a smooth radial function supported in the ball $B_2(0)$ of radius $2$ which is identically one on $B_1(0).$ We let  $\phi(\xi):=\phi_0(\xi)-\phi_0(2\xi)$ and define 
$$ \widehat{P_ju}(\xi):=\phi(2^{-j}\xi)\widehat{u}(\xi), \hspace{5mm} j\in \mathbb{N},$$
$$\widehat{P_{0}u}(\xi):=\phi_0(\xi)\widehat{u}(\xi).$$
We then define for each $k\in \mathbb{N}$,
\begin{equation*}
P_{<k}:=\sum_{j=0}^{k-1}P_j, \hspace{5mm} P_{\geq k}:=\sum_{j=k}^\infty P_j.
\end{equation*}
 With the above notation, we have the standard inhomogeneous Littlewood-Paley decomposition
\begin{equation*}
1=\sum_{j= 0}^\infty P_j.
\end{equation*}
Correspondingly, we have the associated Littlewood-Paley trichotomy, or Bony paraproduct decomposition,
 \begin{equation*}
uv=T_uv+T_vu+\Pi (u,v),  
 \end{equation*}
where the paraproduct
\begin{equation*}
T_vu:=\sum_{k\geq 0}P_{<k-4}vP_ku
\end{equation*}
 selects the portion of the product $uv$ where $u$ is at high frequency compared to $v$. 
We refer the reader to \cite{bony1981calcul} and \cite{metivier2008differential} for some basic properties of these operators.
 We will use the notation $\tilde{P}_j$, $\tilde{P}_{<j}$, $\tilde{P}_{\geq j}$ to denote slightly enlarged or shrunken frequency localizations. For example, we may denote $P_{[j-3,j+3]}$ by $\tilde{P_j}$. 
 \medskip
 
 In this paper, we will have to apply Littlewood-Paley theory in both space and time. To clearly distinguish between these cases, we will follow the convention that $P_j$  denotes a spatial projection and $S_j$ denotes a temporal projection, with similar conventions for $P_{\geq j}$, $S_{\geq j}$, etc.
\medskip

For $1<q<\infty$ and $s\in \mathbb{R}$, we define the Sobolev spaces $W_x^{s,q}(\mathbb{R}^d)$ via Fourier multipliers as the closure of Schwartz functions under the norm
\begin{equation*}
\|u\|_{W_x^{s,q}(\mathbb{R}^d)}:=\|\langle D_x\rangle^su\|_{L^q_x(\mathbb{R}^d)}:=\|\mathcal{F}^{-1}(\langle \xi \rangle^s\widehat{u})\|_{L^q_x(\mathbb{R}^d)}.
\end{equation*}
As usual, we let $H_x^s(\mathbb{R}^d):=W_x^{s,2}(\mathbb{R}^d)$. When $q=\infty$ and $s$ is an integer, $W_x^{s,q}(\mathbb{R}^d)$ will have its usual meaning. When $s$ is not an integer, we will sometimes abuse notation and write $W_x^{s,\infty}(\mathbb{R}^d)$ for the Bessel potential space. 
\medskip

For a time $T>0$ we will use the notation $L^r_TW_x^{s,q}(\mathbb{R}^d)$ to denote either $L^r([0,T]; W_x^{s,q}(\mathbb{R}^d))$ or $L^r([-T,T]; W_x^{s,q}(\mathbb{R}^d))$ depending on the context (i.e.~for the nonlinear heat and Schr\"odinger equations, respectively). We will reserve the notation $L^r_tW_x^{s,q}:=L^r(\mathbb{R}; W_x^{s,q}(\mathbb{R}^d))$ for functions defined globally in time. By slight abuse of notation, we will also write $\|u\|_{W_x^{s,q}L^r_T}$ to mean $\|\langle D_x\rangle^s u\|_{L^q_xL^r_T}$, with a similar convention for $W_t^{s,q}L_x^r$ when using time derivatives. 
\medskip

The Littlewood Paley projections are evidently bounded on $W_x^{s,q}(\mathbb{R}^d)$ and satisfy various \emph{Bernstein inequalities} (see~\cite[p.~333]{tao2006nonlinear}). Such estimates  will be used liberally. We will also need the following somewhat less standard vector-valued variant of these inequalities.
\begin{proposition}
Let $1\leq p,q\leq \infty$, $j>0$ and $s\in \mathbb{R}$. Then we have 
\begin{equation*}
    \|D^s_xP_ju\|_{L^p_xL^q_T}\approx 2^{js}\|P_ju\|_{L^p_xL^q_T}.
\end{equation*}
\end{proposition}
\begin{proof}
See \cite[Proposition 2.9]{us} and its proof.
\end{proof}
Another way we will employ the Littlewood-Paley projections is to define frequency envelopes, which were a tool introduced by Tao in \cite{tao2001global}.  To define these, suppose that we are given a Sobolev type space $X$ such that
\begin{equation*}
\|P_{0}u\|_X^2+\sum_{j=1}^\infty\|P_ju\|_{X}^2\approx \|u\|_{X}^2.
\end{equation*}
A \emph{frequency envelope} for $u$ in $X$ is a positive sequence $(c_j)_{j\in \mathbb{N}_0}$ such that 
\begin{equation*}\label{property1}
  \|P_{0}u\|_X\lesssim c_0\|u\|_X, \ \ \   \|P_ju\|_X\lesssim c_j\|u\|_X, \ \ \ \sum_{j=0}^\infty c_j^2\lesssim 1.
\end{equation*}
We say that a frequency envelope is \emph{admissible} if $c_0\approx 1$ and it is slowly varying, meaning that
$$c_j\leq 2^{\delta|j-k|}c_k, \ \ \ j,k\geq 0, \ \ \ 0<\delta \ll 1.$$
An admissible frequency envelope always exists, say by 
\begin{equation*}\label{env1}
    c_j=2^{-\delta j}+\|u\|_X^{-1}\max_{k\geq 0} 2^{-\delta |j-k|}\|P_ku\|_X.
\end{equation*}
Frequency envelopes will be primarily utilized for studying the regularity of the data-to-solution map for \eqref{NLS}.
\begin{remark}
The reader might wonder why we are using frequency envelopes (which are typically more suited to quasilinear problems) to study the regularity of the data-to-solution map. In our setting, the standard approach of going through the contraction mapping theorem will not work, because we can no longer justify the Lipschitz-type bound one would have to prove (due to the limited regularity of the nonlinearity).
\end{remark}

\subsection{Standard nonlinear estimates}
We now recall various standard estimates that will be used throughout the paper. 
\subsubsection{Moser estimates} We begin by recalling various Moser estimates. The first is a standard vector-valued variant for $C^1$ functions. 
\begin{proposition}\label{Moservec} Let $F\in C^1(\mathbb{C})$, $\alpha\in (0,1)$, $q,r,q_1,q_2,r_2\in (1,\infty)$ and $r_1\in (1,\infty]$ with
\begin{equation*}
\frac{1}{q}=\frac{1}{q_1}+\frac{1}{q_2},\hspace{5mm}\frac{1}{r}=\frac{1}{r_1}+\frac{1}{r_2}.
\end{equation*}
Then
\begin{equation*}
\|D_x^{\alpha}F(u)\|_{L_x^qL_T^r}\lesssim \|F'(u)\|_{L_x^{q_1}L_T^{r_1}}\|D_x^{\alpha}u\|_{L_x^{q_2}L_T^{r_2}}.
\end{equation*}
\end{proposition}
\begin{proof}
See	Theorem A.6 of \cite{kenig1993well}. 
\end{proof}
The scalar version of the above estimate is as follows.
\begin{proposition}\label{scalarMoser}
Let $F\in C^1(\mathbb{C})$, $\alpha\in (0,1)$ and $r,r_1,r_2\in(1,\infty)$ with $\frac{1}{r}=\frac{1}{r_1}+\frac{1}{r_2}$. Then
\begin{equation*}
\|D_x^{\alpha}F(u)\|_{L^r}\lesssim \|F'(u)\|_{L^{r_1}}\|D_x^{\alpha}u\|_{L^{r_2}}.   
\end{equation*}
\end{proposition}
\begin{proof}
See Proposition 3.1 of \cite{MR1124294}.
\end{proof}
For the purposes of this paper, it will be crucial to understand the precise regularity of the nonlinearity $|u|^{p-1}u$. However, we will still occasionally make use of the more standard Moser estimate for $|u|^{p-1}u$ below. 
\begin{proposition}\label{Crude Moser est}
Let $p>1$ and assume that $0<s<p$ and $r,r_1,r_2\in (1,\infty)$ satisfy $\frac{1}{r}=\frac{1}{r_1}+\frac{1}{r_2}$. Then 
\begin{equation*}
\||u|^{p}\|_{W^{s,r}}+\| |u|^{p-1}u\|_{W^{s,r}(\mathbb{R}^d)}\lesssim \|u\|_{L^{r_1(p-1)}(\mathbb{R}^d)}^{p-1}\|u\|_{W^{s,r_2}(\mathbb{R}^d)}.
\end{equation*}
\end{proposition}
The following result due to Killip and Visan from \cite[Appendix A]{MR2709575} will also serve as a convenient tool at various points in the paper.
\begin{proposition}\label{KV Prop}
Let $F$ be a H\"older continuous function of order $0<\alpha<1$. Then for every $0<\sigma<\alpha$, $1<r<\infty$ and $\frac{\sigma}{\alpha}<s<1$ we have
\begin{equation*}
\| D_x^\sigma F(u) \|_{L_x^r}\lesssim \| |u|^{\alpha-\frac{\sigma}{s}} \|_{L_x^{r_1}} \|D_x^s u\|_{L_x^{\frac{\sigma}{s}r_2}}^{\frac{\sigma}{s}}
\end{equation*}
provided that $\frac{1}{r}=\frac{1}{r_1}+\frac{1}{r_2}$ and $(1-\frac{\sigma}{\alpha s})r_1>1$.
\end{proposition}
\subsubsection{Fractional Leibniz rules}
Next, we recall the vector-valued fractional Leibniz rule.
	\begin{proposition}\label{Leib2}
		Let $\alpha\in (0,1)$, $\alpha_1,\alpha_2\in [0,\alpha]$,  $q,q_1,q_2,r,r_1,r_2\in (1,\infty)$ satisfy $\alpha_1+\alpha_2=\alpha$ and $\frac{1}{q}=\frac{1}{q_1}+\frac{1}{q_2}$, $\frac{1}{r}=\frac{1}{r_1}+\frac{1}{r_2}$. Then
		\begin{equation*}
			\|D_x^{\alpha}(fg)-D_x^{\alpha}fg-fD_x^{\alpha}g\|_{L_x^qL_T^r}\lesssim \|D_x^{\alpha_1}f\|_{L_x^{q_1}L_T^{r_1}}\|D_x^{\alpha_2}g\|_{L_x^{q_2}L_T^{r_2}}.
		\end{equation*}
		The endpoint cases $r_1=\infty, \alpha_1=0$ as well as $(q,r)=(1,2)$ are also allowed.
	\end{proposition}
	\begin{proof}
	See \cite[Lemma 2.6]{kenig2003local} or \cite[Lemma 3.8]{kenig2006global}.
	\end{proof}
	We also have the scalar fractional Leibniz rule, also known as the Kato-Ponce inequality.
	\begin{proposition}\label{Leib1}
Let $s>0$ and $\frac{1}{r}=\frac{1}{p_1}+\frac{1}{q_1}=\frac{1}{p_2}+\frac{1}{q_2}$, $1<r<\infty$, $1<p_1,q_2\leq\infty$, $1<p_2,q_1<\infty$. Then
\begin{equation*}\label{katoponce}
\|D_x^s(fg)\|_{L_x^r}\lesssim_{s,d,p_1,p_2,q_1,q_2} \|f\|_{L_x^{p_1}}\|g\|_{W_x^{s,q_1}}+\|f\|_{W_x^{s,p_2}}\|g\|_{L_x^{q_2}}.\
\end{equation*}
	\end{proposition}
	\begin{proof}
			See \cite{grafakos2014kato}.
	\end{proof}

    \subsubsection{Strichartz Estimates}
We now recall the standard Strichartz estimates for the Schr\"odinger equation. Recall that a pair of exponents $(q,r)$ is \emph{Schr\"odinger-admissible} if $\frac{2}{q}+\frac{d}{r}=\frac{d}{2}$, $2\leq q,r\leq \infty$ and $(q,r,d)\neq (2,\infty,2)$.
\begin{theorem}[Strichartz estimates for Schr\"odinger] For any admissible exponents $(q,r)$ and $(\tilde{q},\tilde{r})$ we have the homogeneous Strichartz estimate
\begin{equation*}
\|e^{-it\Delta}u_0\|_{L^q_tL^r_x(\mathbb{R}\times\mathbb{R}^d)}\lesssim_{d,q,r}\|u_0\|_{L^2_x(\mathbb{R}^d)}
\end{equation*}
and the inhomogeneous Strichartz estimate
\begin{equation*}
\bigg\| \int_{s<t} e^{-i(t-s)\Delta}F(s)ds\bigg\|_{L^q_tL^r_x(\mathbb{R}\times\mathbb{R}^d)}\lesssim_{d,q,r,\tilde{q},\tilde{r}}\|F\|_{L_t^{\tilde{q}'}L_x^{\tilde{r}'}(\mathbb{R}\times\mathbb{R}^d)}.
\end{equation*}
\end{theorem}

To conclude the section we remark that, for some real number $\beta$, we will write $\beta_+$ or $\beta+$ to mean $\beta+c\epsilon$ for some $0<\epsilon\ll 1$ and a uniform constant $c>0$ which may change from line to line by a fixed factor. This convention will be useful for streamlining the notation in upcoming sections.
\section{The main nonlinear estimate}\label{Nonlinear est section}
In this section, we establish the variant of \eqref{our desired bound} for $u\mapsto |u|^{p-1}u$ that we will use in our local well-posedness results. In the following, $u:\mathbb{R}\to\mathbb{C}$ will be a function defined on the real line and taking values in $\mathbb{C}$. Our main result is as follows.
\begin{theorem}\label{Nonlinear estimate}
    Let $p>1$, $1< q< \infty$ and $p\leq s<p+\frac{1}{q}$. Then for every $\varepsilon>0$ sufficiently small and $u:\mathbb{R}\to\mathbb{C}$ with $u\in W^{s,q}$, we have
    \begin{equation}\label{Nonlinear estimate1}
        \| |u|^{p-1}u\|_{W^{s,q}}\lesssim_\epsilon \|u\|_{W^{\frac{1}{q}+\epsilon,q}}^{p-1}\|u\|_{W^{s,q}}.
    \end{equation}
\end{theorem}
\begin{remark}
The above inequality is one-dimensional, but it can be used to generate higher-dimensional variants via Fubini's theorem and Sobolev embeddings. The one-dimensional variant will entirely suffice for our applications, as later we will apply it only in the time variable or to individual components $x_1,...,x_d$ in the spatial variable. 
\medskip

We  note that the estimate \eqref{Nonlinear estimate1} is slightly weaker than the estimate \eqref{our desired bound} (in the sense that $W^{\frac{1}{q}+\epsilon,q}$ embeds into $L^{\infty}$). However, \eqref{Nonlinear estimate1} applies in great generality, to complex-valued functions, and, to our knowledge, is the first of its kind. It is conceivable that the factor $\|u\|_{W^{\frac{1}{q}+\epsilon,q}}$ in this estimate could be replaced by $\|u\|_{L^{\infty}}$ with more sophisticated arguments, but we do not pursue this here as it will not be necessary for our main applications.
\end{remark}
Before proving this theorem, we will first set the stage with some notational preliminaries. In the sequel, we let $m$ be the smallest integer such that $0\leq s-m<1$.
    Throughout the proof, we will write $u=f+ig$, where $f=\Re(u)$ and $g=\Im(u)$. To prove \eqref{Nonlinear estimate1}, it is clearly sufficient to prove the estimate
    \begin{equation}\label{Non est}
        \|\partial_x^m\left(|u|^{p-1}u\right)\|_{W^{s-m,q}}\lesssim \|u\|_{W^{\frac{1}{q}+\epsilon,q}}^{p-1}\|u\|_{W^{s,q}}.
    \end{equation}
By repeated applications of the chain rule, we have
\begin{equation*}\label{non est2}
\begin{split}
    \partial_x^m\left(|u|^{p-1}u\right)&=C_1 |u|^{p+1-2m}u_x\Re\left(\overline{u}u_x\right)^{m-1}+C_2 |u|^{p-1-2m}u\Re\left(\overline{u}u_x\right)^m+F(u)
    \\
    &=:N(u,u_x)+N'(u,u_x)+F(u),
\end{split}
\end{equation*}
where $C_1$, $C_2$ are explicit constants and $F(u)$ is a term involving a better distribution of derivatives. By using elementary paradifferential calculus, standard Moser estimates and the fractional chain rule, the term $F(u)$ can be easily shown to satisfy \eqref{Non est}.  The main difficulty is therefore in proving the following bound:
\begin{equation}\label{non est3}
    \|(N,N')\|_{W^{s-m,q}}\lesssim_\epsilon \|u\|_{W^{\frac{1}{q}+\epsilon,q}}^{p-1}\|u\|_{W^{s,q}}.
    \end{equation}
We note that the restriction $s\geq p$ implies that $p-m<1.$ We crucially note that $p-m$ (and $p-s$) may be negative, which is the main difficulty in establishing this estimate. Below, we will focus on estimating the term $N$ as the estimate for $N'$ is virtually identical. 
\medskip

Next, we briefly outline (at a somewhat heuristic level to begin with) our strategy for obtaining \eqref{non est3} and some important technical caveats. Let $j\geq 0$. If $p-m\leq 0$, roughly speaking, we will try to approximate $D_x^{s-m}N(u,u_x)$ by
\begin{equation*}
D_x^{s-m}N(u,u_x)\approx D_x^{s-m}N(u_j,\partial_xu_j)+\mathcal{E}_j,
\end{equation*}
 where $u_j$ is a suitable continuous piecewise-linear approximation of $u$ to be defined below and where $\mathcal{E}_j$ is an error term converging to zero in $L^q$ as $j\to\infty$. As mentioned in the introduction, the motivation for using such an approximation is that the main obstruction in the estimate \eqref{Nonlinear estimate1} should essentially come from the first-order ``linear" degeneracies in $u$. This strategy works well when $p-m\leq 0$  but can fail when $p-m>0$. This is because $\partial_xu_j$ is a piecewise constant function, and therefore is not regular enough to lie in $W^{r,q}$ for any $r>\frac{1}{q}$. To remedy this issue, we will instead obtain an approximation of $D_x^{s-m}N(u,u_x)$ which roughly looks like
\begin{equation*}
D_x^{s-m}N(u,u_x)\approx N(u,D_x^{s-m}u_x)+N(D_x^{s-m}u_j,\partial_xu_j)+\mathcal{E}_j.
\end{equation*}
Before making the above heuristics precise, we list a few simple but important facts that will be used heavily in our analysis below. The first tool that we will need is the following elementary but precise H\"older bound.
\begin{lemma}[Precise H\"older bound]\label{PHB lemma}
    Let $\alpha\in \mathbb{R}$ ($\alpha$ can be negative) and $0\leq \beta \leq 1$. Suppose that $\alpha_1, \alpha_2\in \mathbb{R}$ satisfy $\alpha_1+\alpha_2=\alpha$ and $\alpha_2\in \mathbb{Z}$. Then for every $(z,h)\in \mathbb{C}^2$, we have
    \begin{equation*}\label{PHB}
        \frac{\left| \left|z+h\right|^{\alpha_1}(z+h)^{\alpha_2}-|z|^{\alpha_1}z^{\alpha_2}\right|}{|h|^\beta}\lesssim_{\alpha,\beta} |z+h|^{\alpha-\beta}+|z|^{\alpha-\beta}.
    \end{equation*}
\end{lemma}
\begin{proof}
    We break the proof into two cases. If $|z|\leq 2|h|$, then we simply apply the triangle inequality to estimate
    \begin{equation*}
        \begin{split}
            \left| \left|z+h\right|^{\alpha_1}(z+h)^{\alpha_2}-|z|^{\alpha_1}z^{\alpha_2}\right|&\leq |z+h|^\alpha+|z|^\alpha
            \\
            &\leq |z+h|^{\alpha-\beta}|z+h|^\beta +|z|^{\alpha-\beta}|z|^\beta
            \\
            &\lesssim |h|^\beta\left(|z+h|^{\alpha-\beta}+|z|^{\alpha-\beta}\right).
        \end{split}
\end{equation*}
Next, we assume that $|h|<\frac{1}{2}|z|$. By averaging, we may control
\begin{equation*}
      \left| \left|z+h\right|^{\alpha_1}(z+h)^{\alpha_2}-|z|^{\alpha_1}z^{\alpha_2}\right|\leq |h|\sup_{\tau\in [0,1]}|z+\tau h|^{\alpha-1}.
\end{equation*}
Since $|h|<\frac{1}{2}|z|$, we also have
\begin{equation*}
    |h|\sup_{\tau\in [0,1]}|z+\tau h|^{\alpha-1}\approx |h||z|^{\alpha-1}\leq |h||z|^{\beta-1}|z|^{\alpha-\beta}\leq |h||h|^{\beta-1}|z|^{\alpha-\beta}\leq |h|^\beta |z|^{\alpha-\beta},
\end{equation*}
where we used that $\beta-1\leq 0$ in the second last inequality. This completes the proof.
\end{proof}
Next, we describe the piecewise-linear approximation of $u$ that was alluded to above. Given a real-valued function $h$ and $j\geq 0$, we define the corresponding piecewise-linear approximation $h_j$ of $h$ on the interval $A_{k}^j:=[2^{-j}k,2^{-j}(k+1)]$, $k\in\mathbb{Z},$  by the formula
\begin{equation*}
h_j(x):=h(2^{-j}k)+2^j(x-2^{-j}k)\Delta_{j,k}h,
\end{equation*}
where 
\begin{equation*}
\Delta_{j,k}h:=h\left(\frac{k+1}{2^j}\right)-h\left(\frac{k}{2^j}\right).
\end{equation*}
We then define the complex-valued piecewise-linear approximation $u_j$ of $u$ by
\begin{equation*}
        \begin{split}
            u_j(x):=f_j(x)+ig_j(x):=(\Re(u))_j(x)+i(\Im(u))_j(x).
        \end{split}
    \end{equation*}
As suggested above, we will frequently have to estimate the nonlinear functions $N(u,u_x)$. To simplify notation, we will adopt the convention that
\begin{equation*}
N_j:=N(u_j,\partial_xu_j).
\end{equation*}
That is, $N_j$ is obtained by replacing the arguments in $N$ with $u_j$ and $\partial_x u_j$. To facilitate the use of H\"older-type norms, we also define, for $0<h\leq 1$,
\begin{equation*}
\delta_hf(x):=f^h(x)-f(x), \hspace{5mm} f^h(x):=f(x+h).
\end{equation*}
The next ingredient that we will need for the proof of \Cref{Nonlinear estimate} is the following pair of estimates which follow essentially from Proposition $2$ in \cite{MR1173747}. We will use these bounds to compare the quality of the sequence of piecewise-linear approximations to the Sobolev norm of a given function. 
\begin{proposition}\label{Oswald}
Let $\alpha>s$. The following estimates hold. 
\begin{enumerate}
    \item \label{parta} (First order differences).  Let $f\in W^{\alpha,q}$ with $1< q< \infty$ and $1\leq s<1+\frac{1}{q}.$ 
    Assume that $(\Lambda_j)_{j\in \mathbb{N}}$ is a sequence of subsets of $\mathbb{Z}$ such that if $k$ and $k'$ are two consecutive integers in $\Lambda_j$ we have $\Delta_{j,k}f\Delta_{j,k'}f\leq 0$. There holds
    \begin{equation*}
        \|2^{j(s-\frac{1}{q})}\Delta_{j,k}f\|_{l^q_{j,k}}:=\left(\sum_{j\geq 0}\sum_{k\in \Lambda_j}2^{j(sq-1)}|\Delta_{j,k}f|^q\right)^{\frac{1}{q}}\lesssim_{\alpha} \|f\|_{W^{\alpha,q}}.
    \end{equation*}
    \item\label{partb} (Second order differences).  Let $f\in W^{\alpha,q}$ with $1< q<\infty$ and $1\leq s<1+\frac{1}{q}$. Define the second-order finite difference
    \begin{equation*}
\Delta_{j,k}^2f:=f\left(\frac{k+2}{2^j}\right)-2f\left(\frac{k+1}{2^j}\right)+f\left(\frac{k}{2^j}\right).
\end{equation*}
There holds
\begin{equation*}
   \|2^{j(s-\frac{1}{q})}\Delta_{j,k}^2f\|_{l_{j,k}^q}:=\left(\sum_{j\geq 0}\sum_{k\in \mathbb{Z}}2^{j(sq-1)}|\Delta^2_{j,k} f|^{q}\right)^{\frac{1}{q}}\lesssim_\alpha \|f\|_{W^{\alpha,q}}.
\end{equation*}
\end{enumerate}
\end{proposition}
The above estimates only appear implicitly in the analysis in \cite{MR1173747} (which considers the nonlinear bound for $f\mapsto |f|$ in the real-valued case). For transparency, we outline the proof.
\begin{proof}
We begin by recalling an equivalent characterization of the Besov space $B_{q,q}^s$ from \cite{MR1173747}. For each $m\geq 1$, we define the m'th order modulus of smoothness,
\begin{equation*}
\omega_m(2^{-j},f)_q:=\sup_{0<h\leq 2^{-j}}\|\delta_h^mf\|_{L^q(\mathbb{R})}
\end{equation*}
and the Besov norm
\begin{equation*}
\|f\|_{B_{q,q}^s}\coloneq\|f\|_{L^q}+\|2^{js}\omega_m(2^{-j},f)_q\|_{l^q_j},
\end{equation*}
where $m$ is any integer with $m>s$. Our main technical input in the proof of \Cref{Oswald} will be the following result from \cite{MR1173747}.
\begin{lemma}[Proposition $2$ in \cite{MR1173747}]\label{splineapprox}
For any $f\in B^s_{q,q}$ with $1< q< \infty$ and $\frac{1}{q}<s<1+\frac{1}{q}$, there holds
\begin{equation*}
\|f\|_{B_{q,q}^s}\approx \|f\|_{L^q} +\|2^{js}(f-f_j)\|_{\ell^q_jL^q}.
\end{equation*}
\end{lemma}
With the above result in hand, we turn to proving the second-order difference bound. We write $\Delta_j^2f$ as shorthand for the step-function which restricts to the constant $\Delta_{j,k}^2f$ on each interval $[2^{-j}k,2^{-j}(k+1))$. We have
\begin{equation*}
\|2^{j(s-\frac{1}{q})}\Delta_{j,k}^2f\|_{l_{j,k}^q}\lesssim \|2^{js}\Delta_{j}^2f\|_{l_{j}^qL^q}\lesssim \|2^{js}(f-f_j)\|_{l_j^qL^q}+\|2^{js}\omega_2(2^{-j},f)_q\|_{l_j^q},
\end{equation*}
where the last inequality follows from  equation (12) and the estimates below Proposition $2$ in \cite{MR1173747}. Using \Cref{splineapprox} and the embedding $W^{\alpha,q}\subset B^{s}_{q,q}$ we conclude that
\begin{equation*}
\|2^{j(s-\frac{1}{q})}\Delta_{j,k}^2f\|_{l_{j,k}^q}\lesssim \|f\|_{B^s_{q,q}}\lesssim \|f\|_{W^{\alpha,q}}.
\end{equation*}
This completes the proof of property (ii). Property (i) follows by repeating the arguments used to bound $a_k$ on pages $64$ and $65$ of \cite{MR1173747}, as was previously observed in \cite[Lemma 3.9]{MR1950719}.
\end{proof}
\begin{remark}
In \cite{MR1173747}, the above estimates are phrased in terms of more refined Besov-type norms which only require the weaker restriction $\alpha\geq s$. In practice, this will not be needed in our analysis later, so we prefer to phrase estimates in terms of $W^{\alpha,q}$ norms to avoid cumbersome notation and unnecessary (for our purposes) technical refinements.
\end{remark}
Now, we can state the main auxiliary estimate for this section, which will give us the necessary control over the approximation $N_j$ of $N(u,u_x)$. Before giving the precise statement, we note the following very simple lemma which will be needed in the regime $0<p-m<1$.
\begin{lemma}\label{nonlineardecomp} The nonlinear term $N(u,u_x)$ decomposes into a finite sum of terms of the form $N^{p-m}(u)N'(u_x)$ where $|N^{p-m}(u)|=|u|^{p-m}$ and $N'$ is a smooth $m$-linear expression in $u_x$ and $\overline{u}_x$ with $|N'(u_x)|=|u_x|^m$.
\end{lemma}
\begin{proof}
This follows from straightforward, direct computation. 
\end{proof}
The following proposition is the main technical input needed to obtain the nonlinear estimate \eqref{Non est}.
\begin{proposition}\label{new prop}
 Let $s_0:=s+\epsilon$ with $0<\epsilon\ll 1$ and with $s_0<p+\frac{1}{q}$ and $s_0-m<1$. Define $0<\sigma\coloneq s_0-m<1$. Then the following statements hold.
 \begin{enumerate}
     \item\label{1} Let $p-m\leq 0$ and $0<h\leq 1$. For $0<\epsilon'\ll\epsilon$, there holds
     \begin{equation}\label{firstcase}
     h^{-q(\sigma-\epsilon')}\sum_{j> \log(h^{-1})}2^{jq\epsilon'}\|\delta_h(N_{j+1}-N_j)\|_{L_x^q}^q\lesssim_\epsilon \|u\|_{W^{\frac{1}{q}+\epsilon,q}}^{q(p-1)}\|u\|_{W^{s,q}}^q
    \end{equation}
    and
\begin{equation}\label{firstcase2}
h^{-q\sigma}\sum_{0\leq j\leq  \log(h^{-1})}\|\delta_h N_j\|_{L_x^q}^q\lesssim_\epsilon \|u\|_{W^{\frac{1}{q}+\epsilon,q}}^{q(p-1)}\|u\|_{W^{s,q}}^q.
\end{equation}
    \item\label{2} Let $0<p-m<1$ and $0<h\leq 1$. For $0<\epsilon'\ll\epsilon$, there holds 
     \begin{equation}\label{secondcase}
       h^{-q(\sigma-\epsilon')}\sum_{j>\log(h^{-1})}2^{jq\epsilon'}\|N'_{j+1}\delta_hN_{j+1}^{p-m}-N'_j\delta_hN_j^{p-m}\|_{L_x^q}^q\lesssim_\epsilon \|u\|_{W^{\frac{1}{q}+\epsilon,q}}^{q(p-1)}\|u\|_{W^{s,q}}^q
    \end{equation}
    and
    \begin{equation}\label{secondcase2}
 h^{-q\sigma}\sum_{0\leq j\leq \log(h^{-1})}\|N'_{j}\delta_hN_{j}^{p-m}\|_{L_x^q}^q\lesssim_\epsilon \|u\|_{W^{\frac{1}{q}+\epsilon,q}}^{q(p-1)}\|u\|_{W^{s,q}}^q.
    \end{equation}
    \item\label{3} For every $j\geq 0$, there holds
     \begin{equation}\label{lowfreqbound}
        \|N_j\|_{L^q}\lesssim_\epsilon \|u\|_{W^{\frac{1}{q}+\epsilon,q}}^{p-1}\|u\|_{W^{s,q}}.
     \end{equation}
\item\label{4} Let $p-m\leq 0$. Then $N_j$ converges to $N$ in $L^q$. That is, as $j\to\infty$, we have
\begin{equation*}
N_j\to N\hspace{2mm}\text{in}\hspace{2mm}L^q(\mathbb{R}).
\end{equation*}
 \end{enumerate}
\end{proposition}
Before establishing this proposition, we  show how it implies \Cref{Nonlinear estimate}. 
\medskip

\textbf{Case 1: $p-m\leq 0$.} We begin with the case $p-m\leq 0$. We note that since $\sigma>s-m$, the embedding $B^{\sigma}_{q,\infty}\subset W^{s-m,q}$ together with \eqref{firstcase}-\eqref{firstcase2} and \eqref{lowfreqbound} imply the uniform in $j$ bound,
\begin{equation}\label{unifj}
\|N_j\|_{W^{\sigma',q}}\lesssim_{\epsilon,\sigma'} \|u\|_{W^{\frac{1}{q}+\epsilon,q}}^{p-1}\|u\|_{W^{s,q}}
\end{equation}
for every $\sigma'$ satisfying $s-m\leq\sigma'<\sigma$.  Therefore, the estimate \eqref{non est3} will follow if we can show that $N_j$ converges to $N$ in $W^{s-m,q}$. This follows by interpolating the uniform bound \eqref{unifj} with the fact that $N_j\to N$ in $L^q$. 
\begin{remark}
We note that since $p-m\leq 0$, we must necessarily have $m\geq 2$. Therefore, $N_j$ (which has only factors of $u_x$ and no higher order derivatives) is slightly more regular than a $W^{s-m,q}$ function. This is why we can work with $W^{\sigma',q}$ above. We caution that this reasoning is not valid when $0<p-m<1$, as $m=1$ is possible. This is why the analysis in case 2 below is a bit more delicate.
\end{remark}

\textbf{Case 2: $0<p-m<1$.} In this case, we use the paradifferential decomposition $N=(N-T_{N'}N^{p-m})+T_{N'}N^{p-m}$. On one hand, by standard paraproduct estimates, interpolation and Sobolev embeddings give
\begin{equation*}
\|N-T_{N'}N^{p-m}\|_{W^{s-m,q}}\lesssim \|u\|_{C^{\epsilon}}^{p-m}\|N'\|_{W^{s-m,q}}\lesssim \|u\|_{W^{\frac{1}{q}+\epsilon,q}}^{p-1}\|u\|_{W^{s,q}}.
\end{equation*}
On the other hand, if $s-m<\sigma'<\sigma$, then from the embedding $B_{q,\infty}^{\sigma'}\subset W^{s-m,q}$, we have
\begin{equation*}
\begin{split}
\|T_{N'}N^{p-m}\|_{W^{s-m,q}}&\lesssim \|T_{N'}N^{p-m}\|_{L^q}+\sup_{0<h\leq 1}h^{-\sigma'}\|\delta_h(T_{N'}N^{p-m})-T_{N'}\delta_hN^{p-m}\|_{L^q}
\\
&+\sup_{0<h\leq 1}h^{-\sigma'}\|N'\delta_hN^{p-m}-T_{N'}\delta_hN^{p-m}\|_{L^q}
\\
&+\sup_{0<h\leq 1}h^{-\sigma'}\|N'\delta_hN^{p-m}\|_{L^q}.
\end{split}
\end{equation*}
Using standard paradifferential calculus estimates, the first three terms on the right-hand side can be estimated by $\|u\|_{W^{\frac{1}{q}+\epsilon,q}}^{p-1}\|u\|_{W^{s,q}}$. The last term can be estimated similarly to case 1, but instead we use \eqref{secondcase} and \eqref{secondcase2} to obtain the bound
\begin{equation*}
h^{-\sigma'}\|N'_j\delta_hN_j^{p-m}\|_{L^q}\lesssim_{\epsilon,\sigma'}\|u\|_{W^{\frac{1}{q}+\epsilon,q}}^{p-1}\|u\|_{W^{s,q}}
\end{equation*}
which is uniform in both $j$ and $h$. Since $p-m>0$, $N^{p-m}$ is a H\"older continuous function of $u$. Dominated convergence in the $j$ parameter (holding $h$ fixed) gives the uniform in $h$ bound
\begin{equation*}
h^{-\sigma'}\|N'\delta_hN^{p-m}\|_{L^q}\lesssim_{\epsilon,\sigma'}\|u\|_{W^{\frac{1}{q}+\epsilon,q}}^{p-1}\|u\|_{W^{s,q}},
\end{equation*}
as desired. 
\begin{proof}[Proof of \Cref{new prop}]
We start by proving \eqref{1}, \eqref{2}, and \eqref{3}. In our analysis below, we will show the details for the first two properties as \eqref{3} will follow from similar reasoning and is much easier. Property \eqref{4} will be obtained along the way as a corollary of the estimates used to prove properties \eqref{1} and \eqref{2} -- we will note this when it becomes pertinent. Our analysis will proceed by considering separately the summands in the $l^q$ norm where $h\leq 2^{-j}$ and $h>2^{-j}$. We begin with the first case, which is more difficult.
\subsection{Estimates when $h\leq 2^{-j}$} Here, we define for each $j\geq 0$ and $k\in \mathbb{Z}$ the set $A_k^j:=[k2^{-j},(k+1)2^{-j}]$. If $p-m\leq 0$, the summands in \eqref{1} where $h\leq 2^{-j}$ can be controlled by the sum over $j\geq 0$ and $k\in\mathbb{Z}$ of the terms
\begin{equation*}
    \begin{split}
     &I_{jk}^1:=h^{-q\sigma}\int_{A_k^j} |\partial_xu_j(x)|^q |\partial_xu_j^h(x)|^{q(m-1)}|\delta_hN_j^{p-m}(x)|^qdx,
      \\
      &I_{jk}^2:=h^{-q\sigma}\int_{A_k^j}|\partial_xu_j^h(x)|^{q(m-1)} |u_j^h(x)|^{q(p-m)}|\delta_h\partial_xu_j(x)|^qdx,
      \\
        &I_{jk}^3:=h^{-q\sigma}\int_{A_k^j}|\partial_xu_j(x)|^{q} |u_j(x)|^{q(p-m)}|\delta_h((\partial_xu_j)^{l-1}(\overline{\partial_xu_j})^{m-l})(x)|^qdx,
    \end{split}
\end{equation*}
where  $0\leq l\leq m.$ It will be important to note that since $p-m\leq 0$, we have $m\geq 2$ and therefore the integrand in $I_{jk}^1$ has at least one factor of each of $|\partial_xu_j|^q$ and $|\partial_xu_j^h(x)|^q$. When $0<p-m<1$, it suffices to estimate the sum over $j,k$ of the terms
\begin{equation*}
I_{jk}^{+}:=h^{-q\sigma}\int_{A_k^j} |\partial_xu_j(x)|^{qm}|\delta_hN_j^{p-m}(x)|^qdx
\end{equation*}
which are slightly different than the $I_{jk}^1$ above.
\medskip

\textbf{Estimates for $I_{jk}^1$ and $I_{jk}^+$.} We begin by estimating $I_{jk}^1$ and $I_{jk}^+$. We consider three cases depending on whether $p-m\leq 0$, $0<p-m\leq 1-\frac{1}{q}$ or $1-\frac{1}{q}\leq p-m<1$. 
\medskip

\textbf{Case 1: $p-m\leq 0$.}
Here, we need to estimate $I_{jk}^1$. We note that in this case, we also have $m\geq 2$. We begin by using \Cref{PHB lemma} to estimate $I_{jk}^1$ by
\begin{equation}\label{non est 6}
    \begin{split}
      \int_{A_k^j}|\partial_xu_j(x)|^q |\partial_xu_j^h(x)|^{q(m-1)}\left(|u_j^h(x)|^{-1+\delta q}+|u_j(x)|^{-1+\delta q}\right) \left|\frac{\delta_hu_j(x)}{h}\right|^{q\sigma}\left|\delta_h u_j(x)\right|^{q(p-s_0+\frac{1}{q}-\delta)}dx
    \end{split}
\end{equation}
for some $0<\delta\ll 1$. Below, we will show how to estimate 
\begin{equation}\label{FboundAkj}
    \begin{split}
     \int_{A_k^j}F(x)dx:=  \int_{A_k^j}|\partial_xu_j(x)|^q |\partial_xu_j^h(x)|^{q(m-1)}|u_j^h(x)|^{-1+\delta q}\left|\frac{\delta_hu_j(x)}{h}\right|^{q\sigma}\left|\delta_h u_j(x)\right|^{q(p-s_0+\frac{1}{q}-\delta)}dx
    \end{split}
\end{equation}
 as the other term in \eqref{non est 6} is completely analogous.
\medskip

 How we carry out the above estimate will depend on whether $x+h$ is in $A_k^j$ or lies outside of $A_k^j$. Motivated by this, we denote by $B_k^1$ the portion of $A_k^j$ such that $x\in A_k^j\Rightarrow x+h\in A_k^j$ and let $B_k^2$ denote the portion of $A_k^j$ such that $x\in A_{k}^j\Rightarrow x+h\in A_{k+1}^j$. To simplify notation in the forthcoming estimates, it will be useful to define the following function on $A_k^j$,
\begin{equation*}
M_{j,k}(x):=\begin{cases}
&f_j(x) \hspace{5mm} \text{if}\ |\Delta_{j,k}f|\geq |\Delta_{j,k}g|,
\\
& g_j(x) \hspace{5mm} \text{if}\ |\Delta_{j,k}f|<|\Delta_{j,k}g|,
\end{cases}
\end{equation*}
where we recall that $f:=\Re(u)$ and $g:=\Im(u).$ In other words, $M_{j,k}$ is chosen based on which of $f_j$ or $g_j$ has larger variation on $A_k^j$. We define the function $M_j$ on $\mathbb{R}$ by $M_{j|A_k^j}=M_{j,k}$.
\medskip

We next estimate the portion of \eqref{FboundAkj} corresponding to the region $B_k^1$. We begin by using the definition of $u_j$ to estimate
\begin{equation*}
    \int_{B_k^1} F(x)dx\lesssim 2^{jqs_0}|\Delta_{j,k}M_{j}|^{q\left(p+\frac{1}{q}-\delta\right)}\int_{A_k^j}|M_{j}(x)|^{-1+\delta q}dx.
\end{equation*}
Then, by a linear change of variables (using that $f_j$ and $g_j$ are piecewise-linear), we have
\begin{equation*}
    \int_{B_k^1}F(x)dx\lesssim 2^{j(qs_0-1)}|\Delta_{j,k}M_{j}|^{q(p-\delta)}\int_{M_{j}(A_k^j)}|x|^{-1+q\delta}dx=:J_{jk}^1.
\end{equation*}
To estimate the sum of $J_{jk}^1$ over $j,k$, our strategy will be similar to that of \cite{MR1950719}. We begin by partitioning $\mathbb{R}$ into two sets of disjoint intervals $\mathbb{R}=\bigcup_{n\in \Lambda^f}D_n^f$ (resp.~$\bigcup_{n\in \Lambda^g}D_n^g$), where $D_n^f$ (resp.~$D_n^g$) is a maximal interval on which $\partial_xf_j$ (resp.~$\partial_xg_j$) does not change sign. We let $L_n^f$ (resp.~$L_n^g$) be the set of $k\in \mathbb{Z}$ such that $(k2^{-j},(k+1)2^{-j})$ is included in $D_n^f$ (resp.~$D_n^g$) and let $l_n\in L_n^f$ be such that $|\Delta_{j,l_n}f|=\sup_{k\in L_n^f}|\Delta_{j,k}f|$, with similar definitions for $g$. Using the mutual disjointness of the $f_j(A_k^j)$ on each $D_n^f$ (resp.~$g_j(A_k^j)$ on $D_n^g$)  we have, 
\begin{equation*}
    \sum_{k\in \mathbb{Z}}J^1_{jk}\lesssim 2^{j(qs_0-1)}\|u\|_{L^\infty}^{q\delta}\left(\sum_{n\in \Lambda^f}|\Delta_{j,l_n}f|^{q(p-\delta)}+\sum_{n\in \Lambda^g}|\Delta_{j,l_n}g|^{q(p-\delta)}\right).
\end{equation*}
For either $h=f$ or $h=g$, we have $  \Delta_{j,l_n}h\Delta_{j,l_{n+1}}h\leq 0$. Hence, part \eqref{parta} of \Cref{Oswald} together with simple interpolation inequalities yields
\begin{equation}\label{interp1}
    \sum_{j,k}J^1_{jk}\lesssim_{\epsilon}\|u\|_{W^{\frac{s}{p}+\epsilon,qp}}^{pq}\lesssim \|u\|_{W^{\frac{1}{q}+\epsilon,q}}^{q(p-1)}\|u\|_{W^{s,q}}^q,
\end{equation}
where above and in the sequel we write $\sum_{j,k}$ as shorthand for $\sum_{j\geq 0}\sum_{k\in\mathbb{Z}}$. A similar bound holds for the other term in \eqref{non est 6}.
\medskip

We now turn to the estimate for when $x\in B_k^2$. We begin by observing that since $x+h\in A_{k+1}^j$, we have
\begin{equation*}
    F\leq 2^{jqs_0}|\Delta_{j,k}M_{j}|^q|\Delta_{j,k+1}M_{j}|^{q(m-1)}|M^h_{j,k+1}(x)|^{-1+q\delta}\sup_{l\in \{k,k+1\}}|\Delta_{j,l}M_{j}|^{q\left(p-m+\frac{1}{q}-\delta\right)}.
\end{equation*}
If $|\Delta_{j,k}M_j|\approx|\Delta_{j,k+1}M_{j}|$, we can estimate the integral of $F$ over $B_k^2$ by $J_{jk}^1$. Therefore, without loss of generality, we can assume that we have $|\Delta_{j,k}M_{j}|\geq 4|\Delta_{j,k+1}M_{j}|$ or $|\Delta_{j,k+1}M_{j}|\geq 4|\Delta_{j,k}M_{j}|.$ By the triangle inequality, this gives
\begin{equation*}
    |\Delta_{j,k}M_{j}|+ |\Delta_{j,k+1}M_{j}|\lesssim |\Delta_{j,k}^2u_{j}|,
\end{equation*}
and therefore if $m\geq 2$, we have
\begin{equation*}
    F\lesssim 2^{jqs_0}|\Delta_{j,k}^2u_{j}|^{q\left(p-1+\frac{1}{q}-\delta\right)}|\Delta_{j,k+1}M_{j}|^q|M_{j,k+1}^h(x)|^{-1+q\delta},
\end{equation*}
so that
\begin{equation}\label{Bk2est}
    \begin{split}
        \int_{B_k^2} Fdx&\lesssim  2^{j(qs_0-1)}|\Delta_{j,k}^2u_{j}|^{q\left(p-1+\frac{1}{q}-\delta\right)}|\Delta_{j,k+1}M_{j}|^{q-1}\int_{M_{j}(A_{k+1}^j)}|x|^{-1+q\delta}dx
        \\
        &\lesssim 2^{j(qs_0-1)}\|u\|_{L^\infty}^{q\delta}|\Delta_{j,k}^2u_j|^{q(p-\delta)}.
    \end{split}
\end{equation}
Using part \eqref{parta} of \Cref{Oswald} to sum terms of the form $J_{jk}^1$ arising in the first case and instead part \eqref{partb} of \Cref{Oswald} to sum terms  in the second case, we obtain again
\begin{equation}\label{interp2}
  \sum_{j,k}\int_{B_k^2}Fdx \lesssim_\epsilon \|u\|_{W^{\frac{1}{q}+\epsilon,q}}^{q(p-1)}\|u\|_{W^{s,q}}^q.
\end{equation}
This handles the case $p-m\leq 0$. 
\medskip

\textbf{Case 2: $0<p-m\leq 1-\frac{1}{q}$.} Here, we need to estimate $I_{jk}^+$. We can again split this into an estimate on the sets $B_k^1$ and $B_k^2$. The estimate on $B_k^1$ follows identical reasoning to the previous case. The estimate on $B_k^2$ will require the following bound, which will allow us to replace at least one factor of $\partial_x u_j$ with $\partial_x u_j^h$ in the integrand of $I_{jk}^+$,
\begin{equation}\label{m1case}
\sum_{j,k}K_{jk}^1:=\sum_{j,k}h^{-q\sigma}\int_{B_k^2}|\delta_h\partial_x u_j|^{qm}|\delta_h N_j^{p-m}|^qdx\lesssim_\epsilon \|u\|_{W^{\frac{1}{q}+\epsilon,q}}^{q(p-1)}\|u\|_{W^{s,q}}^q.
\end{equation}
To prove \eqref{m1case}, we use the fact that $|B_k^2|\lesssim h$ as well as $|\delta_h\partial_x u_j|\lesssim 
2^j|\Delta^2_{j,k} u_j|$ and \Cref{PHB} to bound
\begin{equation*}
K_{jk}^1\lesssim 2^{j(q(s-\epsilon)-1)}|\Delta_{j,k}^2u_j|^{q}\|u\|_{C^{\frac{s_0-s+\epsilon}{p-1}}}^{q(p-1)},\hspace{5mm}\text{if}\hspace{5mm}m=1,
\end{equation*}
\begin{equation*}
K_{jk}^1\lesssim 2^{j(q(s_0-p+1)-1)}|\Delta_{j,k}^2u_j|^{q}\|u\|_{C^1}^{q(p-1)},\hspace{5mm}\text{if}\hspace{5mm}m\geq 2.
\end{equation*}
In either case, by \Cref{Oswald} and interpolation, we obtain \eqref{m1case}. Therefore, to estimate the portion of $I_{jk}^+$ on $B_k^2$, it suffices to control the sum
\begin{equation}\label{Ijkplast}
h^{-q\sigma}\sum_{j,k}2^{jmq}\max_{l=k,k+1}|\Delta_{j,l}M_j|^{q(m-1)}\min_{l=k,k+1}|\Delta_{j,l}M_j|^q\int_{B_k^2}|\delta_h N_j^{p-m}(x)|^q dx.
\end{equation}
This can be estimated exactly as in the sequence of bounds \eqref{non est 6}-\eqref{Bk2est}. This handles the case $0<p-m\leq 1-\frac{1}{q}$.
\medskip 

\textbf{Case 3: $1>p-m>1-\frac{1}{q}$.} Here, we are again estimating $I^+_{jk}$. The main difference is that we will instead use \Cref{PHB} with $\beta=1$ to estimate the analogue of the term \eqref{FboundAkj}. Our quantity of interest takes the form
\begin{equation*}\label{non est 6case2}
    \begin{split}
      \int_{A_k^j}F(x)dx\coloneq\int_{A_k^j}|\partial_xu_j(x)|^{qm}|u_j^h(x)|^{q(p-m-1)} \left|\frac{\delta_hu_j(x)}{h}\right|^{q\sigma}\left|\delta_h u_j(x)\right|^{q(1-\sigma)}dx.
    \end{split}
\end{equation*}
Similarly to the previous cases, on $B_k^1$ we can estimate
\begin{equation*}
\begin{split}
\int_{B_k^1}F(x)dx&\lesssim 2^{j(qs_0-1)}|\Delta_{j,k}M_{j}|^{q\kappa}\int_{M_{j,k}(A_k^j)}|x|^{q(p-m-1)}dx=: J_{jk}^1
\end{split}
\end{equation*}
where $\kappa:=m+1-\frac{1}{q}$.
By \Cref{Oswald}, arguing as in cases 1 and 2 and interpolating we have again
\begin{equation}\label{Jk1bound}
\sum_{j,k}J_{jk}^1\lesssim \|u\|^{q(p-\kappa)}_{L^{\infty}}\|u\|^{q\kappa}_{W^{\frac{s_0}{\kappa},q\kappa}}\lesssim_{\epsilon}\|u\|_{W^{\frac{1}{q}+\epsilon,q}}^{q(p-1)}\|u\|_{W^{s,q}}^q.
\end{equation}
On the other hand, by using \eqref{m1case}, we can reduce matters from estimating the portion of $I_{jk}^+$ on $B_k^2$ to estimating a term of the form \eqref{Ijkplast}. This term is estimated by applying \Cref{PHB} with $\beta=1$ as  done directly above and then arguing exactly as in the analysis of the $B_k^2$ integral in the case $p-m\leq 0$.
\medskip

\textbf{Estimates for $I_{jk}^2$ and $I_{jk}^3$.}
We now move on to the estimates for $I_{jk}^2$ and $I_{jk}^3.$ We begin with $I_{jk}^2$. Here, we only need estimates when $p-m\leq 0$. Note that in this case, we necessarily have $m\geq 2$. Our first observation is that, within the set $A_k^j$, $\delta_h(\partial_xu_j)$ is supported on $B_k^2$. Therefore, we obtain
\begin{equation}\label{equ8}
    \begin{split}
        I_{jk}^2&\lesssim 2^{jqm}|\Delta_{j,k}^2 u_j|^q |\Delta_{j,k+1}M_{j}|^{q(m-1)}h^{-q\sigma}\int_{B^2_k}|u_j^h(x)|^{q(p-m)}dx
        \\
        &\lesssim 2^{j(qm-1)}|\Delta_{j,k}^2u_j|^q|\Delta_{j,k+1} M_{j}|^{q(m-1)-1}h^{-q\sigma}\int_{M_{j,k+1}(h+B_k^2)}|x|^{q(p-m)}dx.
    \end{split}
\end{equation}
Thus, upon integrating the last term on the right of \eqref{equ8}, we deduce the bound
\begin{equation*}
    \begin{split}
        I_{jk}^2&\lesssim 2^{j(qs_0-1)}|\Delta_{j,k}^2u_j|^q|\Delta_{j,k+1} M_{j}|^{q(p-1)}\lesssim \|u\|_{C^1}^{q(p-1)}2^{j(q(s_0+1-p)-1)}|\Delta^2_{j,k}u_j|^q.
    \end{split}
\end{equation*}
Since $s_0+1-p<1+\frac{1}{q}$, part \eqref{partb} of \Cref{Oswald} gives
\begin{equation}\label{interp3}
\sum_{j,k}I_{jk}^2\lesssim_{\epsilon} \|u\|_{C^1}^{q(p-1)}\|u\|_{W^{s_0-p+1+\epsilon,q}}^q\lesssim \|u\|_{W^{\frac{1}{q}+\epsilon,q}}^{q(p-1)}\|u\|_{W^{s,q}}^q.
\end{equation}
Note that in the second inequality in \eqref{interp3} we interpolated, using the fact that $s\geq m\geq 2$ and we have $\frac{1}{q}<s_0+1-p+\epsilon<s_0$ and $\frac{1}{q}<1+\frac{1}{q}<s$. Using similar reasoning to handle $I_{jk}^3$, we complete the estimates for $h\leq 2^{-j}.$ 
\subsection{Estimates when $h\geq 2^{-j}$} We now consider the case $h\geq 2^{-j}$. Similarly to the above, we split the analysis into two sub-cases depending on the sign of $p-m$. 
\medskip

\textbf{Case 1: $0<p-m<1$.} Here, we have to estimate the sum over $j\geq\log(h^{-1})$ and $k\in\mathbb{Z}$ of terms  like 
\begin{equation*}\label{sum case 1}
K_{jk}:=h^{-q(\sigma-\epsilon')}2^{jq\epsilon'}\|N'_j\delta_h N_j^{p-m}-N'_{j+1}\delta_h N_{j+1}^{p-m}\|_{L^q(A_{k}^{j+1})}^q.
\end{equation*}
Our starting point in this case is to observe the straightforward estimate 
\begin{equation}\label{sum 2 terms}
\begin{split}
K_{jk}\lesssim &h^{-q(\sigma-\epsilon')}2^{j(q-1+\epsilon'q)}\|\delta_h N_{j+1}^{p-m}\|_{L^{\infty}}^q\|u\|_{C^1}^{q(m-1)}\max_{l=k,k+1}|\Delta_{j+1,l}^2u_{j+1}|^q
\\
+&h^{-q(\sigma-\epsilon')}2^{jq\epsilon'}\|(\partial_x u_{j})^m\delta_h(N_{j+1}^{p-m}-N_{j}^{p-m})\|_{L^q(A_{k}^{j+1})}^q.
\end{split}
\end{equation}
We denote the two terms on the right-hand side of \eqref{sum 2 terms} by $K_{jk}^1$ and $K_{jk}^2$, respectively. If $s>1+\frac{1}{q}$ we can estimate (using implicitly the condition $s\geq p$)
\begin{equation}\label{interpcase2}
K_{jk}^1\lesssim \|u\|_{C^1}^{q(p-1)}2^{j(q(s_0+1-p+\epsilon')-1)}\max_{l=k,k+1}|\Delta_{j+1,l}^2u_{j+1}|^q,
\end{equation}
which when summed in $j,k$ can be controlled by using \Cref{Oswald} and interpolating. If $s\leq 1+\frac{1}{q}$, then we have $m=1$ and we may instead estimate for some $\epsilon>0$,
\begin{equation}\label{smallmcase}
K_{jk}^1\lesssim \|u\|_{C^{\epsilon}}^{q(p-1)}2^{j(q(s-\epsilon)-1)}\max_{l=k,k+1}|\Delta_{j+1,l}^2u_{j+1}|^q,
\end{equation}
which from \Cref{Oswald} and Sobolev embeddings also yields the required bound when summed over $j,k$. For $K_{jk}^2$ we instead proceed by estimating (using translation invariance of the $L^q$ norm and the equivalence $\|f\|_{l_k^qL^q(A^{j+1}_k)}\approx \|f\|_{L^q(\mathbb{R})}$),
\begin{equation*}\label{twosumsK}
\begin{split}
\sum_{k\in\mathbb{Z}} K_{jk}^2\lesssim \sum_{k\in\mathbb{Z}}L_{jk}^1+L_{jk}^2:&=\sum_{k\in\mathbb{Z}}h^{-q\sigma}\|u\|_{C^1}^{q(m-1)}\|N_j^{p-m}-N_{j+1}^{p-m}\|_{L^\infty}^q\|\delta_h\partial_x u_j\|_{L^q(A_k^{j+1})}^q
\\
&+\sum_{k\in\mathbb{Z}}h^{-q\sigma}\|(\partial_x u_j)^m(N_{j+1}^{p-m}-N_j^{p-m})\|_{L^q(A_k^{j+1})}^q.
\end{split}
\end{equation*}
We begin with the estimate for $L_{jk}^1$. If $s>1+\frac{1}{q}$ then by dominated convergence we have the following telescoping estimate:  
\begin{equation*}
\begin{split}
\left(\sum_{k\in\mathbb{Z}}\|\delta_h\partial_xu_j\|_{L^q(A_k^{j+1})}^q\right)^{\frac{1}{q}}\lesssim \|\delta_h\partial_x u\|_{L^q}+\sum_{l\geq j}\left(\sum_{k\in\mathbb{Z}}\|\partial_xu_l-\partial_xu_{l+1}\|_{L^q(A_k^{l+1})}^q\right)^{\frac{1}{q}}.
\end{split}
\end{equation*}
For the first term, we can simply bound
\begin{equation*}
\|\delta_h\partial_x u\|_{L^q}\lesssim h^{s_0-p+\epsilon}\|u\|_{W^{s_0-p+1+\epsilon,q}}.
\end{equation*}
For the latter term, we can estimate each summand by
\begin{equation*}
\|\partial_xu_l-\partial_xu_{l+1}\|_{L^q(A_k^{l+1})}^q\lesssim 2^{lq(1-\frac{1}{q})}|\Delta_{l+1,k}^2u_{l+1}|^q\lesssim 2^{-l\epsilon}h^{q(s_0-p)}2^{l(q(s_0-p+1+\epsilon)-1)}|\Delta_{l+1,k}^2u_{l+1}|^q,
\end{equation*}
which by \Cref{Oswald} and the simple H\"older bound
\begin{equation*}\label{simpleholder}
\|N_{j+1}^{p-m}-N_j^{p-m}\|_{L^{\infty}}\lesssim 2^{-j(p-m)}\|u\|_{C^1}^{p-m}
\end{equation*}
yields
\begin{equation}\label{interpcase2b}
\begin{split}
\sum_{j\geq\log(h^{-1})}\sum_{k\in\mathbb{Z}}L_{jk}^1&\lesssim h^{-q(p-m-\epsilon)}\|u\|_{C^1}^{q(m-1)}\|u\|_{W^{s_0-p+1+\epsilon,q}}^q\sup_{j\geq \log(h^{-1})}2^{qj\epsilon}\|N_{j+1}^{p-m}-N_j^p\|_{L^{\infty}}^q
\\
&\lesssim \|u\|_{C^1}^{q(p-1)}\|u\|_{W^{s_0-p+1+\epsilon,q}}^q.
\end{split}
\end{equation}
We can then interpolate to control the term in the second line above by $\|u\|_{W^{\frac{1}{q}+\epsilon,q}}^{q(p-1)}\|u\|_{W^{s,q}}^q$. 
If $s\leq 1+\frac{1}{q}$ (in which case, $m=1$), one should instead estimate $\|N_j^{p-m}-N_{j+1}^{p-m}\|_{L^{\infty}}$ and $\delta_h\partial_xu_j$ in a similar way to how they were treated in the bound \eqref{smallmcase}. 
\medskip

It remains to estimate the sum over $j\geq \log(h^{-1})$ and $k\in\mathbb{Z}$ of $L_{jk}^2$. First, if $|\partial_x u_j|\geq 4|\partial_x u_{j+1}|$ or $|\partial_x u_{j+1}|\geq 4|\partial_x u_{j}|$ on $A^{j+1}_k$, then we have 
\begin{equation*}
|\partial_x u_{j}|+|\partial_x u_{j+1}|\lesssim |\partial_x u_j-\partial_x u_{j+1}|\lesssim 2^j|\Delta^2_{j+1,k} u_{j+1}|.
\end{equation*}
The contribution of $L_{jk}^2$ in this case can be handled by following similar reasoning to the estimates for $I_{jk}^+$ from before. On the other hand, if $|\partial_x u_j|\approx |\partial_x u_{j+1}|$ then as in the estimate of $I_{jk}^+$, one must consider separately the cases $p-m\leq 1-\frac{1}{q}$ and $p-m>1-\frac{1}{q}$. In the former case, we can use \Cref{PHB} to estimate
\begin{equation*}
L_{jk}^2\lesssim \sup_{l\in\{j,j+1\}}\int_{A_k^{j+1}}|\partial_xu_j|^{mq}|u_l(x)|^{-1+\delta q}\left|\frac{u_{j+1}-u_j}{2^{-j}}\right|^{q\sigma}|u_{j+1}-u_j|^{q(p-s_0+\frac{1}{q}-\delta)}dx.
\end{equation*}
 Using the hypothesis $|\partial_xu_j|\approx |\partial_x u_{j+1}|$ and the bound
\begin{equation*}
|u_{j+1}-u_j|\lesssim 2^{-j}(|\partial_x u_j|+|\partial_x u_{j+1}|)\lesssim 2^{-j}|\partial_x u_{j+1}|
\end{equation*}
as well as the usual change of variables, one can estimate the contributions of these terms as in \eqref{Ijkplast}. In the case $p-m>1-\frac{1}{q}$, one makes a similar adjustment to the analogous case in the estimate for $I_{jk}^+$.  Combining everything together, we obtain
\begin{equation}\label{interp52}
\sum_{j,k}\int_{A_k^{j+1}}K_{jk}^2dx\lesssim \|u\|^{q(p-1)}_{W^{\frac{1}{q}+\epsilon,q}}\|u\|^q_{W^{s,q}}.
\end{equation}
\textbf{Case 2: $p-m\leq 0$.}
Here, we need to estimate the sum over $j> \log(h^{-1})$ of
\begin{equation*}
    I_{h,j}:=h^{-q(\sigma-\epsilon')}2^{jq\epsilon'}\int_\mathbb{R} \left|\delta_h(N_{j+1}-N_j)\right|^qdx\lesssim 2^{jq\sigma}\int_\mathbb{R} |N_{j+1}-N_j|^qdx.
\end{equation*}
It therefore suffices to show that
\begin{equation}\label{L^qconv}
\sum_{j\geq 0}2^{jq\sigma}\int_{\mathbb{R}}|N_{j+1}-N_j|^qdx\lesssim \|u\|_{W^{\frac{1}{q}+\epsilon,q}}^{q(p-1)}\|u\|_{W^{s,q}}^q.
\end{equation}
We proceed by decomposing the above integral with respect to the finer partition at scale $2^{-j-1}$,
\begin{equation*}
2^{jq\sigma}\int_{\mathbb{R}}|N_{j+1}-N_j|^qdx=\sum_{k\in\mathbb{Z}}2^{jq\sigma}\int_{A_{k}^{j+1}}|N_{j+1}-N_j|^qdx.
\end{equation*}
We will then estimate $N_{j+1}-N_j$ in two different ways on $A^{j+1}_k$ depending on the relative sizes of $\partial_xu_j$ and $\partial_xu_{j+1}$ on this interval. 
\medskip

\textbf{Case 2.1: $|\partial_xu_j|\lesssim |\partial_xu_{j+1}|$ on $A_k^{j+1}$.} Here, we estimate
\begin{equation}\label{case21}
\begin{split}
2^{jq\sigma}|N_{j+1}-N_j|^q&\lesssim 2^{jq\sigma}|(\partial_x u_j)^m-(\partial_x u_{j+1})^m|^q|u_{j+1}|^{q(p-m)}
\\
&+2^{jq\sigma}|\partial_x u_j|^{mq}||u_{j+1}|^{p-m}-|u_{j}|^{p-m}|^q
\\
&=:K_{jk}^1+K_{jk}^2.
\end{split}
\end{equation}
Using the hypothesis $|\partial_xu_j|\lesssim |\partial_xu_{j+1}|$ we can control the first term by  
\begin{equation*}
2^{jq\sigma}|(\partial_x u_j)^m-(\partial_x u_{j+1})^m|^q\lesssim 2^{jqs_0}|\Delta^2_{j+1,k}u_{j+1}|^q|\Delta_{j+1,k} M_{j+1}|^{q(m-1)},
\end{equation*}
which, as in the estimate for $I_{jk}^2$, leads to the bound 
\begin{equation}\label{interp4}
\sum_{j\geq 0}\sum_{k\in S_j}\int_{A_k^{j+1}}K_{jk}^1dx\lesssim \|u\|^{q(p-1)}_{W^{\frac{1}{q}+\epsilon,q}}\|u\|^q_{W^{s,q}}
\end{equation}
where $S_j$ is the set of indices $k$ such that $|\partial_x u_j|\lesssim |\partial_x u_{j+1}|$ on $A_k^{j+1}$. On the other hand, similarly to \eqref{FboundAkj}, we can use \Cref{PHB} to estimate the integral of $K_{jk}^2$ on $A_k^{j+1}$ by terms of the form
\begin{equation*}
\begin{split}
\int_{A_k^{j+1}}K_{jk}^2dx&\lesssim \sup_{l=j,j+1}\int_{A_k^{j+1}}|\partial_xu_j|^{mq}|u_l(x)|^{-1+\delta q}\left|\frac{u_{j+1}-u_j}{2^{-j}}\right|^{q\sigma}|u_{j+1}-u_j|^{q(p-s_0+\frac{1}{q}-\delta)}dx.
\end{split}
\end{equation*}
Using the hypothesis $|\partial_xu_j|\lesssim |\partial_x u_{j+1}|$, the simple estimates 
\begin{equation*}
|u_{j+1}-u_j|\lesssim 2^{-j}(|\partial_x u_j|+|\partial_x u_{j+1}|)\lesssim 2^{-j}\|u\|_{C^1}
\end{equation*}
and the usual change of variables, we can control this by
\begin{equation*}
2^{j(q(s_0-p+1-\delta)-1)}\|u\|_{C^1}^{q(p-1-\delta)}\left(|\Delta_{j,\frac{k}{2}}M_j|^q\int_{M_j(A_{\frac{k}{2}}^{j})}|x|^{-1+q\delta}dx+|\Delta_{j+1,k}M_{j+1}|^q\int_{M_{j+1}(A_k^{j+1})}|x|^{-1+q\delta}dx\right).
\end{equation*}
The contribution of these terms in the sum over $j$ can then be handled by applying \Cref{Oswald} and interpolation, resulting in the bound 
\begin{equation}\label{interp5}
\sum_{j,k}\int_{A_k^{j+1}}K_{jk}^2dx\lesssim \|u\|^{q(p-1)}_{W^{\frac{1}{q}+\epsilon,q}}\|u\|^q_{W^{s,q}}.
\end{equation}
\textbf{Case 2.2: $|\partial_x u_{j+1}|\lesssim |\partial_x u_j|$.} This case follows similar reasoning to Case 2.1 except in \eqref{case21}, one estimates $2^{jq\sigma}|N_{j+1}-N_j|^q$ instead by
\begin{equation*}
2^{jq\sigma}|(\partial_x u_j)^m-(\partial_xu_{j+1})^m|^q|u_j|^{q(p-m)}+2^{jq\sigma}|\partial_x u_{j+1}|^{mq}||u_{j+1}^{p-m}-|u_j|^{p-m}|^q=:K_{jk}^1+K_{jk}^2.
\end{equation*}
One then estimates $K_{jk}^1$ and $K_{jk}^2$ by making minor modifications to the above argument. This completes the proofs of parts (i) and (ii) of \Cref{new prop}. Part (iii) can be established by slightly modifying the sequence of estimates above (in fact, the analysis is simpler because it is an estimate in a weaker topology). Finally, we note that property (iv) follows easily from the estimate \eqref{L^qconv}. This finally completes the proof of \Cref{new prop} and thus also the bound \eqref{non est3}. 
\end{proof} 
Later on, the nonlinear estimate proved in this section will suffice in almost every case when establishing our well-posedness results for both the \eqref{NLS} and \eqref{NLH} equations. For \eqref{NLS}, there will be one edge case, which, for technical reasons, will require us to use a mild variation of \eqref{non est3}. By slight abuse, we state it below as a corollary, but really, it is a corollary of very minor modifications of the proof of \eqref{non est3} above (which we will outline below). 
\begin{corollary}\label{nonlinearvariation} Let $p>\frac{3}{2}$, $s\in (2,3)$ and $p\leq s<p+\frac{1}{2}$.  For every $0<\epsilon\ll 1$, there holds
\begin{equation}\label{nonlinearvariationeqn}
\|(N,N')\|_{H^{s-2}}\lesssim_{\epsilon} \|u\|_{H^2}^{\frac{2s-p-1}{2}+\epsilon}\|u\|_{H^1}^{\frac{3p-2s+1}{2}-\epsilon}. 
\end{equation}
\begin{proof}
This corresponds to the situation $m=2$ from the above proof. The proof of \eqref{nonlinearvariationeqn} is identical to the proof of \eqref{non est3}, except for the following minor differences: First, if $p-2\leq 0$ (resp.~$p-2>0$), one replaces the quantities on the right-hand side of \eqref{firstcase}, \eqref{firstcase2} (resp.~\eqref{secondcase}, \eqref{secondcase2}) and \eqref{lowfreqbound} with $\|u\|_{H^2}^{\frac{2s-p-1}{2}+\epsilon}\|u\|_{H^1}^{\frac{3p-2s+1}{2}-\epsilon}$ in the proof of \eqref{nonlinearvariationeqn}. Secondly, instead of interpolating between $W^{\frac{1}{q}+\epsilon,q}$ and $W^{s,q}$, one instead interpolates between $H^1$ and $H^2$ in the bounds \eqref{interp1}, \eqref{interp2}, \eqref{interp3}, \eqref{interp4} and \eqref{interp5} (resp.~\eqref{Ijkplast}, \eqref{Jk1bound}, \eqref{interpcase2}, \eqref{interpcase2b}). Aside from these minor adjustments, the proof is the same as before. 
\end{proof}

\end{corollary}
\section{Sharp well-posedness for the nonlinear heat equation}
In this section, we establish our well-posedness result for \eqref{NLH}. This case will end up being easier to handle than \eqref{NLS} because the characteristic hypersurface corresponding to the heat operator is trivial. We treat this equation first, as some of the nonlinear estimates established here will be reused when establishing our local well-posedness result for \eqref{NLS} in the next section. Our main theorem for \eqref{NLH} is the following. 

\begin{theorem}\label{NLH Well-posedness}
Let $1<p,q<\infty$ and $d\in \mathbb{N}$. Define $s_c=s_c(q):=\frac{d}{q}-\frac{2}{p-1}$. The nonlinear heat equation \eqref{NLH} is locally well-posed in the Sobolev space $W_x^{s,q}(\mathbb{R}^d)$ if $\max\{2,(s_c)_+\}\leq s<p+2+\frac{1}{q}$.
\end{theorem}
\begin{remark}
In \Cref{Ill section} we will complement \Cref{NLH Well-posedness} with the construction of a $C_c^\infty(\mathbb{R}^d)$ data of arbitrarily small norm for which no solution to \eqref{NLH} can be found in the endpoint space $C([0,T]; W_x^{p+2+\frac{1}{q},q}(\mathbb{R}^d))$ for any $T>0$ when $p-1\not\in 2\mathbb{N}$. 
\end{remark}
\begin{remark}
 \Cref{NLH Well-posedness}  excludes the case $q=\infty$ (where $W^{s,\infty}$ is replaced by an appropriate H\"older space) as this case was previously treated (at least for  $1<p<2$) in \cite{cazenave2017non}, is simpler, and has a different behavior at the endpoint. The assumption that $s\geq 2$ in  \Cref{NLH Well-posedness} is for presentation purposes only since our goal is to obtain high regularity well-posedness results and the case $s\in (0,2)$ can be  proven by known methods.
\end{remark}
The proof of  \Cref{NLH Well-posedness} will be divided into four steps. In the first step, we will establish a simple elliptic-type estimate for the linear inhomogeneous heat equation, which reflects the fact that the heat operator's characteristic set is trivial. Then, in \Cref{Subsection2 NLH} we will carry out the requisite nonlinear estimates for the source term $|u|^{p-1}u$ in the \eqref{NLH} equation. We will reuse the estimates here in the case $q=2$ to control a similar term that arises in the analysis of \eqref{NLS} (as the estimates established here will apply equally well to complex-valued functions). We will then complement the resulting nonlinear bounds with difference bounds in a weaker topology in \Cref{DENLH}. Finally, by combining the above three ingredients, well-posedness of \eqref{NLH} is proven in \Cref{Well for NLH} in a straightforward fashion.
\subsection{The linear inhomogeneous heat equation} We begin by establishing the following elementary elliptic-type estimate for the inhomogeneous heat equation. We note that this estimate is almost certainly known, but we provide a simple proof for completeness.
\begin{lemma}\label{Lemma for IHE}
Consider the inhomogeneous heat equation
\begin{equation}\label{IHE}
\begin{cases}
&\partial_tu-\Delta u= f,
\\
&u(0)=u_0.
\end{cases}
\end{equation}
For every $\sigma\in\mathbb{R}$, $0\leq \mu<2$ and $1<q<\infty$ we have 
\begin{equation*}
\|u\|_{L_T^\infty W^{\sigma+\mu,q}_x}\lesssim_{\mu,q} T^{1-\frac{\mu}{2}}\|f\|_{L^\infty_T W_x^{\sigma,q}}+\|u_0\|_{W^{\sigma+\mu,q}_x}.
\end{equation*}
\end{lemma}
\begin{proof}
By applying $D_x^{\sigma}$ to the equation, we may without loss of generality assume that $\sigma=0$. Moreover, by standard estimates for the heat propagator, we have
\begin{equation*}
\|e^{t\Delta}u_0\|_{L_T^{\infty}L_x^q}\lesssim \|u_0\|_{L_x^q}.
\end{equation*}
Hence, we may easily reduce to the case $u_0=0$. For $0<t\leq T$, the solution to \eqref{IHE} may then be written as
\begin{equation*}
u(t,x)=\int_0^t \int_{\mathbb{R}^d} \Phi(x-y,t-s)f(y,s)dyds,
\end{equation*}
where $\Phi(x,t)=\frac{1}{(4\pi t)^\frac{d}{2}}e^{-\frac{|x|^2}{4t}}$. By straightforward computation, we have
\begin{equation*}
\begin{split}
D_x^\mu\Phi(x,t)&= \mathcal{F}^{-1}(|\xi|^\mu e^{-t|\xi|^2})=c_d\frac{1}{ t^{\frac{d+\mu}{2}}}(D_x^\mu e^{- |x|^2/4})\left(\frac{x}{\sqrt{t}}\right)
\end{split}
\end{equation*}
for some dimension-dependent constant $c_d>0$. Hence,
\begin{equation*}
D_x^\mu u(t,x)=c_d\int_0^t\frac{1}{(t-s)^\frac{\mu}{2}} \int_{\mathbb{R}^d} \frac{1}{(t-s)^\frac{d}{2}}(D_x^\mu e^{- |\cdot|^2/4})\left(\frac{x-y}{\sqrt{t-s}}\right)f(y,s)dyds.
\end{equation*}
The desired estimate then follows by Young's inequality and the bound 
\begin{equation*}
\int_0^t\frac{1}{(t-s)^\frac{\mu}{2}}ds\lesssim_{\mu} t^{1-\frac{\mu}{2}}.
\end{equation*}
\end{proof}
We will also need the following simple variation of the above estimate which will be particularly useful for us in the $L^q$-supercritical regime.
\begin{lemma}\label{heatlemma2} For every $\frac{dq}{d+2q}<\tilde{q}\leq q$ there is a $0<\theta\leq 1$ such that for every solution to the inhomogeneous heat equation \eqref{IHE} there holds
\begin{equation*}\label{heatlemmaest}
\|u\|_{L_T^{\infty}L_x^q}\lesssim \|u_0\|_{L_x^q}+ T^{\theta}\|f\|_{L_T^{\infty}L_x^{\tilde{q}}}  .
\end{equation*}
\end{lemma}
\begin{remark}
We note that in the above lemma we are not necessarily assuming that $\tilde{q}\geq 1$.
\end{remark}
\begin{proof}
Again, we may assume without loss of generality that $u_0=0$. Moreover, for $\tilde{q}=q$ this estimate holds with $\theta=1$ by \Cref{Lemma for IHE}. By Marcinkiewicz interpolation theorem (which applies in the range $0<\tilde{q}<1$), it therefore suffices to establish the above estimate when $\tilde{q}=\frac{dq}{d+2q}+\epsilon$ for some small $\epsilon>0$. We consider two cases. If $\frac{dq}{d+2q}\geq 1$ then Sobolev embedding and arguing as in \Cref{Lemma for IHE} (i.e.~applying the multiplier $D_x^{2(1-\epsilon)}$ to the heat kernel in the representation formula) yields  
\begin{equation*}
\|u\|_{L_T^{\infty}L_x^q}\lesssim T^{\epsilon}\|f\|_{L_T^{\infty}L_x^{\frac{dq}{d+2q}+c\epsilon}},
\end{equation*}
for some $c>0$.
On the other hand, when $\frac{dq}{d+2q}<1$ we may begin by noting that for any parameter $r>0$ sufficiently large (i.e.~such that $\frac{dq}{d+2q}>\frac{1}{r})$, we have by Minkowski's inequality, homogeneity and Sobolev embeddings,
\begin{equation*}
\begin{split}
\bigg\|\int_{\mathbb{R}^d} e^{-\frac{|x-y|^2}{4(t-s)}}|f(y)|dy\bigg\|_{L^q_x}&\lesssim \int_{\mathbb{R}^d} \bigg\|e^{-\frac{|x-y|^2}{4r(t-s)}}|f(y)|^\frac{1}{r} \bigg\|_{L_x^{qr}}^rdy
\\
&\lesssim\int_{\mathbb{R}^d}\bigg\| \langle D_x\rangle^{\frac{2}{r}-2c\epsilon}\left( e^{-\frac{|x-y|^2}{4r(t-s)}} \right)|f(y)|^\frac{1}{r}\bigg\|_{L_x^{\frac{drq}{d+2q}+\epsilon}}^rdy
\\
&\lesssim(t-s)^{cr\epsilon-1}\int_{\mathbb{R}^d}\bigg\|\left(\langle D_x\rangle^{\frac{2}{r}-2c\epsilon} e^{-\frac{|\cdot|^2}{4}}\right)^r\left(\frac{x-y}{\sqrt{r(t-s)}}\right)|f(y)| \bigg\|_{L_x^{\frac{dq}{d+2q}+\epsilon}}dy,
\end{split}
\end{equation*}
for some constant $c>0$. Hence, using the representation formula for the solution to the inhomogeneous heat equation, we have
\begin{equation*}
\|u(t)\|_{L^q_x}\lesssim \int_0^t (t-s)^{-d/2} (t-s)^{cr\epsilon-1}\int_{\mathbb{R}^d}\bigg\|\left(\langle D_x\rangle^{\frac{2}{r}-2\epsilon} e^{-\frac{|\cdot|^2}{4}}\right)^r\left(\frac{x-y}{\sqrt{r(t-s)}}\right)|f(y)| \bigg\|_{L_x^{\frac{dq}{d+2q}+\epsilon}}dyds.
\end{equation*}
Now, using the assumption that $\frac{dq}{d+2q}<1$, we may invoke Minkowski's inequality for integrals and then Young's convolution inequality to obtain the desired estimate.
\end{proof}
\subsection{Nonlinear estimate in $W_x^{s-2+\epsilon,q}(\mathbb{R}^d)$}\label{Subsection2 NLH} 
We now establish refined estimates for the nonlinear term in \eqref{NLH}. The gain of two derivatives in the estimate in \Cref{Lemma for IHE} suggests that we will need to be able to control $|u|^{p-1}u$ in $W_x^{s-2+\epsilon,q}$ for some $0<\epsilon\ll 1$. For ease of notation, we will sometimes suppress the $\epsilon$ and instead just write $W_x^{s_+-2,q}$. This way, $s_+$ can grow by a factor of $\epsilon$ from line to line, and we will not have to track the precise constants.
\begin{proposition}\label{farestimateH}
Let $s_c:=\frac{d}{q}-\frac{2}{p-1}$, $s_0\geq\max\{2,(s_c)_+\}$ and $s_0\leq s<p+2+\frac{1}{q}$. For any (real or complex valued) function $u\in W_x^{s,q}(\mathbb{R}^d)$ there holds
\begin{equation*}
\||u|^{p-1}u\|_{W_x^{s_+-2,q}}\lesssim \|u\|_{W_x^{s_0,q}}^{p-1}\|u\|_{W_x^{s,q}}.    
\end{equation*}
\end{proposition}
\begin{proof}
We first observe the following (somewhat crude) estimate, which follows by a Littlewood-Paley decomposition for $|u|^{p-1}u$ and dyadic summation, possibly after slightly enlarging $s_+$:
\begin{equation*}
\||u|^{p-1}u\|_{W_x^{s_+-2,q}}\lesssim \sup_{j\geq 0}\|P_j(|u|^{p-1}u)\|_{W_x^{s_+-2,q}} .   
\end{equation*}
In the nonlinear expression above, it will be convenient to replace $|u|^{p-1}u$ with the corresponding nonlinear term with localized inputs $|u_{<j}|^{p-1}u_{<j}$. For this purpose, we decompose the estimate into two parts: 
\begin{equation*}\label{NL1}
\begin{split}
\|P_j(|u|^{p-1}u)\|_{W_x^{s_+-2,q}}&\leq \|P_j(|u_{<j}|^{p-1}u_{<j})\|_{W_x^{s_+-2,q}}+\|P_j(|u|^{p-1}u-|u_{<j}|^{p-1}u_{<j})\|_{W_x^{s_+-2,q}}
\\
&=:I_j^1+I_j^2,
\end{split}
\end{equation*}
where we write $u_{<j}:=P_{<j+4}u$ and $u_{\geq j}:=P_{\geq j+4}u$ as a shorthand notation. To estimate $I_j^1$, we first observe that by Bernstein, there holds
\begin{equation*}\label{NL2}
I_j^1\lesssim 2^{j(s_+-2)}2^{-j(p+\frac{1}{q}-\epsilon)}\| |u_{<j}|^{p-1}u_{<j}\|_{W_x^{p+\frac{1}{q}-\epsilon,q}}.
\end{equation*}
The nonlinear estimate in \Cref{Nonlinear estimate} and the standard Fubini-type property of Sobolev spaces then yields the bound
\begin{equation*}
\| |u_{<j}|^{p-1}u_{<j}\|_{W_x^{p+\frac{1}{q}-\epsilon,q}}\lesssim \sup_{1\leq i\leq d} \bigg\| \|u_{<j}\|^{p-1}_{W_{x_i}^{\frac{1}{q}+,q}} \|u_{<j}\|_{W_{x_i}^{p+\frac{1}{q}-\epsilon,q}} \bigg\|_{L^q_{\tilde{x}_i}}.
\end{equation*}
The desired estimate for $I_j^1$ now reduces to a simple case analysis. 
\begin{itemize}
\item Case 1: $\frac{d-1}{2}(p-1)\geq q$. If $q\leq \frac{d-1}{2}$ then by Minkowski's inequality and straightforward applications of H\"older's inequality and Sobolev embeddings, we have
\begin{equation*}
\begin{split}
\bigg\| \|u_{<j}\|^{p-1}_{W_{x_i}^{\frac{1}{q}+,q}} \|u_{<j}\|_{W_{x_i}^{p+\frac{1}{q}-\epsilon,q}} \bigg\|_{ L^q_{\tilde{x}_i}}&\lesssim \|u_{<j}\|^{p-1}_{ W^{\frac{1}{q}+,q}_{x_i}L_{\tilde{x}_i}^{\frac{d-1}{2}(p-1)}} \|u_{<j}\|_{ W_x^{p+2+\frac{1}{q}-\epsilon,q}}
\\
&\lesssim \|u_{<j}\|^{p-1}_{W_x^{s_0,q}}\|u_{<j}\|_{ W_x^{p+2+\frac{1}{q}-\epsilon,q}}.
\end{split}
\end{equation*}
On the other hand, if $\frac{d-1}{2}<q$, this forces $p>2$ and hence we have by H\"older and Sobolev again,
\begin{equation*}
\begin{split}
\bigg\| \|u_{<j}\|^{p-1}_{W_{x_i}^{\frac{1}{q}+,q}} \|u_{<j}\|_{W_{x_i}^{p+\frac{1}{q}-\epsilon,q}} \bigg\|_{W^q_{\tilde{x}_i}}&\lesssim \|u_{<j}\|^{p-1}_{W^{\frac{1}{q}+,q}_{x_i}L_{\tilde{x}_i}^{q(p-1)}} \|u_{<j}\|_{ W_x^{p+\frac{1}{q}+\frac{d-1}{q}-\epsilon,q}}
\\
&\lesssim \|u_{<j}\|^{p-1}_{W_x^{s_0,q}}\|u_{<j}\|_{ W_x^{p+2+\frac{1}{q}-\epsilon,q}},
\end{split}
\end{equation*}
where the last line follows by interpolating.
\item Case 2: $\frac{d-1}{2}(p-1)<q$. In this case, we have $s_c<\frac{1}{q}$. Therefore, there are $0<\alpha<1$ and $0<\beta<1$ such that
\begin{equation*}
\begin{split}
\|u_{<j}\|^{p-1}_{W_{x_i}^{\frac{1}{q}+,q}} \|u_{<j}\|_{W_{x_i}^{p+\frac{1}{q}-\epsilon,q}}&\lesssim \|u_{<j}\|^{(p-1)(1-\beta)+(1-\alpha)}_{W_{x_i}^{s_c+,q}} \|u_{<j}\|_{W_{x_i}^{p+2+\frac{1}{q}-\epsilon,q}}^{\alpha+(p-1)\beta}
\\
&=:\|u_{<j}\|^{p-\gamma}_{W_{x_i}^{s_c+,q}} \|u_{<j}\|_{W_{x_i}^{p+2+\frac{1}{q}-\epsilon,q}}^{\gamma}.
\end{split}
\end{equation*}
\end{itemize}
By an algebraic computation, one can see that we have $\gamma:=\alpha+(p-1)\beta\leq 1$. Indeed, this is straightforward to verify directly in one dimension and since the critical exponent $s_c$ is an increasing function of $d$, it follows that $\gamma$ must be a decreasing function of the dimension.  Using this observation, we can apply Minkowski's inequality (noting that $\frac{p-\gamma}{1-\gamma}q\geq q$), H\"older's inequality, Sobolev embeddings and interpolate again to obtain (through straightforward, albeit slightly tedious computation),
\begin{equation*}
\begin{split}
\bigg\| \|u_{<j}\|^{p-1}_{W_{x_i}^{\frac{1}{q}+,q}} \|u_{<j}\|_{W_{x_i}^{p+\frac{1}{q}-\epsilon,q}} \bigg\|_{L^q_{\tilde{x}_i}}&\lesssim \|u_{<j}\|_{W_{x_i}^{s_c+,q}L_{\tilde{x}_i}^{q\frac{p-\gamma}{1-\gamma}}}^{p-\gamma}\|u_{<j}\|_{W_x^{p+2+\frac{1}{q}-\epsilon,q}}^{\gamma}
\\
&\lesssim \|u_{<j}\|_{W_{x}^{s_c+,q}}^{p-1}\|u_{<j}\|_{W_x^{p+2+\frac{1}{q}-\epsilon,q}}.
\end{split}
\end{equation*}
Combining the above estimates and using the restriction $s<p+2+\frac{1}{q}$ and $s_0>s_c$ as well as Bernstein's inequality yields
\begin{equation*}
I_j^1\lesssim \|u\|_{W_x^{s_{0},q}}^{p-1}\|u\|_{W_x^{s,q}}.    
\end{equation*}
To control $I_j^2$, we use  Bernstein's inequality and the fundamental theorem of calculus to write
\begin{equation}\label{highmod3}
\begin{split}
2^{-j(s_+-2)}I_j^2\lesssim\sup_{\tau\in [0,1]}\|P_j(|u_{\tau}|^{p-1}u_{\geq j})\|_{L_x^q}+\sup_{\tau\in [0,1]}\|P_j(|u_{\tau}|^{p-3}u_{\tau}\Re(\overline{u}_{\tau}u_{\geq j}))\|_{L_x^q}.
\end{split}
\end{equation}
Here, we use the notation
\begin{equation*}
u_{\tau}:=\tau u+(1-\tau)u_{<j},\hspace{5mm} 0\leq \tau\leq 1.
\end{equation*}
In what follows, we show how to estimate the first of the above two terms; the latter follows from completely analogous reasoning. By considering the Fourier support of $|u_{\tau}|^{p-1}u_{\geq j}$ (recall the convention for $u_{\geq j}$ above), we may write
\begin{equation*}
\|P_j(|u_{\tau}|^{p-1}u_{\geq j})\|_{L_x^q}\leq \sum_{k\gtrsim j}\|\tilde{P}_k|u_{\tau}|^{p-1}P_ku\|_{L_x^q}\lesssim \sup_{k\gtrsim j}2^{k\epsilon}\|\tilde{P}_k|u_{\tau}|^{p-1}P_ku\|_{L_x^q}
\end{equation*}
for some slightly fattened projector $\tilde{P}_k$. By Bernstein, we then have
\begin{equation}\label{highmod2H}
\|\tilde{P}_k|u_{\tau}|^{p-1}P_ku\|_{L_x^q}\lesssim 2^{-ks}\|\tilde{P}_k|u_{\tau}|^{p-1}\|_{L_x^{\infty}}\|u\|_{W_x^{s,q}}.  
\end{equation}
To control the right-hand side of \eqref{highmod2H}, we consider several cases. 
\begin{itemize}
\item Case 1: $\frac{d(p-1)}{q}\geq 2$. This corresponds to the case where $s_c\geq 0$. Here, Bernstein and Sobolev embedding yields
\begin{equation*}
2^{-k(2-2\epsilon)}\|\tilde{P}_k|u_{\tau}|^{p-1}\|_{L_x^{\infty}}\lesssim \|u\|_{L_x^{\frac{d(p-1)}{2}+c\epsilon}}^{p-1} \lesssim\|u\|_{W^{s_0,q}_x}^{p-1}.
\end{equation*}
\item Case 2: $\frac{d(p-1)}{q}<2$. This corresponds to the situation $s_c<0$. If we further have $d\geq 2$, it follows that $\frac{q}{p-1}\geq 1$ and we have
\begin{equation*}
2^{-k(2-2\epsilon)}\|\tilde{P}_k|u_{\tau}|^{p-1}\|_{L_x^{\infty}}\lesssim \||u_{\tau}|^{p-1}\|_{L_x^{\frac{q}{p-1}}}\lesssim \|u\|_{L_x^q}^{p-1}\lesssim\|u\|_{W_x^{s_0,q}}^{p-1}.    
\end{equation*}
If $d=1$, then for each $q$ we have the embedding $W^{1,q}\subset L^{\infty}$. This gives the crude bound
\begin{equation*}
2^{-k(2-2\epsilon)}\|\tilde{P}_k|u_{\tau}|^{p-1}\|_{L_x^{\infty}}\lesssim \|u\|_{L_t^{\infty}}^{p-1}\lesssim \|u\|^{p-1}_{W^{1,q}}\lesssim \|u\|_{W_x^{s_0,q}}^{p-1}.
\end{equation*}
\end{itemize}
Combining the above estimates and carrying out a very similar analysis for the second term on the right-hand side of \eqref{highmod3}, we arrive at the inequality
\begin{equation*}
I_j^2\lesssim \|u\|_{W_x^{s_{0},q}}^{p-1}\|u\|_{W_x^{s,q}}    
\end{equation*}
as required. This concludes the proof of \Cref{farestimateH}.
\end{proof}
In our ill-posedness results (particularly at the critical endpoint), we will need the following slight variation of the above estimate, which allows for $s_0$ to be equal to the critical exponent (as long as the left-hand side of the inequality is measured at slightly lower regularity). We state it as a corollary though, strictly speaking, it follows from making extremely minor modifications to the above estimates.
\begin{corollary}\label{heatestmodified} Let $s_c=\frac{d}{q}-\frac{2}{p-1}$, $s_0\geq\max\{2,s_c\}$ and $s_0\leq s<p+2+\frac{1}{q}$. For any (real or complex valued) function $u\in W_x^{s,q}(\mathbb{R}^d)$, there holds
\begin{equation*}
\||u|^{p-1}u\|_{W_x^{s_--2,q}}\lesssim \|u\|_{W_x^{s_0,q}}^{p-1}\|u\|_{W_x^{s,q}}.
\end{equation*}
\end{corollary}
\begin{proof}
The argument is virtually identical to the above, except one has to simply balance the H\"older and Sobolev exponents slightly differently. 
\end{proof}
\subsection{Difference bounds for the nonlinear heat equation} \label{DENLH} In this subsection, we establish low regularity difference bounds for solutions to \eqref{NLH}.
\begin{lemma}\label{heatdiff}
Let $u_1$ and $u_2$ be two $C([0,T]; W_x^{s,q}(\mathbb{R}^d))$ solutions to the nonlinear heat equation \eqref{NLH} with $s\geq s_0\geq \max\{0,(s_c)_+\}$. If $T$ is sufficiently small depending on the size of $u_1$ and $u_2$ in $L^\infty_TW^{s_0,q}_x(\mathbb{R}^d)$, there holds
\begin{equation*}
\|u_1-u_2\|_{L_T^{\infty}L_x^q}\lesssim \|u_1(0)-u_2(0)\|_{L_x^q}.
\end{equation*}
\end{lemma}
\begin{proof}
We have two cases. If $s_c\geq 0$, applying \Cref{heatlemma2} with $\tilde{q}=\frac{dq}{d+2q}+\epsilon$ gives us the bound
\begin{equation*}
\|u_1-u_2\|_{L_T^{\infty}L_x^q}\lesssim \|u_1(0)-u_2(0)\|_{L_x^q}+T^{\delta}\||u_1|^{p-1}u_1-|u_2|^{p-1}u_2\|_{L_T^{\infty}L_x^{\frac{dq}{d+2q}+\epsilon}}
\end{equation*}
for some $\delta>0$. The latter term can be estimated by
\begin{equation*}
\begin{split}
T^{\delta}\||u_1|^{p-1}u_1-|u_2|^{p-1}u_2\|_{L_T^{\infty}L_x^{\frac{dq}{d+2q}+\epsilon}}&\lesssim T^{\delta}\|(u_1,u_2)\|^{p-1}_{L_T^{\infty}L_x^{\frac{d}{2}(p-1)+c\epsilon}}\|u_1-u_2\|_{L_T^{\infty}L_x^q}
\\
&\lesssim T^{\delta}\|(u_1,u_2)\|^{p-1}_{L_T^{\infty}{W_x^{s_c+,q}}}\|u_1-u_2\|_{L_T^{\infty}L_x^q}.
\end{split}
\end{equation*}
On the other hand, if $s_c<0$, then it is easy to compute that we have $\frac{dq}{d+2q}<\frac{q}{p}< q$. Applying \Cref{heatlemma2} then yields (for some possibly different $\delta$)
\begin{equation*}
\|u_1-u_2\|_{L_T^{\infty}L_x^q}\lesssim \|u_1(0)-u_2(0)\|_{L_x^q}+T^{\delta}\|(u_1,u_2)\|_{L_T^{\infty}L_x^q}^{p-1}\|u_1-u_2\|_{L_T^{\infty}L_x^q}.
\end{equation*}
In either case, taking $T>0$ small enough yields the desired difference estimate. 
\end{proof}
\subsection{Proof of well-posedness for the nonlinear heat equation}\label{Well for NLH} Let $u_0\in W_x^{s,q}(\mathbb{R}^d)$ where $2+p+\frac{1}{q}>s\geq \max\{2,(s_c)_+\}$. By the contraction mapping theorem and standard estimates for the heat operator, it is straightforward to obtain a unique solution $u_j\in C([0,T_j];W_x^{s,q}(\mathbb{R}^d))$ to the equation
\begin{equation*}
\begin{cases}
&(\partial_t-\Delta)u_j=P_{<j}(|u_j|^{p-1}u_j),
\\
&u_j(0)=u_0.
\end{cases}
\end{equation*}
Here, $P_{<j}$ is the standard Littlewood-Paley projector at frequency $2^j$ and the existence time $T_j$ a priori depends on $j$. However, the uniform bounds from the previous subsections ensure that  we can extend $u_j$ to a solution in $C([0,T];W_x^{s,q}(\mathbb{R}^d))$ on a time interval $[0,T]$ which is uniform in $j$ and satisfies the uniform bound
\begin{equation*}
\|u_j\|_{C([0,T];W_x^{s,q})}\lesssim \|u_0\|_{W_x^{s,q}}.
\end{equation*}
Our goal is to show that $u_j$ converges to a unique solution $u\in C([0,T]; W_x^{s,q}(\mathbb{R}^d))$ satisfying the nonlinear heat equation \eqref{NLH}. The solution $u$ is obtained in two simple steps. First, we note the identity
\begin{equation*}
u_j-u_k=\int_{0}^{t}e^{(t-s)\Delta}P_{k\leq\cdot <j}(|u_j|^{p-1}u_j)+e^{(t-s)\Delta}P_{<k}(|u_j|^{p-1}u_j-|u_k|^{p-1}u_k)ds,\hspace{5mm}j>k\geq 0.
\end{equation*}
This implies that $u_j-u_k$ is smooth. Moreover, the bounds in \Cref{Lemma for IHE} and \Cref{farestimateH} yield the uniform estimate in the stronger topology $C([0,T];W_x^{s+\epsilon,q}(\mathbb{R}^d))$,
\begin{equation*}
\|u_j-u_k\|_{C([0,T];W_x^{s+\epsilon,q})}\lesssim \|u_0\|_{W_x^{s,q}},
\end{equation*}
for some sufficiently small $\epsilon>0$. Next, we note that the difference bounds in $L^q$ from \Cref{heatdiff} and Bernstein's inequality can be easily shown to imply that
\begin{equation*}\label{heatdiffbound}
\|u_j-u_k\|_{C([0,T];L_x^q)}\lesssim 2^{-ks}\|u_0\|_{W_x^{s,q}}.
\end{equation*}
Interpolation establishes convergence in $C([0,T];W_x^{s,q}(\mathbb{R}^d))$. This gives existence. Uniqueness follows immediately from the difference bounds in \Cref{heatdiff}. For continuous dependence, if we have a sequence of data $u_0^n\to u_0$ in $W_x^{s,q}(\mathbb{R}^d)$, the above interpolation argument is easily adapted to show that the corresponding solutions $u^n\to u$ in the strong topology $C([0,T];W_x^{s,q}(\mathbb{R}^d))$.

\section{Sharp well-posedness for the nonlinear Schr\"odinger equation}
We now consider the nonlinear Schr\"odinger equation \eqref{NLS} where $u:[-T,T]\times\mathbb{R}^d\to \mathbb{C}$ and $p>1$. We will take the positive sign in the nonlinearity in the sequel for simplicity, but the following analysis applies equally well when the nonlinearity is $-|u|^{p-1}u$. Our main local well-posedness result is the following.
\begin{theorem}\label{NLSLWP}
Let $p>1$ and $\max\{0,s_c\}<s<\min\{2p+1,p+\frac{5}{2}\}$. Then \eqref{NLS} is locally well-posed in $H_x^s(\mathbb{R}^d)$.    
\end{theorem}
\begin{remark}
We will later complement \Cref{NLSLWP}  with a corresponding non-existence result when $s\geq p+\frac{5}{2}$. In particular, for $p\geq\frac{3}{2}$ the upper bound on $s$ in \Cref{NLSLWP} is sharp. For general $p>1$, \Cref{NLSLWP} ensures that one can always construct at least $H^3(\mathbb{R}^d)$ (or better) solutions to \eqref{NLS}, provided that there are no  obstructions coming from scaling. This constitutes a very large improvement over previously known results.   
\end{remark}

The case $s\leq 1$ in \Cref{NLSLWP} is relatively standard and follows from a straightforward application of Strichartz estimates. Therefore, in the analysis below, we will often assume that $s>1$. In fact, we will soon reduce to the case $s>2$, as the other cases are known.

\subsection{The time-truncated inhomogeneous Schr\"odinger equation}  In our analysis of the well-posedness problem for \eqref{NLS}, we will often need to use fractional time derivatives to obtain the best possible estimates (in terms of allowed regularity for the solution). For this reason, it is convenient to truncate the nonlinearity in the equation by some time-dependent cutoff. We can then extend the solution outside of the support of this cutoff by applying the linear Schr\"odinger propagator, which will allow us to work with a global-in-time solution when proving a priori estimates. To set notation, we let $\eta$ be a smooth, purely time-dependent bump function supported on $[-2,2]$ with $\eta=1$ for $|t|\leq 1$. We then fix a sufficiently small parameter $0<T\ll 1$ to be chosen (which we will eventually think of as our local existence time for the original equation) and define the rescaled function $\eta_T(t):=\eta(T^{-1}t)$. 
\medskip

Our next objective is to describe the precise time truncation procedure (which is loosely inspired by \cite{MR3917711}) as well as our general strategy for obtaining a priori bounds. We will also outline some of the more subtle technical points that will motivate our choices below. Unlike in the case of the nonlinear heat equation, the Schr\"odinger equation has a non-trivial characteristic surface. Therefore, it is natural to consider the multiplier decomposition
\begin{equation*}
1=Q_f+Q_n
\end{equation*}
where 
\begin{equation*}
Q_n:=\sum_{j\geq 0}\tilde{S}_{2j}P_j,\hspace{5mm}Q_f:=\sum_{j\geq 0}S_{<2j-4}P_j+P_{<j-4}S_{2j}.
\end{equation*}
Here, we think of $Q_n$ as localizing the spacetime Fourier support of a function $u$ to the region $|\tau|\approx |\xi|^2$ near the characteristic hypersurface, while $Q_f$ projects onto the ``far" region, away from the characteristic frequencies $|\tau|\approx |\xi|^2$. In the analysis below, we will also make use of the convention 
\begin{equation*}
(\partial_t^{-1}f)(t):=\int_{0}^{t}f(s)ds,\hspace{5mm}t\in\mathbb{R}.
\end{equation*}
As usual, Strichartz spaces will be denoted by $S$ and dual Strichartz spaces by $S'$; the exact choice of Strichartz space will vary depending on the range of parameters that we are considering. We will also use the notation $S_+$, (resp.~$S'_+$) to denote a Strichartz (resp.~dual Strichartz) space obtained by slightly increasing the time integrability parameter. For instance, if $(q,p)$ (resp.~$(q',p')$) is a Strichartz admissible (resp.~dual admissible) pair, then $S_+$ (resp.~$S'_+$) could contain spaces of the form $L_t^{q+\epsilon}L_x^p$ (resp.~$L_t^{q'+\epsilon}L_x^{p'}$). The utility of these latter spaces stems from the need to apply H\"older in time to obtain a non-trivial factor of $T$ to use as a smallness parameter to close our estimates. Our main technical input, which will reduce our analysis to proving a suitable nonlinear estimate for the nonlinearity in \eqref{NLS}, is given by the following proposition; some comments on the motivation for the choices made in this proposition will appear below.
\begin{proposition}[Estimates for the time-truncated inhomogeneous Schr\"odinger equation]\label{truncationestimates}
Given a Schwartz function $F:\mathbb{R}_t\times \mathbb{R}^d_x\to \mathbb{C}$ define 
\begin{equation*}
F_T(t,x):=\eta_{T_1}\partial_t^{-1}(\eta_T\partial_t F)+\eta_{1}(t)F(0,x),\hspace{5mm}T\ll T_1\ll 1.
\end{equation*}
Then $F_T$ and $u$ defined by
\begin{equation*}
u:=e^{-it\Delta}u_0-i\int_{0}^{t}e^{-i(t-s)\Delta}F_T(s,x)ds=:u_h+u_{in}
\end{equation*}
satisfy the following properties:
\begin{enumerate}
\item (Extension property). $F=F_T$ on $[-T,T]$.
\item (Elliptic estimate for $Q_fu$). Let $1< p,q\leq \infty$ and $1\geq \sigma> \frac{1}{q}$ (or $\sigma=0$ when $q=\infty$). For every $\mu>-2$ there is a $\delta>0$ such that 
\begin{equation}\label{farbound}
\|Q_fu\|_{W_t^{\sigma,q}W_x^{s,p}}\lesssim T_1^{\delta}\|F\|_{W_t^{\sigma,q}W_x^{s+\mu,p}}+\|F(0)\|_{W_x^{s+\mu,p}}.
\end{equation}
\item (Bounds for $2< s< 4).$ For every Strichartz space $S=L_t^qL_x^p$ excluding $(q,p)=(4,\infty)$ when $d=1$  and near dual Strichartz space $S'_+$, there holds
\begin{equation*}\label{spacebound}\|u\|_{L_t^{\infty}H_x^s}+\| D_t^{\frac{s}{2}} u\|_{L_t^qL_x^p}\lesssim \|u_0\|_{H_x^s}+T_1^{\delta}\|\langle D_t\rangle^{\frac{s}{2}}F\|_{S_+'}+T_1^{\delta}\|F\|_{L_t^{\infty}H_x^{s-2+}}+\|F(0)\|_{H_x^{s-2+}}.
\end{equation*}
\item (Bounds for $4\leq s<6$). Let $F_{tt}=G+H$ be any partition. For every Strichartz pair $(q,p)$ excluding $(q,p)=(4,\infty)$ and $j=1,2$, there holds
\begin{equation*}\label{spacebound2} 
\begin{split}
\|u\|_{L_t^{\infty}H_x^s}+\|\partial_t^jD_x^{s-2j}u\|_{L_t^qL_x^p}\lesssim &T_1^{\delta}\|F\|_{L_t^{\infty}H_x^{s-2+}}+\|F(0)\|_{H_x^{s-2+}}+T_1^{\delta}\|\partial_tD_x^{s-4+}F\|_{L_t^{q+}L_x^p\cap L_t^{\infty}L_x^2}
\\
+&T_1^{\delta}\|(D_x^{s-4}G,D_t^{\frac{s}{2}-2}H)\|_{{S}'_+}+\|u_0\|_{H_x^s}.
\end{split}
\end{equation*}
\item (Bounds for $s\geq 6$). Let $F_{tt}=G+H$ be any partition. For every Strichartz pair $(q,p)$ excluding $(q,p)=(4,\infty)$ and $j=1,2$, there holds
\begin{equation*}\label{spacebound3} 
\begin{split}
\|u\|_{L_t^{\infty}H_x^s}+\|\partial_t^{j}D_x^{s-2j}u\|_{L_t^qL_x^p}\lesssim &T_1^{\delta}\|F\|_{L_t^{\infty}H_x^{s-2+}}+\|F(0)\|_{H_x^{s-2+}}+T_1^{\delta}\|\partial_tD_x^{s-4+}F\|_{L_t^{q+}L_x^p\cap L_t^{\infty}L_x^2}
\\
+&T_1^{\delta}\|(D_x^{s-4}G,D_x^{s-6}\partial_tH)\|_{{S}'_+}+\|u_0\|_{H_x^s}.
\end{split}
\end{equation*}
\end{enumerate}
\end{proposition}
\begin{remark}
Here, we comment on the seemingly obscure choice of truncation for the inhomogeneous term $F$. The more obvious truncation, which would be given simply by
\begin{equation*}
\tilde{F}_T=\eta_TF,
\end{equation*}
is only suitable in the range $0<s\leq 2$. This is because the time-truncation cutoff is localized at frequency essentially $T^{-1}$, which makes it seemingly impossible to estimate more than a single time derivative applied to $\tilde{F}_T$ in, for instance, $L^1$ in time, as time derivatives hitting the cutoff $\eta_T$ will generate factors of $\frac{1}{T}$. On the other hand, the truncation used above works for estimating up to two time derivatives (or correspondingly up to four spatial derivatives) before running into the same type of problem. This truncation is inspired by the choice in \cite{MR3917711}, which considered Sobolev exponents between two and four in their analysis. One might wonder why we do not extend this truncation to a ``higher order" Taylor type expansion of the form
\begin{equation*}
\tilde{F}_T:=\eta_{T_1}\partial_t^{-k}(\eta_T\partial_t^kF)+\eta_1(t)\sum_{0\leq j\leq k-1}\frac{t^j}{j!}(\partial_t^jF)(0).
\end{equation*}
Such a truncation is ideal for estimating a large number of time derivatives of $u$ but becomes problematic for estimating the corresponding number of spatial derivatives in some of the relevant Strichartz norms. It turns out that our choice of truncation strikes the right balance between these two extremes and is sufficient for obtaining the full range of a priori estimates for \eqref{NLS} that will be necessary to prove our main results. 
\end{remark}
Now, we turn to the proof of \Cref{truncationestimates}.
\begin{proof}
Property (i) is immediate from the fundamental theorem of calculus. We now turn to property (ii). Since the map $F\mapsto F_T$ commutes with spatial derivatives, we may assume without loss of generality that $s=0$. Moreover, by dyadic summation and slightly enlarging $\mu$ (which is allowed because of the strict inequality $\mu>-2$), it suffices to obtain the estimate \eqref{farbound} with $Q_fu$ replaced by $Q_ju:=(S_{<2j-4}P_j+P_{<j-4}S_{2j})u$. In this case, by Young's inequality and then Bernstein in time, we obtain (after possibly slightly enlarging $\mu$ and noting that $Q_je^{-it\Delta}u_0=0$) the elliptic estimate
\begin{equation}\label{mainellipticterm}
\begin{split}
\|Q_ju\|_{W_t^{\sigma,q}L_x^p}&\lesssim \|F_T\|_{W_t^{\sigma_-,q_-}W_x^{\mu,p}}.
\end{split}
\end{equation}
Here, $\sigma_-<\sigma$ and $q_-<q$ denote parameters which are smaller than, but arbitrarily close to, $q$ and $\sigma$. Before estimating the above term, we record the following simple interpolation estimate for later use
\begin{equation}\label{interpolationestimate}
\|\partial_t^{-1}(\eta_T\partial_tF)\|_{L_t^{\infty}W_x^{\mu,p}}\lesssim T^{\theta-\epsilon}\|F\|_{W_t^{\theta,\infty}W_x^{\mu,p}},\hspace{5mm} 0<\theta<1,\hspace{5mm}0<\epsilon\ll 1.
\end{equation}
This follows by crudely interpolating the two estimates
\begin{equation}\label{endpoints}
\|\partial_t^{-1}(\eta_T\partial_tF)\|_{{L_t^{\infty}W_x^{\mu,p}}}\lesssim \|F\|_{L_t^{\infty}W_x^{\mu,p}},\hspace{5mm}\|\partial_t^{-1}(\eta_T\partial_tF)\|_{{L_t^{\infty}W_x^{\mu,p}}}\lesssim T\|F\|_{W_t^{1,\infty}W_x^{\mu,p}}.
\end{equation}
Now, we estimate the  term on the right-hand side of \eqref{mainellipticterm} by the right-hand side of \eqref{farbound}. We first dispense with the easy case $(\sigma,q)=(0,\infty)$, which simply follows from H\"older in time and the first estimate in \eqref{endpoints}. Now, we turn to the case $(\sigma,q)\neq (0,\infty)$. Since $F(0)$ is time-independent, it is clear that $\eta_1F(0)$ can be controlled by the right-hand side of \eqref{farbound}. To control the time-dependent part, we first use the fractional Leibniz rule and then \eqref{interpolationestimate} to dispense with the $\eta_{T_1}$ localization as well as low time frequencies for $\partial_t^{-1}(\eta_T\partial_tf)$, 
\begin{equation*}\label{twotermsloc}
\begin{split}
\|\eta_{T_1}\partial_t^{-1}(\eta_T\partial_tF)\|_{W_t^{\sigma_-,q_-}W_x^{\mu,p}}&\lesssim T_1^{\frac{1}{q}-\sigma+\delta}\|\partial_t^{-1}(\eta_T\partial_tF)\|_{L_t^{\infty}W_x^{\mu,p}}+T_1^{\delta}\|S_{>0}\partial_t^{-1}(\eta_T\partial_tF)\|_{W_t^{\sigma_-,q}W_x^{\mu,p}}
\\
&\lesssim T_1^{\delta}\|F\|_{W_t^{\sigma,q}W_x^{\mu,p}}+T_1^{\delta}\|S_{>0}\partial_t^{-1}(\eta_T\partial_tF)\|_{W_t^{\sigma_-,q}W_x^{\mu,p}}.
\end{split}
\end{equation*}
To estimate the latter term on the right-hand side above, we first observe that by slightly enlarging $\sigma_-$, it suffices to estimate $S_k\partial_t^{-1}(\eta_T\partial_tF)$ for each $k>0$. Here, by Bernstein, we have
\begin{equation*}
\|S_k\partial_t^{-1}(\eta_T\partial_tF)\|_{W_t^{\sigma_-,q}W_x^{\mu,p}}\lesssim 2^{k(\sigma-1)}\|S_k(\eta_T\partial_tF)\|_{L_t^qW_x^{\mu,p}}.
\end{equation*}
We then use the Littlewood-Paley trichotomy to split $S_k(\eta_T\partial_tF)$ into low-high, high-low and high-high interactions, respectively,
\begin{equation*}
S_k(\eta_T\partial_tF)=S_k(S_{<k-4}\eta_T\partial_t\tilde{S}_kF)+S_k(\tilde{S}_k\eta_TS_{<k-4}\partial_tF)+\sum_{l,m\gtrsim k,\, |l-m|\lesssim 1}S_k(S_l\eta_T S_m\partial_tF).
\end{equation*}
We now proceed with a case analysis.
\begin{itemize}
\item (Low-high interactions and high-low interactions). We place the cutoff function in $L_t^\infty$ and then using Bernstein's inequality to obtain
\begin{equation*}
2^{k(\sigma-1)}\left(\|S_k(S_{<k-4}\eta_T\partial_t\tilde{S}_kF)\|_{L_t^qW_x^{\mu,p}}+\|S_k(\tilde{S}_k\eta_TS_{<k-4}\partial_tF)\|_{L_t^qW_x^{\mu,p}}\right)\lesssim \|F\|_{W_t^{\sigma,q}W_x^{\mu,p}}.
\end{equation*}
\item (High-high interactions). Here we use Bernstein and H\"older to obtain
\begin{equation*}
\begin{split}
2^{k(\sigma-1)}\|S_k(S_l\eta_TS_m\partial_tF)\|_{L_t^qW_x^{\mu,p}}&\lesssim 2^{k(\sigma-\frac{1}{q})}\|S_k(S_l\eta_TS_m\partial_tF)\|_{L_t^1W_x^{\mu,p}}
\\
&\lesssim 2^{(k-l)(\sigma-\frac{1}{q})}\|S_lD_t^{1-\frac{1}{q}}\eta_T\|_{L_t^{\frac{q}{q-1}}}\|F\|_{W_t^{\sigma,q}W_x^{\mu,p}}
\\
&\lesssim 2^{(k-l)(\sigma-\frac{1}{q})}\|F\|_{W_t^{\sigma,q}W_x^{\mu,p}}.
\end{split}
\end{equation*}
One then deduces a $W_t^{\sigma,q}W_x^{\mu,p}$ estimate for the full sum of high-high interactions by summing over the appropriate indices and using the hypothesis $\sigma>\frac{1}{q}$.
\end{itemize}
Now, we turn to the proof of property (iii). We first control $u$ in $L_t^{\infty}H_x^s$. Clearly, it suffices to estimate the slightly stronger quantity $\|P_ju\|_{l_j^2L_t^{\infty}H_x^s}$ (which we will need for later use). Thanks to (ii), it further suffices to estimate $Q_nP_ju$ in $l_j^2L_t^{\infty}H_x^s$. For this, we split 
\begin{equation*}
Q_nP_ju=(Q_nP_ju)^{hom}+(Q_nP_ju)^{in}:=(e^{-it\Delta}P_ju_0-e^{-it\Delta}(Q_fP_ju)(0))-i\int_{0}^{t}e^{-i(t-s)\Delta}Q_nP_jF_Tds.
\end{equation*}
By the triangle inequality and the conservation of the $H_x^s$ norms for the linear evolution, we have
\begin{equation*}
\|(Q_nP_ju)^{hom}\|_{l^2_jL_t^\infty H_x^s}\leq \|u_0\|_{H_x^s}+\|Q_fu\|_{L_t^{\infty}H_x^{s+\epsilon}}.
\end{equation*}
 On the other hand, by orthogonality of $P_j$, Strichartz and  Bernstein, there holds
\begin{equation}\label{Qinest}
\|(Q_nP_ju)^{in}\|_{l^2_jL_t^{\infty}H_x^s}\lesssim \|(S_{2j}P_ju)^{in}\|_{l_{j}^2L_t^{\infty}H_x^s}\lesssim \|D_t^{\frac{s}{2}}P_jF_T\|_{l_j^2S'}\lesssim \|D_t^{\frac{s}{2}}F_T\|_{S'}.
\end{equation}
In the above, $S'$ is any dual Strichartz space $L_t^{q'}L_x^{p'}$ with $(q,p)$ admissible, and in the last inequality we used the fact that $q',p'\leq 2$ in order to  apply Minkowski's inequality and the Littlewood-Paley inequality. As a consequence of the above estimates and \eqref{farbound}, we obtain
\begin{equation}\label{slightlystronger}
\|P_ju\|_{l_j^2L_t^{\infty}H_x^s}\lesssim \|u_0\|_{H_x^s}+\|D_t^{\frac{s}{2}}F_T\|_{S'}+T_1^{\delta}\|F\|_{L_t^{\infty}H_x^{s-2+}}+\|F(0)\|_{H_x^{s-2+}}.
\end{equation}
Next, we aim to obtain a similar estimate for $D_t^{\frac{s}{2}}u$. For this, a direct Strichartz estimate gives the bound
\begin{equation*}
\|D_t^{\frac{s}{2}}u\|_{S}\lesssim \|(D_t^{\frac{s}{2}}u)(0)\|_{L_x^2}+\|D_t^{\frac{s}{2}}F_T\|_{S'}.
\end{equation*}
To estimate the initial data term, we note that
\begin{equation}\label{farandnear}
\|(D_t^{\frac{s}{2}}u)(0)\|_{L_x^2}\lesssim \|D_t^{\frac{s}{2}}Q_nu\|_{L_t^{\infty}L_x^2}+\|D_t^{\frac{s}{2}}Q_fu\|_{L_t^{\infty}L_x^2}.
\end{equation}
Arguing similarly to \eqref{Qinest} and applying Bernstein, we have
\begin{equation*}
\begin{split}
\|D_t^{\frac{s}{2}}Q_nu\|_{L_t^{\infty}L_x^2}&\lesssim\|D_t^{\frac{s}{2}}S_{2j}P_ju\|_{l_j^2L_t^{\infty}L_x^2}\lesssim \|P_ju\|_{l_j^2L_t^{\infty}H_x^s},
\end{split}
\end{equation*}
which can be controlled using \eqref{slightlystronger}. For the latter term in \eqref{farandnear}, we can argue as in \eqref{farbound} to obtain by Young's inequality (and interpreting $(i\partial_t-\Delta)$ as elliptic of order $1$ in time, rather than order $2$ as before)
\begin{equation*}
\|D_t^{\frac{s}{2}}Q_fu\|_{L_t^{\infty}L_x^2}\lesssim \|F_T\|_{W_t^{\frac{s}{2}-1+,\infty}L_x^2}.
\end{equation*}
Next, we make the important observation that if $(q',p')$ are dual Strichartz exponents, then $\frac{2}{q'}+d(\frac{1}{p'}-\frac{1}{2'})= 2$. Hence, considering separately the regions of Fourier space where $|\tau|\ll |\xi|^2$ and $|\tau|\gg |\xi|^2$, Bernstein's inequality  gives
\begin{equation}\label{s2-1}
\|F_T\|_{W_t^{\frac{s}{2}-1+,\infty}L_x^2}\lesssim T_1^{\delta}\|F\|_{L_t^{\infty}H_x^{s_+-2}}+\|F(0)\|_{H_x^{s_+-2}}+\|D_t^{\frac{s}{2}}F_T\|_{S'_+}.
\end{equation}
Indeed, the first two terms on the right-hand side of \eqref{s2-1} come from the contribution in the region  $|\tau|\ll |\xi|^2$, whereas the latter term comes from the contribution in the region $|\tau|\gtrsim |\xi|^2$. Note also that we measured $D_t^{\frac{s}{2}}F_T$ in $S'_+$ to account for the slight dyadic summability losses.
\medskip

It finally remains to estimate $\|D_t^{\frac{s}{2}}F_T\|_{S_+'}$ in terms of $F$. Clearly, we have
\begin{equation*}
\|D_t^{\frac{s}{2}}(\eta_1F(0,x))\|_{S_+'}\lesssim \|D_t^{\frac{s}{2}}(\eta_1F(0,x))\|_{L_t^{1+}L_x^2}\lesssim \|F(0,\cdot)\|_{L_x^2}. 
\end{equation*}
On the other hand, from the Littlewood-Paley trichotomy (in time), we see that
\begin{equation*}
\|D_t^{\frac{s}{2}}(\eta_{T_1}\partial_t^{-1}(\eta_T\partial_tF))\|_{S_+'}\lesssim C(T_1)\|\partial_t^{-1}(\eta_T\partial_tF)\|_{L^\infty_tL_x^{\tilde{p}'}}+\|D_t^{\frac{s}{2}}\partial_t^{-1}(\eta_T\partial_tF)\|_{S'_+},
\end{equation*}
where $\tilde{p}$ is Strichartz admissible. By H\"older in time, for any dual admissible exponent $(\tilde{q}',\tilde{p}')$, we have
\begin{equation*}
C(T_1)\|\partial_t^{-1}(\eta_T\partial_tF)\|_{L^\infty_tL_x^{\tilde{p}'}}\lesssim C(T_1)\|\eta_T\partial_tF\|_{L^{1}_tL_x^{\tilde{p}'}}\lesssim T^{\delta}\|\partial_tF\|_{L^{\tilde{q}'+}_tL_x^{\tilde{p}'}}.
\end{equation*}
On the other hand, the Littlewood-Paley trichotomy and H\"older in time gives
\begin{equation*}
\|D_t^{\frac{s}{2}}\partial_t^{-1}(\eta_T\partial_tF)\|_{S'_+}\lesssim \|D_t^{\frac{s}{2}-1}(\eta_T\partial_tF)\|_{S'_+}\lesssim T^{\delta}\|D_t^{\frac{s}{2}}F\|_{S'_+}+T^{\delta}\|F\|_{W^{\frac{s}{2}-1+,\infty}_tL_x^2}.
\end{equation*}
The last term on the right can be estimated as in \eqref{s2-1}. Combining the above estimates gives (iii).
\medskip

Now, we proceed to parts (iv) and (v), which we do simultaneously. First, by arguing almost exactly as in part (iii), we can estimate $u$ in $L_t^{\infty}H_x^s$ and also   $Q_n\partial_t^jD_x^{s-2j}u$ in $L_t^{q}L_x^p$ by
\begin{equation*}\label{part41}
\begin{split}
\|Q_n\partial_t^jD_x^{s-2j}u\|_{L_t^{q}L_x^p}+\|u\|_{L_t^{\infty}H_x^s}\lesssim &\|u_0\|_{H_x^s}+T_1^{\delta}\|F\|_{L_t^{\infty}H_x^{s-2+}}+\|F(0,x)\|_{H_x^{s-2+}}
\\
+&\|\partial_t^2D_x^{s-4}S_{2j}P_jF_T\|_{l_j^2S'},
\end{split}
\end{equation*}
where the last term comes from modifying the estimate \eqref{Qinest}.  On the other hand, for any admissible $(q,p)$, using the equation
\begin{equation*}
i\partial_t^2D_x^{s-4}u=\partial_tD_x^{s-4}\Delta u+\partial_tD_x^{s-4}F_T
\end{equation*}
and Sobolev embedding, we have the elliptic estimate 
\begin{equation*}
\|Q_f\partial_t^{j}D_x^{s-2j}u\|_{L_t^qL_x^p}\lesssim \|\partial_tF_T\|_{L_t^qW_x^{s-4+,p}}\lesssim T_1^{\delta}\|\partial_tF\|_{L_t^{q+}W_x^{s-4+,p}}+\|F(0,x)\|_{H_x^{s-2+}}.
\end{equation*}
It remains then to estimate $\|\partial_t^2D_x^{s-4}S_{2j}P_jF_T\|_{l_j^2S'}$. To do this, we first expand
\begin{equation*}
\begin{split}
\partial_t^2D_x^{s-4}F_T&=\partial_t^2\eta_1D_x^{s-4}F(0,\cdot)+\partial_t^2\eta_{T_1}\partial_t^{-1}(\eta_TD_x^{s-4}\partial_tF)+\partial_t(\eta_TD_x^{s-4}\partial_tF)
\\
&=:F_1+F_2+F_3,
\end{split}
\end{equation*}
where we used that $\eta_{T_1}=1$ on the support of $\eta_T$. To estimate $F_1$ and $F_2$ we use the dual norm $L_t^1L_x^2$ which yields
\begin{equation*}
\|S_{2j}P_jF_1\|_{l_j^2L_t^1L_x^2}+\|P_jS_{2j}F_2\|_{l_j^2L_t^1L_x^2}\lesssim \|F(0,\cdot)\|_{H_x^{s-4}}+T_1^{-1}T\|D_x^{s-4}\partial_tF\|_{L_t^{\infty}L_x^2}.
\end{equation*}
This is sufficient if $T\ll T_1$. To handle $F_3$, we decompose it into two components corresponding to the low-high (time) frequency interactions for $\eta_T$  and a remainder, as follows:
\begin{equation*}
\begin{split}
P_jS_{2j}\partial_t(\eta_TD_x^{s-4}\partial_tF)&=P_jS_{2j}\partial_t(S_{<2j-4}\eta_TD_x^{s-4}\partial_t\tilde{S}_{2j}F)+P_jS_{2j}\partial_t(S_{\geq 2j-4}\eta_TD_x^{s-4}\partial_tF)
\\
&=:F_{3,j}^{lh}+F_{3,j}^r.
\end{split}
\end{equation*}
For the remainder term, we use Bernstein (in space and time) to estimate
\begin{equation*}
\begin{split}
\|F_{3,j}^r\|_{L_t^1L_x^2}&\lesssim 2^{2j}\|S_{\geq 2j-4}\eta_T\|_{l_j^{\infty}L_t^1}\|P_jD_x^{s-4}\partial_tF\|_{L_t^{\infty}L_x^2}\lesssim 2^{-j\delta}\|D_t^{1-\delta}\eta_T\|_{L_t^1}\|D_x^{s-4+3\delta}\partial_tF\|_{L_t^{\infty}L_x^2}
\\
&\lesssim 2^{-j\delta}T^{\delta}\|D_x^{s-4+3\delta}\partial_tF\|_{L_t^{\infty}L_x^2},
\end{split}
\end{equation*}
which yields
\begin{equation*}
\|F_{3,j}^r\|_{l_j^2L_t^1L_x^2}\lesssim T^{\delta}\|D_x^{s-4+}\partial_tF\|_{L_t^{\infty}L_x^2}.
\end{equation*}
For the low-high term, by Bernstein, H\"older in $t$ and splitting $\partial_t^2F=G+H$, we have
\begin{equation*}
\|F_{3,j}^{lh}\|_{l_j^2S'}\lesssim T^{\delta}\|\tilde{S}_{2j}P_jD_x^{s-4}\partial_t^2F\|_{l_j^2S'_+}\lesssim T^{\delta}(\|\tilde{S}_{2j}P_jD_x^{s-4}G\|_{l_j^2S'_+}+\|\tilde{S}_{2j}P_jD_x^{s-4}H\|_{l_j^2S'_+}).
\end{equation*}
From Bernstein, Minkowski and the Littlewood-Paley inequality, we can bound 
\begin{equation*}
\|\tilde{S}_{2j}P_jD_x^{s-4}G\|_{l_j^2S'_+}\lesssim \|P_jD_x^{s-4}G\|_{l_j^2S'_+}\lesssim \|D_x^{s-4}G\|_{S'_+}.
\end{equation*}
On the other hand, if $4\leq s<6$, then by Bernstein (in both the space and time variable) and the same argument, we can estimate
\begin{equation*}
\|\tilde{S}_{2j}P_jD_x^{s-4}H\|_{l_j^2S'_+}\lesssim \|\tilde{S}_{2j}P_jD_t^{\frac{s}{2}-2}H\|_{l_j^2S'_+}\lesssim \|D_t^{\frac{s}{2}-2}H\|_{S'_+}.
\end{equation*}
If $s\geq 6$, similar arguments allow one to instead estimate 
\begin{equation*}
\|\tilde{S}_{2j}P_jD_x^{s-4}H\|_{l_j^2S'_+}\lesssim \|D_x^{s-6}\partial_tH\|_{S'_+}.
\end{equation*}
This concludes the proofs of parts (iv) and (v).
\end{proof}
\subsection{A priori estimates for the time-truncated  nonlinear Schr\"odinger equation}\label{apriorischrodinger} 
In this subsection, we prove a priori estimates for the equation 
\begin{equation}\label{truncated NLS}
\begin{cases}
&i\partial_tu-\Delta u=N_T,
\\
&u(0)=u_0,
\end{cases}
\end{equation}
where $N_T$ is defined as in \Cref{truncationestimates} with $N(t,x)=|u|^{p-1}u$ and $0<T\ll 1$. Since we will only be estimating expressions involving the nonlinear term below, by slight abuse of notation, we can assume without loss of generality throughout this subsection that $u$ is supported on the time interval $t\in [-2,2]$. This will allow us to harmlessly apply H\"older's inequality in time in various spots. We caveat, however, that the solution itself is not localized in time (rather outside of $[-2,2]$, it solves the linear Schr\"odinger equation).
\medskip

Estimates corresponding to data in $H^s_x(\mathbb{R}^d)$ when $s\leq 2$ have been treated extensively in the literature \cite{MR3546788,MR4388268,MR3424613, MR4581790} and do not need the refined nonlinear estimates that we consider below. Therefore, we will consider data $u_0\in H^s_x(\mathbb{R}^d)$ where $s$ satisfies the constraint $\min\{2p+1,p+\frac{5}{2}\}>s>\max\{2,s_c\}$. Given a small enough $T>0$, our goal will be to find a suitable scale of Strichartz-type spaces $X^s\subset C(\mathbb{R};H_x^s(\mathbb{R}^d))$ in which we can close a good a priori estimate for $u$. The choice of $X^s$ will vary essentially depending on whether we are in a $L^2$-subcritical or $L^2$-supercritical regime and on the range of $s$. Consequently, we will have to carry out a somewhat large case analysis. As the aim of this paper is to study the high regularity local well-posedness problem, we will not attempt to heavily optimize the estimates below (although this should be possible) to obtain low regularity continuation or persistence of regularity criteria. Rather, some of the bounds will be quite crude to simplify our analysis in places (especially for low dimensions and small powers $p$ where simple Sobolev embeddings often suffice).
\subsubsection{Estimates when $2<s<4$}
We begin our analysis by considering the regime $2<s<4$. We split this case further into two subcases. The first subcase corresponds to the situation when either $d=1,2$ or when the problem is $L^2$-subcritical (i.e.~$s_c\leq 0$), while in the second subcase, we will consider $d\geq 3$ and supercritical exponents $s_c>0$. The reason for treating dimensions one and two separately from higher dimensions is because, in the latter case, we will have access to the endpoint Strichartz norm $L_t^2L_x^{\frac{2d}{d-2}}$, unlike when $d=1$ and $d=2$.
\medskip

\paragraph{\textit{Estimates in the range $2<s<4$ when $d=1,2$ or $s_c\leq 0$}} In this setting, we will take
\begin{equation*}
X^s:=C(\mathbb{R};H_x^s(\mathbb{R}^d))\cap \mathcal{T}^s,
\end{equation*}
where we write for $s\geq 0$,
\begin{equation*}
\mathcal{T}^s:=\bigcap_{i}W_t^{\frac{s}{2},q_i}L_x^{p_i},\hspace{5mm} \frac{2}{q_i}+\frac{d}{p_i}=\frac{d}{2}.
\end{equation*}
Here, the intersection ranges over some finite index set of admissible pairs $(q_i,p_i)$, with the exact choice depending on the particular case that we are considering in our analysis. Our estimates in this setting are summarized by the following proposition.
\begin{proposition}\label{subcrit}
Let $s\in (2,4)$ and let $s_c\leq 0$ or $d=1,2$. Then there is a $\delta>0$ and $s_0<s$ such  that for all $u$ satisfying the equation \eqref{truncated NLS}  we have
\begin{equation}\label{subcritprop}
\|u\|_{X^s}\lesssim_{\|u_0\|_{H^{s_0}_x}} \|u_0\|_{H_x^s}+T_1^{\delta}\|u\|_{X^{s_0}}^{p-1}\|u\|_{X^s}.
\end{equation}
\end{proposition}
\begin{proof}
Part (iii) of \Cref{truncationestimates} allows us to estimate
\begin{equation}\label{Xs estimate}
\|u\|_{X^s}\lesssim \|u_0\|_{H_x^s}+\||u_0|^{p-1}u_0\|_{H^{s_+-2}_x}+T_1^{\delta}\||u|^{p-1}u\|_{L_t^{\infty}H_x^{s_+-2}}+T_1^{\delta}\|\langle D_t\rangle^{\frac{s}{2}}(|u|^{p-1}u)\|_{S_+'}.
\end{equation}
In the above, $S'_+$ will take the form $L_t^{q'+}(\mathbb{R};L_x^{r'})$ for some well-chosen admissible pair $(q,r)$. Thanks to the nonlinear estimate in \Cref{farestimateH} which we used previously to analyze \eqref{NLH}, it suffices only to estimate the last term in \eqref{Xs estimate} by the right-hand side of \eqref{subcritprop}. For this, we consider a few subcases. As mentioned earlier, below we will assume for convenience that $u$ is supported on the time interval $[-2,2]$ in the nonlinear estimates below.
\medskip 

\textbf{Case 1: $1\leq d\leq 2$.}
First, we dispense with the one and two-dimensional cases where we do not have access to the Strichartz endpoint space $L_t^2L_x^{\frac{2d}{d-2}}$. In this case, it suffices to crudely estimate the nonlinearity in the space $S'_+:=L_t^{1+}L_x^2$. Precisely, we have by Minkowski's inequality, \Cref{Nonlinear estimate} (using the restriction $\frac{s}{2}<p+\frac{1}{2}$) and simple applications of H\"older's inequality,
\begin{equation*}
\begin{split}
\|D_t^{\frac{s}{2}}(|u|^{p-1}u)\|_{L_t^{1+}L_x^2}&\lesssim \||u|^{p-1}u\|_{L_x^{2}H_t^{\frac{s}{2}}}
\\
&\lesssim \|\|u\|_{H_t^{\frac{1}{2}+}}^{p-1}\|u\|_{H_t^{\frac{s}{2}}}\|_{L_x^2}
\\
&\lesssim \|u\|^{p-1}_{W_t^{\frac{1}{2}+\epsilon,\rho}L_x^{r'(p-1)}}\|u\|_{W_t^{\frac{s}{2},q}L_x^r},
\end{split}
\end{equation*}
where $r,\rho,q>2$ are chosen so that $(q,r)$ and $(\rho, r'(p-1))$ are admissible Strichartz pairs. We can, for instance, simply take $r$ such that $\frac{1}{2}=\frac{1}{r}+\frac{1}{r'}$ with $r>2$ sufficiently close to 2, so that $(\rho,r'(p-1))$ is close but not equal to the forbidden endpoint $(2,\infty)$ when $d=2$. In this case, we have
\begin{equation*}
\|u\|_{W_t^{\frac{1}{2}+\epsilon,\rho}L_x^{r'(p-1)}}\lesssim  \|u\|_{X^{s_0}}
\end{equation*}
for some $1<s_0<s$. This handles the case $1\leq d\leq 2$.
\medskip

\textbf{Case 2: $d\geq 3$ and $s_c\leq 0$.} We specialize to two subcases depending on whether $p<2$ or $p\geq 2$. In our analysis below, we will use the (near) endpoint dual norm $S'_+:=L_t^{2}L_x^{\frac{2d}{d+2}+}$ to estimate the nonlinear term.
\medskip

\textbf{Case 2.1: $p\geq 2$.} This case only arises when $3\leq d\leq 4$. From Minkowski, H\"older in time and \Cref{Nonlinear estimate}, we have
\begin{equation*}
\begin{split}
\|D_t^{\frac{s}{2}}(|u|^{p-1}u)\|_{L_t^{2}L_x^{\frac{2d}{d+2}+}}&\lesssim \|\|u\|^{p-1}_{H_t^{\frac{1}{2}+\epsilon}}\|u\|_{H_t^{\frac{s}{2}}}\|_{L_x^{\frac{2d}{d+2}+}}\lesssim \|u\|_{L_x^{d(p-1)}{H_t^{\frac{1}{2}+\epsilon}}}^{p-1}\|u\|_{H_t^{\frac{s}{2}}L_x^{2+}}
\\
&\lesssim \|u\|_{H_t^{\frac{1}{2}+\epsilon}L_x^{d(p-1)}}^{p-1}\|u\|_{H_t^{\frac{s}{2}}L_x^{2+}}\lesssim \|u\|_{H_t^{\frac{1}{2}+\epsilon}(L_x^{\frac{2d}{d-2}}\cap L_x^2)}^{p-1}\|u\|_{H_t^{\frac{s}{2}}L_x^{2+}}
\\
&\lesssim \|u\|_{X^{s_0}}^{p-1}\|u\|_{X^s},
\end{split}    
\end{equation*}
for some small $\delta>0$ and $1<s_0<s$. We note that in the third line we interpolated, using that $2\leq d(p-1)\leq \frac{2d}{d-2}$, which one can easily verify is admissible given the postulated restrictions on $d$ and $p$.
\medskip

\textbf{Case 2.2: $p<2$.} Here, if $\frac{2}{p}< \frac{2d}{d+2}$ then we have $\frac{2}{d}+1< p\leq \frac{4}{d}+1$, so that 
\begin{equation*}
\begin{split}
\|D_t^{\frac{s}{2}}(|u|^{p-1}u)\|_{L_t^{2}L_x^{\frac{2d}{d+2}+}}&\lesssim \|\|u\|_{H_t^{\frac{1}{2}+\epsilon}}^{p-1}\|_{L_x^{\frac{2}{p-1}+}}\|u\|_{H_t^{\frac{s}{2}}L_x^r}\lesssim \|u\|_{H_t^{\frac{1}{2}+\epsilon}L_x^{2+}}^{p-1}\|u\|_{H_t^{\frac{s}{2}}L_x^r}\lesssim \|u\|_{X^{s_0}}^{p-1}\|u\|_{X^s}
\end{split}    
\end{equation*}
with $r=\frac{2d}{d(2-p)+2}$. Since our hypotheses ensure that $2\leq r\leq \frac{2d}{d-2}$, we see that $r$ is admissible for some exponent $2\leq q\leq\infty$, which suffices.
\medskip

If, on the other hand, $\frac{2}{p}\geq\frac{2d}{d+2}$, then the dual exponent $\frac{2}{2-p}$ of $\frac{2}{p}$ is admissible so that if $q$ is the corresponding time exponent, we have similarly to the above that
\begin{equation*}
\||u|^{p-1}u\|_{W_t^{\frac{s}{2},q'}L_x^{\frac{2}{p}}}\lesssim \|u\|_{X^{s_0}}^{p-1}\|u\|_{X^s}.
\end{equation*}
This handles the remaining subcritical case and therefore concludes the proof of \Cref{subcrit}.
\end{proof}
\paragraph{\textit{Estimates in the range $2<s<4$ when $d\geq 3$ and $s_c>0$}} In this setting, we will obtain estimates in the more restrictive function space
\begin{equation}\label{supercritspace}
X^s:=C(\mathbb{R};H_x^s(\mathbb{R}^d))\cap \mathcal{T}^s\cap H_t^{\frac{1}{2}+}L_x^{\frac{d}{2}(p-1)+}.
\end{equation}
\begin{remark}\label{functionspacemotivation}
Here we elaborate on the reason for the addition of the seemingly strange $H_t^{\frac{1}{2}+}L_x^{\frac{d}{2}(p-1)+}$ norm. To motivate this, by Sobolev embeddings, one should observe that this norm scales almost like the Strichartz-type norm $H_t^{\frac{1}{2}}W_x^{s_c-1,\frac{2d}{d-2}}$. Unfortunately, directly estimating $u$ in this norm seems to be impossible for certain ranges of $d$ and $p$. This is essentially due to the fact that the maximal allowable Sobolev exponent in the nonlinear estimate for $|u|^{p-1}u$ is worse in $L_x^{\frac{2d}{d-2}}$ than the corresponding estimate in $L_x^2$ based Sobolev spaces. For this reason, it is better to work directly with the norm $L_x^{\frac{d}{2}(p-1)+}$. While we do not have a direct Strichartz estimate in this case, we can at least estimate $u$ in the region of Fourier space with $|\tau|\ll |\xi|^2$ with the elliptic estimate from \Cref{truncationestimates}. In the region $|\tau|\gtrsim |\xi|^2$, we can safely Sobolev embed and estimate the resulting contribution in the norm $\|\cdot\|_{H_t^{\frac{s_c}{2}+\epsilon}L_x^{\frac{2d}{d-2}}}$ which we do have a Strichartz estimate for.
\end{remark}
The main estimate in this setting is summarized by the following proposition.
\begin{proposition}\label{supercrit}
Let $s\in (2,4)$, $s_c>0$ and $d\geq 3$. Let $X^s$ be as in \eqref{supercritspace}. Then there is a $\delta>0$ such that for all $u$ satisfying the equation \eqref{truncated NLS} we have
\begin{equation*}
\|u\|_{X^s}\lesssim_{\|u_0\|_{H_x^{s_0}}} \|u_0\|_{H_x^s}+T_1^{\delta}\|u\|_{X^{s_0}}^{p-1}\|u\|_{X^s}.
\end{equation*}
\end{proposition}
\begin{proof}
In this case, $s_c>0$ ensures that we have $\frac{d}{2}(p-1)> 2$. Similarly to the case above, we can apply statements (ii) and (iii) in \Cref{truncationestimates} to obtain the initial estimate
\begin{equation}\label{threeterms}
\begin{split}
\|u\|_{X^s}\lesssim \|u_0\|_{H_x^s}+T_1^{\delta}\||u|^{p-1}u\|_{L_t^{\infty}H_x^{s_+-2}}+&T_1^{\delta}\|\langle D_t\rangle^{\frac{s}{2}}(|u|^{p-1}u)\|_{L_t^{2}L_x^{\frac{2d}{d+2}+}}
\\
+&T_1^{\delta}\||u|^{p-1}u\|_{H_t^{\frac{1}{2}+}W_x^{-2+,\frac{d}{2}(p-1)+}}
\end{split}
\end{equation}
for some $\delta>0$ and implicit constant depending on $\|u_0\|_{H_x^{s_0}}$. The estimate for the second term on the right-hand side of \eqref{threeterms} is completely analogous to the subcritical case. For the third term, arguing as before, we may obtain the bound
\begin{equation*}
\begin{split}
\|D_t^{\frac{s}{2}}(|u|^{p-1}u)\|_{L_t^2L_x^{\frac{2d}{d+2}+}}&\lesssim \||u|^{p-1}u\|_{L_x^{\frac{2d}{d+2}+}H_t^{\frac{s}{2}}} \lesssim \|\|u\|^{p-1}_{H_t^{\frac{1}{2}+\epsilon}}\|u\|_{H_t^{\frac{s}{2}}}\|_{L_x^{\frac{2d}{d+2}+}} 
\\
&\lesssim \|\|u\|_{H_t^{\frac{1}{2}+\epsilon}}\|_{L_x^{\frac{d}{2}(p-1)+}}^{p-1}\|u\|_{H_t^{\frac{s}{2}}L_x^{\frac{2d}{d-2}}}
\\
&\lesssim \|u\|_{H_t^{\frac{1}{2}+\epsilon}L_x^{\frac{d}{2}(p-1)+}}^{p-1}\|u\|_{H_t^{\frac{s}{2}}L_x^{\frac{2d}{d-2}}}
\\
&\lesssim \|u\|_{X^{s_0}}^{p-1}\|u\|_{X^s}.
\end{split}
\end{equation*}
It remains to obtain an estimate for the final term on the right-hand side of \eqref{threeterms}. If $d\geq 4$, we can simply Sobolev embed in space and use the  (vector-valued) Moser estimate to obtain
\begin{equation*}
\begin{split}
\||u|^{p-1}u\|_{H_t^{\frac{1}{2}+}W_x^{-2+,\frac{d}{2}(p-1)+}}&\lesssim \||u|^{p-1} u\|_{H_t^{\frac{1}{2}+}L_x^{\frac{d(p-1)}{2p}+}}\lesssim \|u\|_{L_t^{\infty}L_x^{\frac{d}{2}(p-1)+}}^{p-1}\|u\|_{H_t^{\frac{1}{2}+}L_x^{\frac{d}{2}(p-1)+}}
\\
&\lesssim \|u\|_{X^{s_0}}^{p-1}\|u\|_{X^s}.
\end{split}
\end{equation*}
On the other hand, if $d\leq 3$, we directly have by Moser and the restriction $s_0>2$,
\begin{equation*}
\begin{split}
\||u|^{p-1}u\|_{H_t^{\frac{1}{2}+}W_x^{-2+,\frac{d}{2}(p-1)+}}&\lesssim \||u|^{p-1}u\|_{H_t^{\frac{1}{2}+}L_x^{\frac{d}{2}(p-1)}}\lesssim \|u\|^{p-1}_{L_t^{\infty}L_x^{\infty-}}\|u\|_{H_t^{\frac{1}{2}+}L_x^{\frac{d}{2}(p-1)+}}
\\
&\lesssim \|u\|^{p-1}_{X^{s_0}}\|u\|_{X^{s}}.
\end{split}
\end{equation*}
This completes the proof of the proposition.
\end{proof}
\subsubsection{Estimates when $s\geq 4$}
Now that we have dispensed with the case $2<s<4$ we turn our attention to the cases $4\leq s<6$ and $s\geq 6$. Here, we define $X^s$ by
\begin{equation*}
X^s:=C(\mathbb{R};H_x^s(\mathbb{R}^d))\cap \mathcal{T}^s_1\cap\mathcal{T}^s_2,
\end{equation*}
where we write for $s\geq 0$ and $j=1,2$,
\begin{equation*}
\mathcal{T}^s_j:=\bigcap_{i}W_t^{j,\, q_i}W_x^{s-2j,\, p_i},\hspace{5mm} \frac{2}{q_i}+\frac{d}{p_i}=\frac{d}{2}.
\end{equation*}
Thanks to \Cref{truncationestimates} parts (iv) and (v) and the nonlinear estimate from \Cref{farestimateH}, we need to control:
\begin{itemize}
\item\label{items3}
$\|\partial_tD_x^{s-4+}(|u|^{p-1}u)\|_{L_t^{q+}L_x^r\cap L_t^{\infty}L_x^2}$ when $s\geq 4$ and $(q,r)\neq (4,\infty)$ is an admissible Strichartz pair;
\vspace{2mm}
\item\label{items1} $\|(D_x^{s-4}G,D_x^{s-6}\partial_tH)\|_{S_+'}$ when $s\geq 6$;
\vspace{2mm}
\item\label{items2} $\|(D_x^{s-4}G,D_t^{\frac{s}{2}-2}H)\|_{S_+'}$ when $4\leq s< 6$;
\end{itemize}
\vspace{2mm}
where $\partial_t^2(|u|^{p-1}u)=G+H$ is a well-chosen partition. We choose the partition in the following way: $G$ will consist of the portion of $\partial_t^2(|u|^{p-1}u)$ where exactly one factor of $u_{tt}$ (or its complex conjugate) appears. This includes, for instance, expressions of the form $|u|^{p-1}u_{tt}$. $H$ will consist of all remaining terms (i.e.~components which are quadratic in $u_t$ but involve no higher order factors, e.g., $|u|^{p-2}u_tu_t$). 
Below, the estimates are done for such model terms; the extension to the general case is just a matter of notation. 
\medskip

\paragraph{\textit{Estimates for $\|\partial_tD_x^{s-4+}(|u|^{p-1}u)\|_{L_t^{q+}L_x^{r}\cap L_t^{\infty}L_x^2}$ for admissible $(q,r)\neq (4,\infty)$}} Our goal here will be to establish the estimate
\begin{equation*}
\|\partial_tD_x^{s-4+}(|u|^{p-1}u)\|_{L_t^{q+}L_x^r\cap L_t^{\infty}L_x^2}\lesssim \|u\|_{X^{s_0}}^{p-1}\|u\|_{X^s}
\end{equation*}
for some $s_0$ satisfying $s_c<s_0<s-\epsilon$. We show the details for the $L_t^{q+}L_x^r$ component as the $L_t^{\infty}L_x^2$ bound follows by a slight modification (it is almost a special case, in fact).
\begin{remark}
Of course, the restriction for $s_0$ given by $s_c<s_0<s-\epsilon$ can be greatly optimized in many regimes, but we will not do so here as it will not be important for our main results.
\end{remark}
We begin by expanding 
\begin{equation*}
D_x^{s-4+}\partial_t(|u|^{p-1}u)=\frac{p+1}{2}D_x^{s-4+}(|u|^{p-1}u_t)+\frac{p-1}{2}D_x^{s-4+}(|u|^{p-3}u^2\overline{u}_t).
\end{equation*}
We will show the details for the first term as the latter term can be dealt with by nearly identical reasoning. 
Note that here we are applying few enough derivatives to be able to carry out the estimate using standard fractional Leibniz-type rules. As above, we will need to consider the cases $p\leq 2$ and $p> 2$ separately.
\medskip

\textbf{Case 1: $p\leq 2$.}
By the fractional Leibniz rule and \Cref{KV Prop}, for $r\leq r_1,r_2,s_1,s_2<\infty$ satisfying $\frac{1}{r}=\frac{1}{s_1}+\frac{1}{s_2}=\frac{1}{r_1}+\frac{1}{r_2}$, there holds
\begin{equation*}
\|D_x^{s-4+}(|u|^{p-1}u_t)\|_{L_t^{q+}L_x^r}\lesssim \|u\|_{L_t^{\infty}W_x^{\frac{s-4}{p-1}+,(p-1)r_1}}^{p-1}\|u_t\|_{L_t^{q+}L_x^{r_2}}+\|u\|^{p-1}_{L_t^{\infty}L_x^{s_1(p-1)}}\|D_x^{s-4+}u_t\|_{L_t^{q+}L_x^{s_2}}.
\end{equation*}
Note that we can always choose $s_1$ and $s_2$ such that we have the embeddings
\begin{equation}\label{chooses1s2}
\|D_x^{s-4+}u_t\|_{L_t^{q+}L_x^{s_2}}\lesssim \|u\|_{W^{1,q+}_tW_x^{s-2,\tilde{r}}},\hspace{5mm}\|u\|_{L_t^{\infty}L_x^{s_1(p-1)}}\lesssim \|u\|_{L_t^{\infty}H_x^{s_0}},
\end{equation}
where  $s_c<s_0<s-\epsilon$ and $\tilde{r}$ is such that $(q+,\tilde{r})$ is admissible. We can then interpolate to conclude that
\begin{equation*}
\|u\|_{W_t^{1,q+}W_x^{s-2,\tilde{r}}}\lesssim \|u\|_{W_t^{1,\infty}H_x^{s-2}}^{\theta}\|u\|_{W_t^{1,q}W_x^{s-2,r}}^{1-\theta},
\end{equation*}
for some $0<\theta<1$. To estimate the other term, we note that since $p\leq 2$ we have $s-2<\frac{5}{2}$ (from the restriction $s<p+\frac{5}{2}$). Therefore, since necessarily $r\leq \frac{2d}{d-2}$, if $d\geq 7$ we can choose $r_1$ and $r_2$ such that
\begin{equation*}
\|u_t\|_{L_t^{q+}L_x^{r_2}}\lesssim \|D_x^{s-2}u_t\|_{L_t^{q+}L_x^{\tilde{r}}},\hspace{5mm}\|u\|_{L_t^{\infty}W_x^{\frac{s-4}{p-1},(p-1)r_1}}\lesssim \|u\|_{L_t^{\infty}H_x^{s_0}}.
\end{equation*}
On the other hand, if $d\leq 6$ we can simply take $r_2=r$ and $r_1=\infty$ so that we trivially have
\begin{equation*}
\|u_t\|_{L_t^{q+}L_x^{r_2}}\lesssim \|u\|_{W_t^{1,q+}W_x^{s-2,\tilde{r}}},\hspace{5mm}\|u\|_{L_t^{\infty}W_x^{\frac{s-4}{p-1},\infty}}\lesssim \|u\|_{L_t^{\infty}H_x^{4-\epsilon}}.
\end{equation*}
Combining everything above gives
\begin{equation}\label{extrastrichartz}
\|D_x^{s-4+}(|u|^{p-1}u_t)\|_{L_t^{q+}L_x^r}\lesssim \|u\|_{X^{s_0}}^{p-1}\|u\|_{X^s}
\end{equation}
for some $s_c<s_0<s-\epsilon$.
\medskip

\textbf{Case 2: $p\geq 2$.} Here, by \Cref{Crude Moser est} and \Cref{Leib1}, we instead have for $r\leq r_1,r_2,r_3,s_1,s_2<\infty$ with $\frac{1}{r}=\frac{1}{r_1}+\frac{1}{r_2}+\frac{1}{r_3}=\frac{1}{s_1}+\frac{1}{s_2}$,
\begin{equation*}
\|D_x^{s-4+}(|u|^{p-1}u_t)\|_{L_t^{q+}L_x^r}\lesssim \|u\|^{p-2}_{L_t^{\infty}L_x^{(p-2)r_1}}\|u\|_{L_t^{\infty}W_x^{s-4+,r_2}}\|u_t\|_{L_t^{q+}L_x^{r_3}}+\|u\|^{p-1}_{L_t^{\infty}L_x^{s_1(p-1)}}\|D_x^{s-4+}u_t\|_{L_t^{q+}L_x^{s_2}}.
\end{equation*}
The latter term can be estimated exactly as in \eqref{chooses1s2} by choosing $s_1$ and $s_2$ appropriately. If $d\geq 8$, we can take $r_1=\frac{d}{2}\frac{p-1}{p-2}$, $r_2=\frac{2d}{d-8}-$ so that
\begin{equation*}
\|u\|_{L_t^{\infty}L_x^{(p-2)r_1}}\lesssim \|u\|_{L_t^{\infty}H_x^{s_c+}},\hspace{5mm}\|u\|_{L_t^{\infty}W_x^{s-4+,r_2}}\lesssim \|u\|_{L_t^{\infty}H_x^s},\hspace{5mm}\|u_t\|_{L_t^{q+}L_x^{r_3}}\lesssim \|u_t\|_{L_t^{q+}W_x^{s_c-2+,\tilde{r}}}.
\end{equation*}
On the other hand, if $d\leq 7$, we can simply take $r_1=r_2=\infty$ and $r_3=r$  and then Sobolev embed to obtain a crude but sufficient estimate. Combining the above together, one again obtains \eqref{extrastrichartz}.
\medskip
\paragraph{\textit{Estimates for $\|D_x^{s-4}G\|_{S_+'}$}} 
We now analyze the term $\|D_x^{s-4}G\|_{S_+'}$ which is relevant for all $s\geq 4$. Analogously to the case $s<4$, in the analysis below we will consider separately the $L^2$ subcritical and supercritical regimes. We begin with the latter case, where we further split our analysis into two subcases depending on whether $\frac{3}{2}< p<2$ or $p\geq 2$ (note that these are the only two cases because of the restriction $4\leq s<p+\frac{5}{2}$). One of the terms that appears in the definition of $G$ is $|u|^{p-1}u_{tt}$. Below, we will show how to estimate this particular expression, as the other terms have a very similar structure and can be handled by identical reasoning.
\medskip

\textbf{Case 1.1: $s_c\geq 0$ and $\frac{3}{2}< p<2$.} In this case, we must have $d\geq 4$, so we can use the endpoint dual space $L_t^2L_x^{\frac{2d}{d+2}}$. To proceed, we apply the fractional Leibniz rule and \Cref{KV Prop} to obtain
\begin{equation*}
\|D_x^{s-4}(|u|^{p-1}u_{tt})\|_{L_t^2L_x^{\frac{2d}{d+2}}}\lesssim \|u\|_{L_t^{\infty}L_x^{\frac{d}{2}(p-1)}}^{p-1}\|u\|_{H_t^2W_x^{s-4,\frac{2d}{d-2}}}+\|u\|_{L_t^{\infty}W_x^{\frac{s-4}{p-1},(p-1)r}}^{p-1}\|u_{tt}\|_{L_t^2L_x^{r'}},
\end{equation*}
where $r$ is a parameter to be chosen satisfying $\frac{d+2}{2d}=\frac{1}{r}+\frac{1}{r'}$. Since $s_c\geq 0$ and $d\geq 4$, we can choose $r$ such that we have $2\leq (p-1)r\leq\frac{d}{2}(p-1)$ and also the embeddings $ W_x^{s-4,\frac{2d}{d-2}}\subset L_x^{r'}$ and $H_x^{\sigma}\subset L_x^{(p-1)r}$ for some $\sigma>\max\{s_c-1,0\}$. In view of the restrictions $\frac{s-4}{p-1}\leq 1$ and $s\geq 4$, this implies that
\begin{equation}\label{referencebound2}
\|D_x^{s-4}(|u|^{p-1}u_{tt})\|_{L_t^2L_x^{\frac{2d}{d+2}}}\lesssim \|u\|^{p-1}_{X^{s_0}}\|u\|_{X^s}.
\end{equation}

\textbf{Case 1.2: $s_c\geq 0$ and $p\geq 2$.} In this case, the analysis is a bit simpler because we can use standard Moser estimates. When $d\geq 3$, the analogous estimate is
\begin{equation*}
\|D_x^{s-4}(|u|^{p-1}u_{tt})\|_{L_t^2L_x^{\frac{2d}{d+2}}}\lesssim \|u\|_{L_t^{\infty}L_x^{\frac{d}{2}(p-1)}}^{p-1}\|u_{tt}\|_{L_t^2W_x^{s-4,\frac{2d}{d-2}}}+\|u\|^{p-2}_{L_t^{\infty}L_x^{\frac{d}{2}(p-1)+}}\|u\|_{L_t^{\infty}W_x^{s-4,r}}\|u_{tt}\|_{L_t^2L_x^{r'}}.
\end{equation*}
Here, we have $\frac{d+2}{2d}=\frac{2}{d}\frac{p-2}{p-1}+\frac{1}{r}+\frac{1}{r'}$ where by the assumption $s_c>0$ we can ensure that $r$ and $r'$ are such that $W_x^{s-4,\frac{2d}{d-2}}\subset L_x^{r'}$ and $H_x^{4-\epsilon}\subset L_x^r$. This again gives the bound \eqref{referencebound2}. When $1\leq d\leq 2$, we can instead use the dual Strichartz norm $L_t^1L_x^2$, H\"older's inequality and Sobolev embeddings to obtain a similar result. 
\medskip

Next, we turn to the $L^2$ subcritical case $s_c<0$. Since $p>\frac{3}{2}$ when $s\geq 4$, we inherit the restriction $d\leq 7$. Again, we split the analysis into two subcases depending on whether $p<2$ or $p\geq 2.$
\medskip

\textbf{Case 2.1: $s_c<0$ and $\frac{3}{2}<p<2$.}
Here, we estimate using the fractional Leibniz rule and \Cref{KV Prop},
\begin{equation*}
\begin{split}
\|D_x^{s-4}(|u|^{p-1}u_{tt})\|_{L_t^{1+}L_x^2}&\lesssim \|u\|_{L_t^{\infty}L_x^{\infty}}^{p-1}\|u_{tt}\|_{L_t^{\infty}H_x^{s-4}}+\|u\|_{L_t^{\infty}W_x^{\frac{s-4}{p-1}+,r(p-1)}}^{p-1}\|u_{tt}\|_{L_t^{\infty}L_x^{r'}},
\end{split}
\end{equation*}
 where $\frac{1}{2}=\frac{1}{r}+\frac{1}{r'}$ with $r'$ taken very close to $2$ and $r$ taken very close to $\infty$. The restrictions on $p$ and $s$ ensure that we have $\frac{s-4}{p-1}<\frac{1}{2}$. Since $d\leq 7$, Sobolev embedding then yields the bound
\begin{equation*}
\|D_x^{s-4}(|u|^{p-1}u_{tt})\|_{L_t^{1+}L_x^2}\lesssim \|u\|_{L_t^{\infty}H_x^{4-\epsilon}}^{p-1}\|u_{tt}\|_{L_t^{\infty}H_x^{s-4}}\lesssim \|u\|_{X^{s_0}}^{p-1}\|u\|_{X^s}.
\end{equation*}

\textbf{Case 2.2: $s_c<0$ and $p\geq 2$.} Using  \Cref{Crude Moser est} and \Cref{Leib1}, the estimate in this case is
\begin{equation*}
\|D_x^{s-4}(|u|^{p-1}u_{tt})\|_{L_t^{1+}L_x^2}\lesssim \|u\|_{L_t^{\infty}L_x^{\infty}}^{p-1}\|u_{tt}\|_{L_t^{\infty}H_x^{s-4}}+\|u\|_{L_t^{\infty}L_x^{r_1}}^{p-2}\|u\|_{L_t^{\infty}W_x^{s-4,r_2}}\|u_{tt}\|_{L_t^{\infty}L_x^{r_3}}\lesssim \|u\|_{X^{s_0}}^{p-1}\|u\|_{X^{s}},
\end{equation*}
where above, we took $\frac{1}{2}=\frac{1}{r_1}+\frac{1}{r_2}+\frac{1}{r_3}$ with $r_1,r_2$ sufficiently large and $r_3$ close to $2$ when $s>4$ and $r_1=r_2=\infty$ and $r_3=2$ when $s=4$. The result then followed by Sobolev embedding.
\medskip

\paragraph{\textit{Estimates for} $\|D_x^{s-6}\partial_tH\|_{S_+'}$} We now analyze the remaining term in \Cref{items1}; namely, the expression $\|D_x^{s-6}\partial_tH\|_{S_+'}$. Here, we may assume that $s\geq 6$. 
Given the definition of $H$, two of the terms that we will need to control take the form
\begin{equation}\label{2termss>6}
\|D_x^{s-6}(|u|^{p-3}u_tu_tu_t)\|_{S'_+}+\|D_x^{s-6}(|u|^{p-2}u_tu_{tt})\|_{S'_+}.
\end{equation}
We will only show how to estimate the expressions in \eqref{2termss>6} as the actual terms in the expansion of $H$ have a similar form and can be handled by identical reasoning.
\medskip

\textbf{Case 1: $p>4$.} We begin by analyzing the first term in \eqref{2termss>6}. If $d\leq 11$ (which implies that $H^s\subset H^6_x\subset L_x^{\infty}$), we can work in the (almost) dual norm $L_t^{1+}L_x^2$. In this case, we can use  \Cref{Crude Moser est}, \Cref{Leib1} and Sobolev embeddings to crudely estimate
\begin{equation*}
\begin{split}
\|D_x^{s-6}(|u|^{p-3}u_tu_tu_t)\|_{L^{1+}_tL_x^{2}}&\lesssim_{\epsilon} \|u\|^{p-3}_{L^\infty_tH^{s-\epsilon}_x}\|u_t\|_{L_t^\infty L_x^{6-}}^3+\|u\|_{L_t^{\infty}L_x^{\infty}}^{p-3}\|u_t\|_{L_t^{\infty}L_x^{2r}}^2\|D_x^{s-6}u_t\|_{L_t^{\infty}L_x^{r'}}
\end{split}
\end{equation*}
for any $0<\epsilon\ll 1$. Here, $r,r'$ satisfy $\frac{1}{2}=\frac{1}{r}+\frac{1}{r'}$. We can take $r'=\frac{2d}{d-8}-$ if $11\geq d\geq 8$ and $r'$ arbitrarily close to $\infty$  otherwise. In either case (since $s\geq 6$), Sobolev embeddings yield
\begin{equation*}
\|D_x^{s-6}(|u|^{p-3}u_tu_tu_t)\|_{L^{1+}_tL_x^{2}}\lesssim \|u\|_{X^{s_0}}^{p-1}\|u\|_{X^s}
\end{equation*}
for some $s_0<s$. On the other hand, if $d\geq 12$, we will use the norm $L_t^{2}L_x^{\frac{2d}{d+2}+}$ to estimate the above term. This corresponds (from our convention for $ S_+ ' $) to a dual norm of the form $L_t^{q'}L_x^{r'}$ where $(q',r')$ is admissible and $q'<2$ is sufficiently close to $2$. We then compute that
\begin{equation*}
\begin{split}
\|D_x^{s-6}(|u|^{p-3}u_tu_tu_t)\|_{L^2_tL_x^{\frac{2d}{d+2}+}} &\lesssim \|u\|^{p-3}_{L_t^{\infty}L_x^{\frac{d}{2}(p-1)+}}\|u_t\|^2_{L_t^{\infty}L_x^{\frac{d(p-1)}{2p}}}\|D_x^{s-6}u_t\|_{L_t^2L_x^{\frac{2d}{d-10}}}
\\
&+\|u\|^{p-4}_{L^\infty_tL_x^{\frac{d}{2}(p-1)+}}\|D_x^{s-6}u\|_{L^\infty_tL_x^{\frac{2d}{d-12}-}}\|u_t\|_{L_t^6L_x^{\frac{3d(p-1)}{5p+1}}}^3
\\
&\lesssim \|u\|^{p-3}_{L_t^{\infty}L_x^{\frac{d}{2}(p-1)+}}\|u_t\|^2_{L_t^{\infty}L_x^{\frac{d(p-1)}{2p}}}\|D_x^{s-6}u_t\|_{L_t^2L_x^{\frac{2d}{d-10}}}
\\
&+\|u\|^{p-4}_{L^\infty_tL_x^{\frac{d}{2}(p-1)+}}\|D_x^{s-6}u\|_{L^\infty_tL_x^{\frac{2d}{d-12}-}}\|u\|^\frac{3}{2}_{L^\infty_tL_x^{\frac{1}{2}d(p-1)}}\|u_{tt}\|^{\frac{3}{2}}_{L^3_tL_x^\frac{3d(p-1)}{10p-4}}
\\
&\lesssim \|u\|_{X^{s_0}}^{p-1}\|u\|_{X^s},
\end{split}
\end{equation*}
for some $s_c<s_0<s$. Note that in the first inequality above we used Propositions~\ref{Crude Moser est} and \ref{Leib1}. In the second estimate, we interpolated the third factor in the second line. Then, in the final estimate, we used the embeddings $W_x^{s_c-4,\frac{6d}{3d-4}}\subset L_x^{\frac{3d(p-1)}{10p-4}}$, $H_x^{s_c}\subset L_x^{\frac{d}{2}(p-1)}$, $H_x^{s_c-2}\subset L_x^{\frac{d}{2p}(p-1)}$, $W_x^{4,\frac{2d}{d-2}}\subset L_x^{\frac{2d}{d-10}}$ and $H_x^6\subset L_x^{\frac{2d}{d-12}}$ together with the fact that $(3,\frac{6d}{3d-4})$ is an admissible Strichartz pair.
\medskip

Next, we analyze the second term in \eqref{2termss>6}. Similarly to the above, if $d\leq 11$ we have (noting the crude embedding $H_x^{s-\epsilon}\subset L^q$ for any $q\geq 2$),
\begin{equation*}
\begin{split}
\|D_x^{s-6}(|u|^{p-2} u_tu_{tt})\|_{L^{1+}_t L_x^2} &\lesssim \|u\|_{L_t^{\infty}H_x^{s-\epsilon}}^{p-2}\|u_{tt}\|_{L_t^{\infty}L_x^{r}}\|u_t\|_{L_t^{\infty}L_x^{r'}}
\\
&+\|u\|_{L_t^{\infty}L_x^{\infty}}^{p-2}\|u_t\|_{L_t^{\infty}L_x^{r'}}\|D_x^{s-6}u_{tt}\|_{L_t^{\infty}L_x^{r}}
\\
&+\|u\|_{L_t^{\infty}L_x^{\infty}}^{p-2}\|u_{tt}\|_{L_t^{\infty}L_x^{r}}\|D_x^{s-6}u_t\|_{L_t^{\infty}L_x^{r'}}
\end{split}
\end{equation*}
for any $\frac{1}{2}=\frac{1}{r}+\frac{1}{r'}$ with $2<r,r'<\infty$. Since $d\leq 11$, we can choose $r$ such that $H_x^{2-\epsilon}\subset L_x^{r}$ and $H_x^{4-\epsilon}\subset L_x^{r'}$.  This yields (since $s\geq 6$),
\begin{equation*}
\|D_x^{s-6}(|u|^{p-2} u_tu_{tt})\|_{L^{1+}_t L_x^2}\lesssim \|u\|^{p-1}_{X^{s_0}}\|u\|_{X^s}.
\end{equation*}
If $d\geq 12$, we again estimate this term in the space $L_t^{2}L_x^{\frac{2d}{d+2}+}$. Here we have
\begin{equation*}
\begin{split}
\|D_x^{s-6}(|u|^{p-2}u_tu_{tt})\|_{L_t^{2}L_x^{\frac{2d}{d+2}+}}&\lesssim \|u\|^{p-3}_{L_t^{\infty}L_x^{\frac{d}{2}(p-1)+}}\|D_x^{s-6}u\|_{L_t^{\infty}L_x^{\frac{2d}{d-12}-}}\|u_t\|_{L_t^{\infty}L_x^{\frac{d(p-1)}{2p}}}\|u_{tt}\|_{L_t^2L_x^{\frac{d(p-1)}{3p-1}}}
\\
&+\|u\|_{L_t^{\infty}L_x^{\frac{d}{2}(p-1)+}}^{p-2}\|D_x^{s-6}u_t\|_{L_t^{\infty}L_x^{\frac{2d}{d-8}}}\|u_{tt}\|_{L_t^2L_x^{\frac{d(p-1)}{3p-1}}}
\\
&+\|u\|_{L_t^{\infty}L_x^{\frac{d}{2}(p-1)+}}^{p-2}\|u_t\|_{L_t^{\infty}L_x^{\frac{d(p-1)}{2p}}}\|D_x^{s-6}u_{tt}\|_{L_t^2L_x^{\frac{2d}{d-6}}},
\end{split}
\end{equation*}
which by Sobolev embeddings yields
\begin{equation*}
\|D_x^{s-6}(|u|^{p-2}u_tu_{tt})\|_{L_t^2L_x^{\frac{2d}{d+2}}}\lesssim \|u\|_{X^{s_0}}^{p-1}\|u\|_{X^s}.
\end{equation*}

\textbf{Case 2: $p\leq 4$.} Again, we recall that we are assuming that $s\geq 6$, which forces $p>\frac{7}{2}$. Under these assumptions, we have $\frac{s-6}{p-3}<\frac{p-\frac{7}{2}}{p-3}\leq \frac{1}{2}$. If $d\leq 11$, then fractional Leibniz and \Cref{KV Prop} gives
\begin{equation*}
\begin{split}
\|D_x^{s-6}(|u|^{p-3}u_tu_tu_t)\|_{L^{1+}_tL^2_x}&\lesssim \|u\|^{p-3}_{L_t^{\infty}W_x^{\frac{1}{2},q}}\|u_t\|_{L_t^{\infty}L_x^{6+}}^3+\|u\|^{p-3}_{L_t^{\infty}L_x^{\infty}}\|u_t\|^2_{L_t^{\infty}L_x^{2r}}\|D_x^{s-6}u_t\|_{L_t^{\infty}L_x^{r'}},
\end{split}
\end{equation*}
where $\frac{1}{2}=\frac{1}{r}+\frac{1}{r'}$ and $q\gg 1$ is sufficiently large. As $d\leq 11$, we can choose $r, r'$ such that $H^{4-\epsilon}\subset L_x^{r'}$ and $H^{2-\epsilon}\subset L_x^{2r}$ and then use the embedding $H^{\frac{11}{2}}_x\subset L^{q}$ to obtain
\begin{equation*}
\|D_x^{s-6}(|u|^{p-3}u_tu_tu_t)\|_{L^{1+}_tL^2_x}\lesssim \|u\|_{X^{s_0}}^{p-1}\|u\|_{X^{s}}.
\end{equation*}
Finally, if $d\geq 12$ we can estimate
\begin{equation}\label{s6bound}
\begin{split}
\|D_x^{s-6}(|u|^{p-3}u_tu_tu_t)\|_{L^2_tL^{\frac{2d}{d+2}+}_x}\lesssim &\|u\|^{p-3}_{L_t^{\infty}W_x^{\frac{s-6}{p-3},(p-3)r}}\|u_t\|_{L_t^{6}L_x^{3r'}}^3
\\
+&\|u\|^{p-3}_{L_t^{\infty}L_x^{\frac{d}{2}(p-1)+}}\|u_t\|^2_{L_t^{\infty}L_x^{\frac{d(p-1)}{2p}}}\|D_x^{s-6}u_t\|_{L_t^{2}L_x^{\frac{2d}{d-10}}},
\end{split}
\end{equation}
where we choose $r$ such that $2\leq (p-3)r<\infty$, $L_x^{(p-3)r}\subset H_x^{s_c-\frac{s-6}{p-3}+}$ and $L_x^{3r'}\subset W_x^{s-2-,\frac{6d}{3d-2}}$ (observing that $(6,\frac{6d}{3d-2})$ is admissible). Sobolev embeddings then yield a good bound for the terms on the right-hand side of \eqref{s6bound} in this case as well.
\medskip 

\paragraph{\textit{Estimates for $\|D_t^{\frac{s}{2}-2}H\|_{S_+'}$}} Finally, we deal with the remaining term in \Cref{items2}. The objective is to obtain an estimate for $\|D_t^{\frac{s}{2}-2}H\|_{S_+'}$ when $4\leq s<6$. Let us first assume that $\frac{3}{2}<p\leq \frac{s}{2}<3$. As usual, we need to perform a somewhat tedious case analysis.
\medskip

\textbf{Case 1: $d>6$.} First, let us assume that $d>6$. Here, we Sobolev embed in time, apply Minkowski's inequality and then use the variation of the nonlinear estimate in \Cref{nonlinearvariation} to bound
\begin{equation*}\label{edgecase}
\begin{split}
\|D_t^{\frac{s}{2}-2}H\|_{L_t^{2+}L_x^{\frac{2d}{d+2}}}&\lesssim \|D_t^{\frac{s}{2}-2+}H\|_{L_t^{2}L_x^{\frac{2d}{d+2}}}\lesssim \bigg\|\|u\|_{H_t^1}^{p-\theta}\|u\|_{H^2_t}^{\theta}\bigg\|_{L_x^{\frac{2d}{d+2}}},
\end{split}
\end{equation*}
where $\theta:=\frac{s-p-1}{2}+\epsilon$ and $0<\epsilon\ll 1$. We observe that due to the restriction $s<p+\frac{5}{2}$, we have $0<\theta<1$.
\medskip

\textbf{Case 1.1: $s_c>2$.} If we further have $d\geq 10$, then by H\"older's inequality, Minkowski's inequality and Sobolev embeddings, the following estimate holds: 
\begin{equation*}
\begin{split}
\bigg\|\|u\|_{H_t^1}^{p-\theta}\|u\|_{H^2_t}^{\theta}\bigg\|_{L_x^{\frac{2d}{d+2}}}&\lesssim \|u\|_{H_t^1L_x^{\frac{d(p-1)}{p+1}}}^{p-1}\|u\|_{H_t^1L_x^{\frac{2d}{d-2s+2}}}^{1-\theta}\|u\|_{H_t^2L_x^{\frac{2d}{d-2s+6}}}^{\theta}
\\
&\lesssim\|u\|_{H_t^1L_x^{\frac{d(p-1)}{p+1}}}^{p-1}\|u\|_{H_t^2W_x^{s-4,\frac{2d}{d-2}}}^{\theta}\|u\|_{H_t^1W_x^{s-2,\frac{2d}{d-2}}}^{1-\theta}.
\end{split}
\end{equation*}
Since $s_c>2$, we have $\frac{d(p-1)}{p+1}>\frac{2d}{d-2}$ and thus $\|u\|_{H_t^1L_x^{\frac{d(p-1)}{p+1}+}}\lesssim \|u\|_{H_t^1W_x^{s_c-2+,\frac{2d}{d-2}}}$, which yields
\begin{equation*}
\bigg\|\|u\|_{H_t^1}^{p-\theta}\|u\|_{H^2_t}^{\theta}\bigg\|_{L_x^{\frac{2d}{d+2}}}\lesssim \|u\|_{X^{s_0}}^{p-1}\|u\|_{X^s}
\end{equation*}
for some $s_c<s_0<s$. If $6<d\leq 9$, we estimate crudely
\begin{equation*}
\begin{split}
\bigg\|\|u\|_{H_t^1}^{p-\theta}\|u\|_{H^2_t}^{\theta}\bigg\|_{L_x^{\frac{2d}{d+2}}}\lesssim \bigg\|\|u\|_{H_t^1}^{p-1}\|u\|_{H^2_t}\bigg\|_{L_x^{\frac{2d}{d+2}}}\lesssim \|u\|_{H_t^1L_x^{\frac{d(p-1)}{s-2}}}^{p-1}\|u\|_{H_t^2L_x^{\frac{2d}{d-2s+6}}}.
\end{split}
\end{equation*}
The latter factor can be estimated as before. On the other hand, since $s_c>2$ and $d\leq 9$, we must also have $p>\frac{9}{5}$. Combining this with the hypothesis $d> 6$ and $2p\leq s<p+\frac{5}{2}$, we see that $2\leq \frac{d}{p+\frac{1}{2}}(p-1)\leq\frac{d}{s-2}(p-1)\leq \frac{d}{2}$. Since $d\leq 9$, we have the embedding $W_x^{2,\frac{2d}{d-2}}\subset L_x^{\frac{d}{2}}$, and thus, we can crudely estimate by interpolation
\begin{equation*}
\|u\|_{H_t^1L_x^{\frac{d(p-1)}{s-2}}}\lesssim \|u\|_{H_t^1W_x^{2,\frac{2d}{d-2}}\cap W_t^{1,\infty}L_x^2},
\end{equation*}
which suffices. 
\medskip

\textbf{Case 1.2: $s_c\leq 2$.} On the other hand, if $s_c\leq 2$, then by H\"older we have, for every $2\leq r\leq\frac{2d}{d-2}$, the crude estimate
\begin{equation*}
\begin{split}
\bigg\|\|u\|_{H_t^1}^{p-\theta}\|u\|_{H^2_t}^{\theta}\bigg\|_{L_x^{\frac{2d}{d+2}}}\lesssim \bigg\|\|u\|_{H_t^1}^{p-1}\|u\|_{H^2_t}\bigg\|_{L_x^{\frac{2d}{d+2}}}\lesssim \|u\|_{H_t^1L_x^{\frac{r}{r-1}(p-1)}}^{p-1}\|u\|_{H_t^2L_x^r}.
\end{split}
\end{equation*}
If $0\leq s_c\leq 2$ we can take $r=\frac{2d}{d-2}$, which gives by H\"older in time and Sobolev embedding,
\begin{equation*}
\|u\|_{H_t^1L_x^{\frac{r}{r-1}(p-1)}}^{p-1}\|u\|_{H_t^2L_x^{\frac{2d}{d-2}}}\lesssim \|u\|_{W^{1,\infty}_tH_x^{s_c}}^{p-1}\|u\|_{H_t^2L_x^{\frac{2d}{d-2}}}\lesssim \|u\|_{X^{s_0}}^{p-1}\|u\|_{X^s}.
\end{equation*}
On the other hand, if $s_c<0$ then since $p>\frac{3}{2}$ and $d>6$, we have $\frac{d}{2}(p-1)<2<(p-1)d$ and, therefore, we can choose $r$ such that $\frac{r}{r-1}(p-1)=2$. This gives by H\"older in time,
\begin{equation*}
\|u\|_{H_t^1L_x^{\frac{r}{r-1}(p-1)}}^{p-1}\|u\|_{H_t^2L_x^{r}}\lesssim \|u\|_{W^{1,\infty}_tL_x^2}^{p-1}\|u\|_{W_t^{2,q}L_x^{r}}\lesssim \|u\|_{X^{s_0}}^{p-1}\|u\|_{X^s},
\end{equation*}
where $q>2$ is such that $(q,r)$ is admissible.
\medskip

\textbf{Case 2: $d\leq 6$.}
Next, we turn to the case $d\leq 6$. Here we can estimate more crudely using the dual norm $L_t^{1}L_x^2$. First, if $1\leq d\leq 4$, we can  use \Cref{nonlinearvariation} to bound
\begin{equation*}
\|D_t^{\frac{s}{2}-2+}H\|_{L_t^{1+}L_x^{2}}\lesssim \bigg\|\|u\|_{H_t^1}^{p-\theta}\|u\|_{H^2_t}^{\theta}\bigg\|_{L_x^{2}}\lesssim \|u\|_{H_t^1L_x^{\infty-}}^{p-1}\|u\|_{H_t^2L_x^{2+}}\lesssim \|u\|_{X^{s_0}}^{p-1}\|u\|_{X^s},
\end{equation*}
where we used the embedding $H^{4-\epsilon}\subset L_x^{\infty-}$. If $5\leq d\leq 6$ then we instead have 
\begin{equation*}
\|D_t^{\frac{s}{2}-2+}H\|_{L_t^{1+}L_x^{2}}\lesssim \bigg\|\|u\|_{H_t^1}^{p-\theta}\|u\|_{H^2_t}^{\theta}\bigg\|_{L_x^{2}}\lesssim \|u\|_{H_t^1L_x^{(p-1)d}}^{p-1}\|u\|_{H_t^2L_x^{\frac{2d}{d-2}}}\lesssim \|u\|_{X^{s_0}}^{p-1}\|u\|_{X^s},
\end{equation*}
where we used that $2<\frac{d}{2}\leq (p-1)d<2d$ and also the estimate $\|u\|_{H_t^1L_x^{(p-1)d}}\lesssim \|u\|_{H_t^1W_x^{\frac{3}{2},\frac{2d}{d-2}}}$. This concludes the analysis in the case $\frac{s}{2}\geq p$. 
\medskip

It now remains to consider the case $\frac{s}{2}<p$. Here, the number of derivatives applied to the nonlinearity does not exceed the H\"older regularity of the function $z\mapsto |z|^{p-1}z$. Therefore, we can proceed using more standard Moser-type estimates. For this, it is convenient to write $H$ in the form
\begin{equation*}
H=\frac{p+1}{2}\partial_t|u|^{p-1}u_t+\frac{p-1}{2}\partial_t(|u|^{p-3}u^2)\overline{u}_t.
\end{equation*}
We will show the details for estimating the first term, as the latter follows almost identical reasoning. When $d\geq 3$, we will use the dual norm $L_t^2L_x^{\frac{2d}{d+2}}$. For any $\frac{2d}{d+2}<r_1,r_2,r_3<\infty$ satisfying $\frac{1}{r_1}+\frac{1}{r_2}+\frac{1}{r_3}=\frac{d+2}{2d}$, the fractional Leibniz rule combined with the vector-valued Moser bound in \Cref{Moservec} and Sobolev embedding in $t$ yields the simple bound
\begin{equation*}
\begin{split}
\|D_t^{\frac{s}{2}-2+}(\partial_t|u|^{p-1}u_t)\|_{L_t^{2+}L_x^{\frac{2d}{d+2}}}&\lesssim \|D_t^{\frac{s}{2}-2+}\partial_t|u|^{p-1}\|_{L_t^{2+}L_x^{\frac{2dr_1}{d(r_1-2)+2r_1}}}\|u_t\|_{L_t^{\infty}L_x^{r_1}}
\\
&+\|u_t\|_{L_t^{\infty}L_x^{r_1}}\|u\|_{H^{\frac{s}{2}-1+}_tL_x^{r_2}}\|u\|^{p-2}_{L_t^{\infty}L_x^{(p-2)r_3}}
\\
&\lesssim \|u_t\|_{L_t^{\infty}L_x^{r_1}}\|u\|_{H^{\frac{s}{2}-1+}_tL_x^{r_2}}\|u\|^{p-2}_{L_t^{\infty}L_x^{(p-2)r_3}}.
\end{split}
\end{equation*}
It is straightforward to verify (as in the case analysis of many of the previous estimates) that one can choose $r_1,r_2,r_3$ such that we have
\begin{equation*}
\|D_t^{\frac{s}{2}-2+}(\partial_t|u|^{p-1}u_t)\|_{L_t^{2+}L_x^{\frac{2d}{d+2}}}\lesssim \|u\|_{X^{s_0}}^{p-1}\|u\|_{X^s}.
\end{equation*}
Note that in the last estimate we interpolated
\begin{equation*}
\|u\|_{H_t^{\frac{s}{2}-1+}W_x^{2-,\frac{2d}{d-2}}}\lesssim \|u\|_{H_t^1W_x^{s-2,\frac{2d}{d-2}}}+\|u\|_{H_t^2W_x^{s-4,\frac{2d}{d-2}}}\lesssim \|u\|_{X^s}.
\end{equation*}
When $1\leq d\leq 2$, one can simply carry out a similar (and even more straightforward) estimate for $\|D_t^{\frac{s}{2}-2+}(\partial_t|u|^{p-1}u_t)\|_{L_t^{1+}L_x^{2}}$. We omit the details as the analysis is again similar to many of the previous estimates in this section.
\subsection{Proof of well-posedness for the nonlinear Schr\"odinger equation}
In this subsection, we aim to establish \Cref{M1 OOTP}. Below, we will restrict our attention to the case $s>2$ since the remaining cases are already known (see \cite{MR3546788,MR4581790,MR3917711} and references therein). 
\subsubsection{Difference bounds for the nonlinear Schr\"odinger equation} As in the case of the nonlinear heat equation \eqref{NLH}, to construct solutions and ultimately obtain continuity of the data-to solution-map in $H^s(\mathbb{R}^d)$, we will need good difference bounds for \eqref{NLS} in a weaker topology on the time interval $[-T,T]$. The following lemma is standard, but we include the proof for completeness.
\begin{lemma}\label{NLSdiff}
Let $u_1$ and $u_2$ be $C([-T,T]; H_x^{s}(\mathbb{R}^d))$ solutions to the nonlinear Schr\"odinger equation \eqref{NLS} for some time $T>0$. Assume that $s>\max\{2,s_c\}$. For every $s_0$ satisfying $\max\{2,s_c\}<s_0\leq s$, we can choose $T$ depending only the size of $u_1$ and $u_2$ in $L_T^{\infty}H^{s_0}_x(\mathbb{R}^d)$ (i.e.~the weaker norm) such that
\begin{equation*}
\|u_1-u_2\|_{S}\lesssim \|u_1(0)-u_2(0)\|_{L_x^2}
\end{equation*}
where we write $S:=L_T^{\infty}L_x^2$ if $1\leq d\leq 2$ and $S:=L_T^{\infty}L_x^2\cap L_T^2L_x^{\frac{2d}{d-2}}$ if $d\geq 3$.
\end{lemma}
\begin{proof}
When $1\leq d\leq 4$, Strichartz estimates, H\"older in $T$ and the Sobolev embedding $H^{s_0}\subset L^{\infty}$ yield
\begin{equation*}
\begin{split}
\|u_1-u_2\|_{S}&\lesssim \|u_1(0)-u_2(0)\|_{L_x^2}+T\||u_1|^{p-1}u_1-|u_2|^{p-1}u_2\|_{L_T^{\infty}L_x^2}
\\
&\lesssim\|u_1(0)-u_2(0)\|_{L_x^2}+T\|(u_1,u_2)\|_{L_T^{\infty}H_x^{s_0}}^{p-1}\|u_1-u_2\|_{L_T^{\infty}L_x^2}.
\end{split}
\end{equation*}
On the other hand, if $d\geq 5$, Strichartz estimates  yield
\begin{equation*}
\|u_1-u_2\|_{S}\lesssim \|u_1(0)-u_2(0)\|_{L_x^2}+\||u_1|^{p-1}u_1-|u_2|^{p-1}u_2\|_{S'},
\end{equation*}
where $S'$ is some dual Strichartz norm. If $s_c\geq 0$, we take $S'=L_T^{2-}L_x^{\frac{2d}{d+2}+}$ where $(2-,\frac{2d}{d+2}+)$ is a dual admissible pair with time exponent slightly less than $2$. We then estimate the latter term by 
\begin{equation*}
\begin{split}
\||u_1|^{p-1}u_1-|u_2|^{p-1}u_2\|_{L_T^{2-}L_x^{\frac{2d}{d+2}+}}&\lesssim T^{\delta}\||u_1|^{p-1}u_1-|u_2|^{p-1}u_2\|_{L_T^{2}L_x^{\frac{2d}{d+2}+}}
\\
&\lesssim T^{\delta}\|(u_1,u_2)\|_{L_T^{\infty}L_x^{\frac{d}{2}(p-1)+}}^{p-1}\|u_1-u_2\|_{L_T^2L_x^{\frac{2d}{d-2}}}
\\
&\lesssim T^{\delta}\|(u_1,u_2)\|_{L_T^{\infty}H_x^{s_c+}}^{p-1}\|u_1-u_2\|_{L_T^2L_x^{\frac{2d}{d-2}}}
\end{split}
\end{equation*}
for some $\delta>0$. If $s_c<0$, we must also have  (since $d\geq 5$) the restriction $1<p<2$. If we further have $\frac{2}{d}+1<p< \frac{4}{d}+1$ then $r=\frac{2d}{d(2-p)+2}$ satisfies $2\leq r\leq \frac{2d}{d-2}$. In this case, we have
\begin{equation*}
\||u_1|^{p-1}u_1-|u_2|^{p-1}u_2\|_{L_T^{2}L_x^{\frac{2d}{d+2}}}\lesssim \|u\|^{p-1}_{L_T^{\infty}L_x^{2}}\|u_1-u_2\|_{L_T^2L_x^r}\lesssim T^{\delta}\|u_1-u_2\|_{S}.
\end{equation*}
If $p\leq\frac{2}{d}+1$ then the dual exponent $\frac{2}{2-p}$ of $\frac{2}{p}$ is admissible so that if $q>2$ is the corresponding time exponent, we have 
\begin{equation*}
\begin{split}
\||u_1|^{p-1}u_1-|u_2|^{p-1}u_2\|_{L_T^{q'}L_x^{\frac{2}{p}}}&\lesssim \|(u_1,u_2)\|_{L_T^{\infty}L_x^2}^{p-1}\|u_1-u_2\|_{L_T^{q'}L_x^2}
\\
&\lesssim T^{\delta}\|(u_1,u_2)\|^{p-1}_{L_T^{\infty}L_x^{2}}\|u_1-u_2\|_{S}
\end{split}
\end{equation*}
for some $\delta>0$. The proof is then concluded by taking $T$ small enough.
\end{proof}
\subsubsection{Local well-posedness} Now, we have all of the necessary ingredients to obtain the local well-posedness result in \Cref{M1 OOTP}. We begin by establishing existence. First, we note that by restricting the time interval, it suffices to establish the analogous result for the time-truncated equation for some $0<T<1$ sufficiently small. To this end, we  let $u_0\in H^s(\mathbb{R}^d)$ where $p+\frac{5}{2}>s>\max\{2,s_c\}$. For simplicity of notation, we write $M\coloneq \|u_0\|_{H^s}$ and let $T>0$ be a sufficiently small parameter to be chosen.
\medskip

As a starting point, we follow the scheme used in the proof of well-posedness for the nonlinear heat equation. We begin with a straightforward application of the contraction mapping theorem, which enables us to construct for each $l,j\in\mathbb{N}$ (by only applying space derivatives and using Bernstein's inequality in the spatial variable to estimate the nonlinear term) solutions $u_j^l\in C([-T_j,T_j];H_x^{s+\epsilon}(\mathbb{R}^d))$  to the regularized and time-truncated equation,
\begin{equation*}
\begin{cases}
&(i\partial_t-\Delta)u_j^l=(P_{<j}(|u_j^l|^{p-1}u^l_j))_{T_j},
\\
&u_j^l(0)=P_{<l}u_0.
\end{cases}
\end{equation*}
Note that here we are using the notation from \Cref{truncationestimates}. As in the case of the nonlinear heat equation, a priori, the length of the time interval $[-T_j,T_j]\subset [-T,T]$ depends on $j$. However, from the estimates in \Cref{apriorischrodinger}, we have the uniform in $j$ and $l$ bound,
\begin{equation*}
\|u_j^l\|_{L^{\infty}_t([-T_j,T_j];H_x^{s+\epsilon}(\mathbb{R}^d))}\lesssim_M \|P_{<l}u_0\|_{H^{s+\epsilon}_x(\mathbb{R}^d))}
\end{equation*}
on $[-T_j,T_j]$ for every $\epsilon\geq 0$ sufficiently small. We can use this uniform bound and Bernstein's inequality in $j$ to iterate the contraction mapping theorem to obtain a smooth global solution $u_j^l$ to a time-truncated equation with a wider truncation scale. Iterating this process (and restricting the time interval so that the truncated nonlinearity agrees with $P_{<j}(|u_j^l|^{p-1}u_j^l$)), we obtain a solution $u_j^l$ to the regularized nonlinear Schr\"odinger equation (without truncation),
\begin{equation*}\label{regNLS}
\begin{cases}
&(i\partial_t-\Delta)u_j^l=P_{<j}(|u_j^l|^{p-1}u_j^l),
\\
&u_j^l(0)=P_{<l}u_0,
\end{cases}
\end{equation*}
 on some time interval $[-T,T]$ whose length depends only on $M$ and where $u_j^l$ satisfies the uniform bound
\begin{equation*}\label{ujregbound}
\|u_j^l\|_{L^{\infty}_t([-T,T];H_x^{s+\epsilon}(\mathbb{R}^d))}\lesssim_M \|P_{<l}u_0\|_{H^{s+\epsilon}_x(\mathbb{R}^d)}.
\end{equation*}
Our next aim is to show that for fixed $l$, the sequence $u_j^l$ converges in $C([-T,T];H_x^{s})$ as $j\to\infty$ to a solution $u^l\in C([-T,T];H_x^{s+\epsilon})$ of the original equation \eqref{NLS}, with the regularized initial data $u^l(0)=P_{<l}u_0$. To do this, we write the equation for the difference of two solutions $u^l_j$ and $u^l_k$ for $k<j$ as
\begin{equation*}
\begin{cases}
&(i\partial_t-\Delta)(u_j^l-u_k^l)=P_{<j}(|u_j^l|^{p-1}u_j^l-|u_k^l|^{p-1}u_k^l)+P_{k\leq \cdot <j}(|u_k^l|^{p-1}u_k^l),
\\
&(u_j^l-u_k^l)(0)=0.
\end{cases}
\end{equation*}
First, we note that, as in \Cref{NLSdiff} and the argument for the \eqref{NLH} equation, one can easily show that
\begin{equation*}
\|u_j^l-u_k^l\|_{L_t^{\infty}([-T,T]; L_x^2)}\to 0
\end{equation*}
as $j,k\to\infty$. It follows that (after possibly relabeling the $\epsilon$ above) for every $\epsilon\geq 0$ sufficiently small,  the sequence $u_j^l$ converges in $C([-T,T];H_x^{s+\epsilon})$ to some $u^l\in C([-T,T];H_x^{s+\epsilon})$ solving the time-truncated equation \eqref{truncated NLS} with initial data $P_{<l}u_0$,
\begin{equation*}
\begin{cases}
&(i\partial_t-\Delta)u^l=|u^l|^{p-1}u^l,
\\
&u^l(0)=P_{<l}u_0,
\end{cases}
\end{equation*}
which also satisfies the uniform bound,
\begin{equation}\label{uniflbound}
\|u^l\|_{L_T^{\infty}H_x^{s+\epsilon}}\lesssim_M \|P_{<l}u_0\|_{H_x^{s+\epsilon}}.
\end{equation}
 At this point, we aim to pass to the limit in $l$ to obtain the existence part of our well-posedness theorem. To execute this, we will need to appeal to the language of frequency envelopes. 
\medskip

From here on, we let $c_l$ be a $H^s$ frequency envelope for the initial data $u_0$ in $H^s$. From the bound \eqref{uniflbound} and \Cref{NLSdiff}, we obtain the bounds
\begin{itemize}
\item (Higher regularity bound). For all $0<\epsilon\ll 1$ sufficiently small, there holds
\begin{equation*}
\begin{split}
\|u^l\|_{L_T^{\infty}H_x^{s+\epsilon}}\lesssim_M c_l2^{l\epsilon}.
\end{split}
\end{equation*}
\item (Uniform $H^s$ bound).
\begin{equation*}
\|u^l\|_{L_T^{\infty}H_x^{s}}\lesssim_M 1.
\end{equation*}
\item ($L^2$ difference bound).
\begin{equation*}
\|u^{l+1}-u^l\|_{L_T^{\infty}L_x^2}\lesssim_M c_l2^{-ls}.
\end{equation*}
\end{itemize}
From these bounds and a standard telescoping estimate (see for instance \cite{primer}), there is a function $u\in C([-T,T];H_x^s)$ such that $u^l\to u$ in $C([-T,T];H_x^s)$ as $l\to\infty$ with the error estimate
\begin{equation*}
\|u^l-u\|_{L_T^{\infty}H_x^s}\lesssim \|(c_j)_{j\geq l}\|_{l^2}.
\end{equation*}
This gives the required existence result. We also note that uniqueness follows immediately from the difference bound in \Cref{NLSdiff}. 
\medskip

It remains to establish continuous dependence. Let $\delta>0$. We let $u_0^n$ be a sequence in $H^s$ such that $u_0^n\to u_0$ in $H^s$ with $\|u_0\|_{H_x^s}\lesssim_M 1$. We then let $u^n$ and $u$ be the corresponding $C([-T,T];H^s)$  solutions to \eqref{NLS} (after possibly slightly restricting the time interval) which satisfy the uniform $H^s$ bounds
\begin{equation*}
\|u^n\|_{L_T^{\infty}H_x^{s}}+\|u\|_{L_T^{\infty}H_x^{s}}\lesssim_M 1.
\end{equation*}
Given $l\in\mathbb{N}$, we let $u^{n,l}$ and $u^l$ be the solutions corresponding to the regularized data $P_{<l}u_0$. On the one hand, by the telescoping argument above, we have
\begin{equation}\label{lfreqenvbound}
\|u^l-u\|_{L_T^{\infty}H_x^s}+\|u^{n,l}-u^n\|_{L_T^{\infty}H_x^s}\lesssim \|(c_j)_{j\geq l}\|_{l^2},
\end{equation}
which is a bound which is uniform in $n$ (for $n$ large enough). On the other hand, we can estimate $\|u^{n,l}-u^l\|_{L_T^{\infty}H_x^s}$ by interpolating between the $L_T^{\infty}H_x^{s+\epsilon}$ and $L^2$ bounds for the regularized solutions $u^l$ and $u^{n,l}$ to obtain 
\begin{equation}\label{nlbound}
\|u^{n,l}-u^l\|_{L_T^{\infty}H_x^s}\lesssim \|u^{n,l}-u^l\|_{L_T^{\infty}L^2_x}^{\gamma}\|u^{n,l}-u^l\|_{L_T^{\infty}H_x^{s+\epsilon}}^{1-\gamma}\lesssim_{M,l} \|u^{n}_0-u_0\|_{L_x^2}^{\gamma}
\end{equation}
for some $0<\gamma(\epsilon)<1$. Importantly, the implicit constant in the last estimate above depends only on $M$ and $l$ and not on $n$. To conclude, we then take $l$ sufficiently large so that
\begin{equation*}
\|(c_j)_{j\geq l}\|_{l^2}\lesssim\delta
\end{equation*}
and then take $n=n(l)$ large enough so that
\begin{equation*}
\|u^{n}_0-u_0\|_{L_x^2}^{\gamma}\lesssim \delta.
\end{equation*}
In view of \eqref{lfreqenvbound} and \eqref{nlbound}, we may therefore conclude that
\begin{equation*}
\|u^n-u\|_{L_T^{\infty}H_x^s}\lesssim \delta.
\end{equation*}
This establishes continuous dependence and thus also \Cref{M1 OOTP}.
\subsubsection{Improved well-posedness for one-dimensional Schr\"odinger equations}
Finally, we demonstrate that the low modulation threshold in \Cref{M1 OOTP} can be improved in certain cases by better exploiting the dispersive properties of \eqref{NLS}. This will complete the proof of \Cref{NLS improved}.
\begin{theorem}
Let $p>1$. Then \eqref{NLS} is locally well-posed in $H^s(\mathbb{R})$ when $2<s<\min\{3p,p+\frac{5}{2}\}$.
\end{theorem} 
\begin{proof}
We will only show how to improve the a priori $X^s$ bound for $D_t^{\frac{s}{2}}u$ from \Cref{apriorischrodinger}. The full well-posedness result will then follow verbatim by adapting the argument from the previous subsection. By \Cref{M1 OOTP} we may assume without loss of generality that $p\leq \frac{3}{2}$ and $s<4$. In this case, $\left(\frac{4}{p-1},\frac{2}{2-p}\right)$ is an admissible Strichartz pair with dual exponent $\left(\frac{4}{5-p},\frac{2}{p}\right)$. Since $\frac{4}{5-p}\leq \frac{2}{p}$ when $1\leq p\leq \frac{5}{3}$, Minkowski's inequality and \Cref{Nonlinear estimate} give the bound
\begin{equation*}
\begin{split}
\big\|D_t^\frac{s}{2}\left(|u|^{p-1}u\right)\|_{L_t^{\frac{4}{5-p}+}L_x^\frac{2}{p}}&\lesssim \big\|D_t^\frac{s}{2}\left(|u|^{p-1}u\right)\|_{L_x^{\frac{2}{p}}L_t^\frac{2}{p}}\lesssim \big\|\|u\|_{H_t^{\frac{p}{2}+\epsilon}}^{p-1}\|u\|_{W_t^{\frac{s}{2},\frac{2}{p}}} \|_{L_x^{\frac{2}{p}}}
\\
&\lesssim \|u\|^{p-1}_{H_t^{\frac{p}{2}+\epsilon}L^2_x}\|u\|_{H_t^\frac{s}{2}L_x^2}\lesssim \|u\|_{X^{s_0}}^{p-1}\|u\|_{X^s},
\end{split}
\end{equation*}
when $\frac{p}{2}<\frac{s}{2}<p+\frac{1}{\frac{2}{p}}=\frac{3}{2}p$ or $s<3p.$ With this estimate and the strategy of the previous two subsections, well-posedness readily follows.
\end{proof}
\begin{remark}
In this particular result, the restriction $s<3p$ comes from the limited decay of the nonlinearity when $p$ is small. Whether this is a serious obstruction that can be used to obtain examples of ill-posedness at this new endpoint we leave as an interesting open question.
\end{remark}
\begin{remark}
When applying \Cref{Nonlinear estimate} in many of the estimates in \Cref{apriorischrodinger}, we had to use Minkowski's inequality to switch the order of the $L^p$ norms in space and time. One might ask if a vector-valued version of \Cref{Nonlinear estimate} of the following form is true whenever $q'<r'$ and $(q,r)$ is Strichartz admissible:
\begin{equation*}
\|D_t^{\frac{s}{2}}(|u|^{p-1}u)\|_{L_t^{q'}L_x^{r'}}\lesssim \|u\|^{p-1}_{X^{s_0}}\|u\|_{W_t^{\frac{s}{2},q}L_x^r},\hspace{5mm} \frac{s}{2}<p+\frac{1}{q},
\end{equation*}
for some reasonable choice of $X^{s_0}$. The motivation for this would be that, in principle, one could use the dual Strichartz norm $L_t^1L_x^2$ to substantially improve the restriction $s<3p$ above (or $s<2p+1$ in higher dimensions). Unfortunately, such an estimate is not true for a general Schwartz function $u$, as the one-dimensional example $u(t,x)=\chi(x)\chi(t)(x-t)$ shows (here $\chi$ is some smooth cutoff function localized to the origin).
\end{remark}
\section{Non-existence results}\label{Ill section}
In this section, we prove our main ill-posedness theorems. Our proof of high regularity non-existence for \eqref{NLS} and \eqref{NLH} is somewhat inspired by the ideas in \cite[Theorem 2.1]{cazenave2017non}. However, it is executed more efficiently and applies to the full range of Sobolev spaces. This, in particular, answers an open question from \cite{cazenave2017non}. Noteworthy features in our argument include: 1) It is no longer restricted to the regime $1<p<2$; 2) It is simpler than the argument in \cite{cazenave2017non};  3) It includes an improved dimension reduction procedure to avoid non-optimal (dimension-dependent) restrictions such as $s>p+2+\frac{d}{q}$; 4) It works at optimal regularity levels and includes the sharp endpoint $s=p+2+\frac{1}{q}$. 
\subsection{Non-existence for the nonlinear Schr\"odinger and heat equations} Our aim in this subsection is to prove the following theorems.
\begin{theorem}\label{ill for S+H}
   Let $p>1$ with $p-1\not\in 2\mathbb{N}$ and $1<q<\infty$. Then there is an initial datum $u_0\in C_c^\infty(\mathbb{R}^d)$ of arbitrarily small norm such that for any $T>0$ there is no corresponding solution $u\in C([0,T];W^{s,q}(\mathbb{R}^d))$ to \eqref{NLH} for any $s\geq \max\{ p+2+\frac{1}{q},s_c\}$.
\end{theorem}
\begin{theorem}\label{ill for S+H2} Let $p>1$ with $p-1\notin 2\mathbb{N}$. Then there is an initial datum $u_0\in C_c^\infty(\mathbb{R}^d)$ of arbitrarily small norm such that for any $T>0$ there is no  corresponding solution $u\in C([0,T];H^{s}(\mathbb{R}^d))$ to \eqref{NLS} for any $s\geq \max\{p+\frac{5}{2},s_c\}$. 
\end{theorem}
To prove Theorems~\ref{ill for S+H} and \ref{ill for S+H2}, we begin by recalling a characterization of $\| D_x^sf\|_{L^q_x}$ due to Strichartz (see \cite{MR0215084} or the appendix of \cite{MR2318286}). 
\begin{proposition}\label{Strich character}
    For $0<s<1$ and $1<q<\infty$ we have 
\begin{equation*}
    \| D_x^sf\|_{L^q_x}\approx \|\mathcal{D}_s(f)\|_{L^q_x},
\end{equation*}
where
\begin{equation*}
    \mathcal{D}_s(f)(x)=\left(\int_0^\infty\left|\int_{|y|<1}\frac{|f(x+ry)-f(x)|}{r^s}dy\right|^2r^{-1}dr\right)^\frac{1}{2}.
\end{equation*}
\end{proposition}
 The heart of the proofs of Theorems~\ref{ill for S+H} and \ref{ill for S+H2} is the following one-dimensional lemma.
\begin{lemma}\label{main lem ill}
 Fix a standard cutoff $\chi$ which is equal to $1$ on $[-1,1]$,  $n\in \mathbb{N}$ and $1<q<\infty$. Suppose that  $p\in (n,n+1)$, $\lambda\in \mathbb{C}\setminus\{0\}$, $\delta>0$ and $T>0$. Let $h\in L^1([0,T];W_x^{p+\frac{1}{q},\,q}(\mathbb{R}))$ be such that $h(t,0)=0$ for almost every $0\leq t\leq T$. Then there is no solution $w\in C([0,T];W_x^{p+\frac{1}{q},q}(\mathbb{R}))$ to the integral equation 
\begin{equation}\label{intequation}
w(t,x)=w_0(x)+\int_{0}^{t}\lambda|w(s,x)|^{p-1} w(s,x)+h(s,x)ds,
\end{equation}
with data $w_0=\delta\chi x$. More specifically, for all $0<t\ll T$ we have 
$$\|w(t)\|_{W_x^{p+\frac{1}{q},\, q}(\mathbb{R})}=\infty.$$
Moreover, a similar result holds when $p$ is an even integer.
\end{lemma}

\begin{proof}
Let us first assume that $1<q<\infty$ and  $p\in (n,n+1-\frac{1}{q})$. Suppose for the sake of contradiction that $w\in C([0,T];W_x^{p+\frac{1}{q},q})$ solves \eqref{intequation}. As in  \Cref{Nonlinear est section}, by repeated applications of the chain rule, we have
\begin{equation*}
\partial_x^n\left(|w|^{p-1}w\right)=C_1 |w|^{p+1-2n}w_x\Re(\overline{w}w_x)^{n-1}+C_2 |w|^{p-1-2n}w\Re(\overline{w}w_x)^n+F,
\end{equation*}
where $C_1$, $C_2$ are explicit constants satisfying $C_1+C_2\neq 0$  and, by interpolation, $F\in L_T^{1}W^{p+\frac{1}{q}-n,q}$ is a term involving a better distribution of derivatives. For simplicity of notation, let us define $v=\partial_xw(t,x)$ and $f=\partial_x^nh(t,x).$ By differentiating the equation $n$ times in space, we have
\begin{equation}\label{n-times diff ill}
\begin{split}
\partial_x^nw(t)=\partial_x^nw_0+&\int_0^t\left(C_1 |w|^{p+1-2n}v\Re(\overline{w}v)^{n-1}+C_2 |w|^{p-1-2n}w\Re(\overline{w}v)^n\right)ds
\\
+&\int_0^t\left( F(s)+f(s)\right)ds.
\end{split}
\end{equation}
Define 
\begin{equation*}
I(t,x):=\int_0^t\left(C_1 |w|^{p+1-2n}v\Re(\overline{w}v)^{n-1}+C_2 |w|^{p-1-2n}w\Re(\overline{w}v)^n\right)ds.
\end{equation*}
We wish to show that $I(t,\cdot)\not \in W^{p+\frac{1}{q}-n,q}(\mathbb{R})$ for every $0<t\ll T$, which will give us the desired conclusion as the other terms on the right-hand side of \eqref{n-times diff ill} are in $L^\infty_TW^{p+\frac{1}{q}-n,q}$. To prove this, we will combine \Cref{Strich character} with the following H\"older type estimate.
\begin{lemma}\label{lemma high reg ill}
Let $0<x\ll 1$ and let $\frac{1}{4}x\leq h\leq 2x$. Then for small enough $t>0$, we have
\begin{equation*}
|I(t,x+h)-I(t,x)|\approx t|x|^{p-n}.
\end{equation*}
\end{lemma}

To conclude \Cref{main lem ill} from  \Cref{lemma high reg ill}, we bound
\begin{equation*}
    \begin{split}
        \int_0^\infty\left|\int_{|y|<1}\frac{|I(t,x+ry)-I(t,x)|}{r^{p-n+\frac{1}{q}}}dy\right|^2r^{-1}dr&\geq \int_{\frac{1}{2}|x|}^{2|x|}\left|\int_{\frac{1}{2}<y<1}\frac{|I(t,x+ry)-I(t,x)|}{r^{p-n+\frac{1}{q}}}dy\right|^2r^{-1}dr
        \\
        &\gtrsim t^2|x|^{-\frac{2}{q}}.
    \end{split}
\end{equation*}
 Hence, the fact that $I(t,\cdot)\not \in W^{p-n+\frac{1}{q},q}(\mathbb{R})$ for $0<t\ll 1$ follows immediately from \Cref{Strich character}.
\end{proof}
\begin{proof}[Proof of \Cref{lemma high reg ill}]
We consider first the term $|w|^{p+1-2n}v\Re(\overline{w}v)^{n-1}$. Note that we may expand $\Re(\overline{w}v)^{n-1}$ as a linear combination of monomials of the form $w^{k_1}\overline{w}^{l_1}v^{k_2}\overline{v}^{l_2}$ where $k_1+l_1=k_2+l_2=n-1.$ For this reason, we begin by analyzing terms of the form $|w|^{p+1-2n}w^{k_1}\overline{w}^{l_1}$.
\medskip

Note that we have the embedding $W^{p+\frac{1}{q},q}(\mathbb{R})\subset C^{1,\alpha}(\mathbb{R})$ for some $\alpha>0$. Therefore, we  observe that for $y\approx x$ and every $0<\alpha\leq 1$, there holds
 \begin{equation}\label{basic expansion}
|(|w|^{p+1-2n}w^{k_1}\overline{w}^{l_1})(s,y)-(|w_x|^{p+1-2n}w_x^{k_1}\overline{w}_x^{l_1})(s,0)|y|^{p+1-2n}y^{k_1+l_1}|\lesssim_{\|w\|_{C^{1,\alpha}}}|x|^{(1+\alpha)(p-n)}.
\end{equation}
Indeed, this follows from a Taylor expansion, the fact that $w(s,0)=0$, the assumption that $p-n\in (0,1)$ and (a slight variation of) \Cref{PHB}.
\medskip

In the sequel, we let $R$ denote a remainder term with $R\ll_{\|w\|_{C^{1,\alpha}}}|x|^{p-n}$ for some $0<\alpha\leq 1$. From \eqref{basic expansion}, we obtain
\begin{equation*}\label{mainapprox}
\begin{split}
(|w|^{p+1-2n}w^{k_1}\overline{w}^{l_1})(x+h)&-(|w|^{p+1-2n}w^{k_1}\overline{w}^{l_1})(x)
\\
&= (|w_x|^{p+1-2n}w_x^{k_1}\overline{w}_x^{l_1})(s,0)(|x+h|^{p+1-2n}(x+h)^{n-1}-|x|^{p+1-2n}x^{n-1})+R 
\\
&\approx |x|^{p-n}+R
\end{split}
\end{equation*}
where in the last line we used the assumptions $x\geq 0$, $h\geq\frac{1}{4}x$,  $|x|^{(1+\alpha)(p-n)}\ll_{\alpha} |x|^{p-n}$ as well as the fact that for small enough $s$, we have $|w_x(s,0)|\approx 1$ (which follows from the equation satisfied by $w$ and the hypothesis for $h$). Using the above, the equation for $w$ and the hypothesis for $h$, we have $v(s,x)=1+o(1)$ where $o(1)$ denotes an error term that goes to zero as $s\to 0$. We conclude that for $t$ sufficiently small,
\begin{equation*}
\begin{split}
(|w|^{p+1-2n}w^{k_1}\overline{w}^{l_1}v^{k_2+1}\overline{v}^{l_2})(t,x+h)&-(|w|^{p+1-2n}w^{k_1}\overline{w}^{l_1}v^{k_2+1}\overline{v}^{l_2})(t,x)\approx |x|^{p-n}+R.
\end{split}
\end{equation*}
Thus, the same approximation holds for the term $|w|^{p+1-2n}v\Re(\overline{w}v)^{n-1}$. By a similar argument, we may handle the term $|w|^{p-1-2n}w\Re(\overline{w}v)^n$. It then follows from the definition of $R$ that we have
\begin{equation*}
|I(t,x+h)-I(t,x)|\approx t|x|^{p-1}    
\end{equation*}
as desired. This proves \Cref{main lem ill} when  $1<q<\infty$ and $p\in (n,n+1-\frac{1}{q})$. Suppose now that we have fixed $1<q<\infty$ and a general $p\in (n,n+1)$; we wish to show that the conclusion of \Cref{main lem ill} still holds, i.e., $\|w(t)\|_{W_x^{p+\frac{1}{q},\, q}(\mathbb{R})}=\infty$. To do this, we pick $q^*$ large enough so that $p\in (n,n+1-\frac{1}{q^*})$.  By Sobolev embeddings, we may bound the norm $\|\cdot\|_{W_x^{p+\frac{1}{q},\, q}(\mathbb{R})}$ from below by $\|\cdot\|_{W_x^{p+\frac{1}{q^*},\, q^*}(\mathbb{R})}$. In particular, from the assumption that $h\in L^1([0,T];W^{p+\frac{1}{q},\,q}(\mathbb{R}))$, we also have $h\in L^1([0,T];W^{p+\frac{1}{q^*},\,q^*}(\mathbb{R}))$. Therefore, we may apply \Cref{main lem ill}  with the pair $(p,q^*)$ to conclude that $\|w(t)\|_{W_x^{p+\frac{1}{q^*},\, q^*}(\mathbb{R})}=\infty$. This immediately yields that $\|w(t)\|_{W_x^{p+\frac{1}{q},\, q}(\mathbb{R})}=\infty$,  as desired.
\medskip

It now remains to deal with the endpoints $p\in 2\mathbb{N}$. The analysis proceeds similarly to before, but with the sgn function arising at leading order. We leave the straightforward modifications to the reader.
\end{proof}

Note that \Cref{main lem ill} is entirely one-dimensional. To apply it in higher dimensions, we recall a ``Fubini" type theorem for $W^{s,q}(\mathbb{R}^d)$ due to Strichartz \cite{MR0215084}.
\begin{theorem}\label{Strich fub}
    Let $f$ be a function defined almost everywhere on $\mathbb{R}^d$ and let $f_k(x_1,\dots,\hat{x}_k,\dots,x_d)$ be the $W^{s,q}(\mathbb{R})$ norm of $f$ restricted to the line through $(x_1,\dots,x_d)$ parallel to the $x_k$ axis, if it is defined, and $\infty$ otherwise. Then $f\in W^{s,q}(\mathbb{R}^d)$ if and only if $f_k\in L^q(\mathbb{R}^{d-1})$ for $k=1,\dots, d.$ In this case, we have
    \begin{equation*}
        \|f\|_{W^{s,q}(\mathbb{R}^d)}\approx \sum_{k=1}^d \|f_k\|_{L^q(\mathbb{R}^{d-1})}.
    \end{equation*}
\end{theorem}
Now, we are ready to prove Theorems~\ref{ill for S+H} and \ref{ill for S+H2}.
\begin{proof}[Proof of Theorems~\ref{ill for S+H} and \ref{ill for S+H2}.]
We will show the details for the proof of \Cref{ill for S+H} as \Cref{ill for S+H2} follows from identical reasoning. 
\medskip

    Our first task will be to reduce the question of non-existence to an application of the one-dimensional \Cref{main lem ill}. To this end, for every $x\in \mathbb{R}^d$ (with $d\geq 2$), we write $x=(x_1,x')$ and consider an initial datum $u_0\in C_c^\infty(\mathbb{R}^d)$ of the form $u_0(x_1,x')=\delta \chi_1(x_1)\chi_2(x') x_1$ where $\chi_1$ and $\chi_2$ are radial $C_c^\infty$ cutoffs (separately in $x_1$ and $x'$) which are equal to one on the unit scale. Suppose for the sake of contradiction that there exists a solution $u\in  C([0,T]; W^{s,q}(\mathbb{R}^d))$ to \eqref{NLH} with $s\geq\max\{p+2+\frac{1}{q},s_c\}$ and initial data $u_0$. From the hypothesis on $s$, standard difference type estimates ensure that $u$ is unique and that $u$ (and therefore, $\Delta u$) is odd in $x_1$ (in the sense of distributions). 
    \medskip

Our main task at this point is to show that for almost every $x'\in B_1'\coloneq \{x'\in\mathbb{R}^{d-1}: |x'|\leq 1\}$ (i.e.~the region where $\chi_2=1$), the restriction  $u(\cdot,\cdot,x')\in C([0,T]; W^{p+\frac{1}{q},q}(\mathbb{R}))$ solves the integral equation \eqref{intequation} and that $h:=\Delta u(\cdot,\cdot,x')\in L^1([0,T];W^{p+\frac{1}{q},q}(\mathbb{R}))$ with $h(t,0,x')=0$ for almost every $t$. By \Cref{main lem ill}, this will immediately give a contradiction. 
    \medskip
    
  To proceed, we begin by first observing directly from the equation \eqref{NLH} and \Cref{heatestmodified} that $u\in C^{1}\left([0,T];W^{p+\frac{1}{q}-\epsilon,q}(\mathbb{R}^d)\right)$. Interpolating with $u\in C([0,T];W^{p+2+\frac{1}{q},q}(\mathbb{R}^d))$ as on \cite[p.~43]{MR3753604}, we have $u\in C^{1-\epsilon}([0,T];W^{p+\frac{1}{q},q}(\mathbb{R}^d))$ for some $\epsilon>0$ small. In particular,  $u\in W^{1-\epsilon,q}\left([0,T];W^{p+\frac{1}{q},q}(\mathbb{R}^d)\right)$. By the Fubini property, it follows that
    \begin{equation*}
        u(x')\in W^{1-\epsilon,q}\left([0,T];W^{p+\frac{1}{q},q}(\mathbb{R})\right)
    \end{equation*}
     for almost every $x'\in \mathbb{R}^{d-1}$. By Sobolev embedding in time, it follows that for such $x'$ we have
     \begin{equation*}
        u(x')\in C\left([0,T];W^{p+\frac{1}{q},q}(\mathbb{R})\right).
    \end{equation*}
 Similarly, since $\Delta u\in C([0,T];W^{p+\frac{1}{q},q}(\mathbb{R}^d))$, we have  $\Delta u\in L^q([0,T];W^{p+\frac{1}{q},q}(\mathbb{R}^d))$, so Fubini can again be applied (together with H\"older in time) to deduce the desired regularity for $h$ for almost every $x'\in\mathbb{R}^{d-1}$.
 \medskip
 
We now define the function 
\begin{equation*}
f(t,x):=u(t,x)-u_0(x)+i\int_{0}^{t}|u(s,x)|^{p-1}u(s,x)ds+i\int_{0}^{t}(\Delta u)(s,x)ds.
\end{equation*}
We know that $f\in C([0,T]; L^q_x(\mathbb{R}^d))$ by \Cref{farestimateH}  and since $u$ is a solution to \eqref{NLS} we have that $f$ is equal to zero almost everywhere. In particular, by Fubini, for almost every $x'\in \mathbb{R}^{d-1}$ we have 
\begin{equation*}
        \|f(x')\|_{L^{\infty}_TL^q_x(\mathbb{R})}=0.
    \end{equation*}
So, for almost every $x'\in \mathbb{R}^{d-1},$ we have 
\begin{equation*}
       f(t,x_1,x')=0
    \end{equation*}
    for almost every $(t,x_1)\in [0,T]\times \mathbb{R}.$ On the other hand, since (from above) $u(x')\in C\left([0,T];W^{p+\frac{1}{q},q}\right)$ and $h\in L^1([0,T];W_x^{p+\frac{1}{q},q})$, we may deduce that for almost every $x'$, there holds $f(t,x_1,x')=0$ for all $t,x_1$. We can then apply \Cref{main lem ill} to obtain a contradiction and conclude the proof.
\end{proof}



\bibliographystyle{plain}
\bibliography{refs.bib}

\end{document}